\documentclass[12pt,letterpaper]{amsart}
\usepackage{tikz-cd}
\usepackage{verbatim}
\usepackage{xcolor}
\usepackage{url}
\usepackage{hyperref}
\numberwithin{equation}{section}
\usepackage{amssymb}
\usepackage{enumitem}
\usepackage[margin=3.3cm]{geometry}
\usepackage{mathtools}
\usepackage{setspace}
\usetikzlibrary{arrows.meta,calc}
\usepackage{mathtools}
\usepackage{mathabx}
\usetikzlibrary{calc,intersections}
\let\cal\mathcal

\def\Bscr{{\cal B}}

\def\Dscr{{\cal D}}
\def\Escr{{\cal E}}
\def\Fscr{{\cal F}}
\def\Gscr{{\cal G}}
\def\Hscr{{\cal H}}

\def\Nscr{{\cal N}}
\def\Oscr{{\cal O}}
\def\Pscr{{\cal P}}
\def\Qscr{{\cal Q}}

\def\Sscr{{\cal S}}
\def\Tscr{{\cal T}}
\def\Uscr{{\cal U}}
\def\Vscr{{\cal V}}

\let\blb\mathbb
\def\CC{{\blb C}}

\def\FF{{\blb F}} 
\def\QQ{{\blb Q}}

\def \PP{{\blb P}}
\def \AA{{\blb A}}
\def \ZZ{{\blb Z}}

\def \NN{{\blb N}}
\def \RR{{\blb R}}
\def \HH{{\blb H}}
\def \LL{{\blb L}}

\let\oldmarginpar\marginpar
\def\marginpar#1{\setlength{\marginparwidth}{0.15\textwidth}\oldmarginpar{\tiny \baselineskip 0pt\lineskip 0pt\lineskiplimit 0pt  \raggedright #1}}

\makeatletter
\@namedef{subjclassname@2020}{\textup{2020} Mathematics Subject Classification}
\makeatother

\def\perf{\operatorname{perf}}
\def\id{\text{id}}
\def\Id{\operatorname{id}}

\def\mod{\operatorname{mod}}
\def\Gr{\operatorname{Gr}}

\def\Gr{\operatorname{Gr}}

\def\coh{\mathop{\text{\upshape{coh}}}}

\def\GL{\operatorname {GL}}
\def\SL{\operatorname {SL}}
\def\SO{\operatorname {SO}}
\def\PGL{\operatorname {PGL}}
\def\PSL{\operatorname {PSL}}

\def\Ext{\operatorname {Ext}}
\def\Hom{\operatorname {Hom}}

\def\End{\operatorname {End}}
\def\RHom{\operatorname {RHom}}

\def\H{\operatorname {H}}
\def\End{\operatorname {End}}

\def\id{{\operatorname {id}}}

\def\rk{\operatorname {rk}}

\def\Pic{\operatorname {Pic}}

\def\r{\rightarrow}
\def\l{\leftarrow}

\DeclareMathOperator{\Aut}{Aut}

\DeclareMathOperator{\NS}{NS}

\DeclareMathOperator{\REnd}{REnd}
\DeclareMathOperator{\cone}{cone}

\newtheorem{lemma}{Lemma}[section]
\newtheorem{proposition}[lemma]{Proposition}
\newtheorem{theorem}[lemma]{Theorem}
\newtheorem{corollary}[lemma]{Corollary}

\newtheorem{observation}[lemma]{Observation}

\theoremstyle{definition}

\newtheorem{definition}[lemma]{Definition}

{

}

\theoremstyle{remark}

\newtheorem{remark}[lemma]{Remark}

\newdimen\uboxsep \uboxsep=1ex
\def\uboxn#1{\vtop to 0pt{\hrule height 0pt depth 0pt\vskip\uboxsep
\hbox to 0pt{\hss #1\hss}\vss}}

\def\uboxs#1{\vbox to 0pt{\vss\hbox to 0pt{\hss #1\hss}
\vskip\uboxsep\hrule height 0pt depth 0pt}}

\def\Stab{\operatorname{Stab}}
\def\ar{\operatorname{ar}}  
\def\Re{\operatorname{Re}}
\def\Im{\operatorname{Im}}
\def\Isom{\operatorname{Isom}}
\def\Sph{\operatorname{Sph}}

\title[Groups generated by spherical twists on K3 surfaces]{Groups generated by spherical twists on K3 surfaces and full exceptional collections on Fano threefolds}

\author{Anya Nordskova}
\author{Michel Van den Bergh}
\thanks{The second author is a senior researcher at the Research
  Foundation Flanders (FWO). Both authors are also supported by  the ERC grant SCHEMES}
\address[Anya Nordskova and Michel Van den Bergh]{Vakgroep Wiskunde, Universiteit Hasselt, Universitaire Campus \\
  B-3590 Diepenbeek}
\email{anya.nordskova@uhasselt.be}
\email{michel.vandenbergh@uhasselt.be}
\address[Michel Van den Bergh]{Vakgroep Wiskunde en Data Science, Vrije Universiteit Brussel, Pleinlaan 2, 1050 Brussel} 
\email{michel.van.den.bergh@vub.be}
\begin{document}
\begin{abstract} Let $Y$ be a smooth K3 surface of Picard rank $1$.  We prove that the subgroup $\Gscr$ of $\Aut \Dscr^b(Y)$ generated by spherical twists with respect to all spherical objects is free.  Moreover,  we provide a precise recipe to find free generators of $\Gscr$ and determine the cases when $\Gscr$ is finitely generated, depending on the degree of $Y$.  This description in particular yields a precise classification of spherical objects in $\Dscr^b(Y)$. 

We apply these results to verify the first three-dimensional case of a conjecture due to Bondal and Polishchuck, namely,  we establish the transitivity of the braid group action on full exceptional collections for Fano threefolds of Picard rank 1. 
\end{abstract}
\maketitle
 \tableofcontents 
\section{Introduction} 
Let $Y$ be a smooth K3 surface of Picard rank $1$.  One of the main results of the groundbreaking work
of Bridgeland \cite{BridgelandK3} and Bayer-Bridgeland \cite{BayerBridgeland} (also building on some preceding work \cite{HMS,Kawatani,Orlov97}) is the existence of a short exact sequence of groups (see Theorem \ref{thm:mainBB})
 \begin{equation}
\label{eq:mainBBintro}
 1 \to  \Aut^0 \Dscr^b(Y) \to \Aut \Dscr^b(Y) \to \Aut^+ \H^*(Y,\ZZ)  \to 1,
\end{equation}
where $\Aut^+ \H^*(Y,\ZZ)$ is the group of Hodge isometries of $\H^*(Y,\ZZ)$ which preserve the orientation of the positive $4$-planes and $\Aut^0\Dscr(Y)$
is the product of the group generated by the double shift $[2]$ and the group
freely generated by the squares of spherical twists $T_{S}^2$ with $S$ running over all spherical vector bundles on $Y$.  The sequence \eqref{eq:mainBBintro} is obtained as a corollary of a highly non-trivial theorem on the contractibility of $\Stab^\dagger(Y)$,  a distinguished component of the moduli space of stability conditions,  which carries a natural $\Dscr^b(Y)$-action.  
\medskip

As our first main result, we refine Bayer-Bridgeland's sequence obtaining a more precise description of $\Aut \Dscr^b(Y)$ as well as a more precise classification of spherical objects in $\Dscr^b(Y)$.  Let us now elaborate on this. 
\medskip

It is well known that K3 surfaces of Picard rank $1$ have a single discrete invariant, the minimal degree $d=2\delta$ of an ample divisor. 
In this introduction
we will assume $d\ge 4$ for simplicity and refer to the text for the case $d=2$.   First note that the algebraic cycles 
provide us with a subgroup of rank $3$ in $\H^\ast(Y,\ZZ)$
\[
\Nscr(Y):=\H^0(Y,\ZZ)\oplus \Pic(Y)\oplus \H^4(Y,\ZZ)\cong \ZZ^3
\]
which is equipped with the Mukai quadratic form $q(r,l,s)=\delta l^2-rs$. We define $\Aut^+ \Nscr(Y)$ as the isometries
of $\Nscr(Y)$ which preserve the orientation of the positive 2-planes. It turns out that there is a well-defined restriction map
\begin{equation}
\label{eq:resintro}
\Aut^+ \H^\ast(Y,\ZZ)\r \Aut^+ \Nscr(Y)
\end{equation}
which is injective (assuming $d\ge 4$, see Lemma \ref{lem:HvsN}).  The following precise descriptions of the groups $\Aut^+ \Nscr(Y)$ and $\Aut^+ \H^*(Y, \ZZ)$ are well known to experts and go back to Dolgachev \cite{Dolgachev} and Kawatani \cite{Kawatani14}.  For the benefit of the reader, we give independent proofs of these results in the Appendix (see \S \ref{sec:autN}, \S \ref{sec:autH}).  Recall that for $\delta \geq 1$ the  \emph{Fricke group} $F(\delta)$ is obtained by adjoining the \emph{Fricke involution} $f_\delta=\left(\begin{smallmatrix} 0 &-1\\ \delta &0\end{smallmatrix}\right)$ to the standard \emph{congruence subgroup}
\[
\Gamma_0(\delta):=\left\{{\footnotesize\begin{pmatrix} a&b\\ c&d\end{pmatrix}} \in \PSL_2(\ZZ)\mid c\cong 0\mod \delta\right\}\subset \PSL_2(\RR),
\]
while the \emph{Atkin-Lehner} group, denoted $\Gamma^+_0(\delta)$,  is obtained from $\Gamma_0(\delta)$ by adjoining the ``Atkin-Lehner elements'',  of which $f_\delta$ is a special case (see \S\ref{sec:fuchsian}).  In particular, $F(\delta) \subset \Gamma^+_0(\delta)$.
\begin{proposition}[{\cite[Theorem 7.1,  Remark 7.2.2]{Dolgachev}, \cite[Theorem 3.3]{Kawatani14}},  see also Remark \ref{rem:pm1}, Proposition \ref{prop:imHN}]  Assume $d\ge 4$. We have the isomorphisms
\begin{equation}
\label{eq:FGammaintro}
\begin{aligned}
\Aut^+ \H^*(Y,\ZZ)/\{\pm 1\}& \cong F(\delta)\\
\Aut^+ \Nscr(Y)/\{\pm 1\}&\cong \Gamma^+_0(\delta)
\end{aligned} 
\end{equation}
\end{proposition}
One has $F(\delta)/\Gamma_0(\delta)\cong\ZZ/2\ZZ$ and $\Gamma^+_0(\delta)/\Gamma_0(\delta)=(\ZZ/2\ZZ)^c$ where $c$ is the number of distinct prime factors in $\delta$.
Hence \eqref{eq:resintro} is in general not surjective.  The first counterexample is given by $\delta=6$, i.e.\ $d=12$.

The groups that appear in \eqref{eq:FGammaintro} 
are examples of so-called Fuchsian groups,  that is,  they are discrete subgroups of the group of orientation preserving isometries of the complex upper half-plane $\HH$ on which they act by
\emph{M\"obius transformations} ${z\mapsto (az+b)/(cz+d)}$.  It was observed already by Kawatani \cite{Kawatani} that $\Stab^\dagger(Y)$, a distinguished component of the space of stability
conditions on $\Dscr^b(Y)$,  is closely related to the hyperbolic plane $\HH$.  On the other hand,  as mentioned above,  $\Stab^\dagger(Y)$ is equipped with a natural action of $\Aut \Dscr^b (Y)$ and the existence of the exact sequence \eqref{eq:mainBBintro} ultimately boils down to the fact that $\Stab^\dagger(Y)$ is simply connected,  which was originally conjectured in \cite{BridgelandK3} and proved in \cite{BayerBridgeland}.  Leveraging this together with \eqref{eq:FGammaintro} we obtain the following,  which was already observed by Fan and Lai in \cite{FanLai} (see (1.2) on p.4). 
\begin{observation}[Lemma \ref{lem:tautological2}]  Let $Y$ be a K3 surface of Picard rank $1$ and degree $d = 2 \delta \ge 4$. Let $\Delta^+(Y)=\{v = (r,l,s)\in \Nscr(Y)\mid q(v)-1, r>0\}$ and let $\HH^0$ be the complement  in $\HH$ of the points 
$(l\delta +i\sqrt{\delta})/(r\delta)$ for $(r,l,s)\in \Delta^+(Y)$. Then 
\[
\Aut \Dscr^b(Y)/[1]\cong \pi_1([\HH^0/F(\delta)])
\]
where $[\HH^0/F(\delta)]$ is the orbifold quotient of $\HH^0$ by $F(\delta)$.
\end{observation}
$[\HH^0/F(\delta)]$ is an orbifold Riemann surface, so in particular its fundamental group is well understood (see \S\ref{subsec:prelimrs}). This yields a good understanding
of the group $\Aut \Dscr^b(Y)/[1]$. A similar approach can be applied to subgroups of $\Aut \Dscr^b(Y)/[1]$ that contain $\Aut^0 \Dscr^b(Y)/[2]$ such as  (the image of) the group 
\[
\Gscr = \langle T_S \ | \ S  \in \Dscr^b(Y) \text{ spherical} \rangle \subset \Aut \Dscr^b(Y)
\]
 generated 
by all spherical twists.  It turns out that the image $G$ of $\Gscr$ in $\Aut^+ \H^*(Y, \ZZ) / \{ \pm 1 \} = F(\delta)$ acts freely on $\HH^0$ (see Lemma \ref{lem:fuchs}\eqref{it:fuchs:it1}),  so 
that $[\HH^0/G]=\HH^0/G$ is a (non-compact) Riemann surface. 
Hence $\Gscr=\pi_1([\HH^0/G])$ is actually a free group.  Elaborating on this leads to our first main result.
\begin{theorem}[Theorem \ref{thm:sphtwistfree}]\label{thm:introsphtwistfree} Let $Y$ be a K3 surface of Picard rank $1$. 
Then the subgroup $\Gscr$ of $\Aut \Dscr^b(Y)$ generated by spherical twists is freely generated by a collection of spherical 
twists associated to spherical vector bundles. These vector bundles are orbit representatives for the action of $\Gscr$ on the set of all spherical objects in $\Dscr^b(Y)$. 
\end{theorem}
The actual result in the text is more detailed and in fact gives a precise recipe to find the generators of $\Gscr$ in terms of a suitable fundamental domain for the action of $G$ on $\HH$.  We apply this recipe explicitly to some low-degree K3 surfaces in \S\ref{sec:examples},  in particular for those that appear as anticanonical sections of Fano threefolds of Picard rank $1$ possessing a full exceptional collection.  These examples are of particular interest to us and we will discuss them in detail below.  In these cases we show that the group $\Gscr \cong F_4$  is freely generated by four elements, which may suggest the optimistic conjecture that $\Gscr$ is always finitely generated.  This is far from the case,  however,  and in fact we derive a complete (finite) list of degrees $d$ for which it is true: 

\begin{theorem}[Theorem \ref{thm:fingendeg}] Let $Y$ be a K3 surface of Picard rank $1$ and degree $d = 2\delta$.  Let
  $\Gscr$ be the group generated by all spherical
  twists. Then $\Gscr$ is finitely generated if and only if $\delta \in \{1,2,3,5,7,11,17,19,23,29,31,41,47,59,71\}$.  \end{theorem}
  
In \S\ref{sec:examples} we consider in particular the smallest infinitely generated example,  namely,  $d=8$.  In that case it turns out that $\Gscr$ is freely generated by $T_{\Oscr(nH)}$, $n \in \ZZ$,  where $H$ is a generator of $\Pic(Y)$. 

\subsection{Application to Bondal-Polishchuk's conjecture}

Bondal and Polishchuk conjectured \cite[Conjecture 2.2]{BondalPolishchuk} that the action of 
$B_n \ltimes \ZZ^n$,  where $B_n$ is the $n$-strand braid group, by mutations and shifts is transitive on full exceptional collections of length $n$ for any triangulated category.  This conjecture has been recently disproved by Chang,  Haiden and Schroll \cite{CHS} who found a series of partially wrapped Fukaya categories producing counterexamples.
Restricting to bounded derived categories $\Dscr^b(X)$ of smooth projective varieties $X$, the conjecture has been verified in a small number of special cases,  namely,  by Rudakov \cite{Rudakov1}, \cite{Rudakov2} for $\PP^2$ and $\PP^1 \times \PP^1$,  by Kuleshov and Orlov \cite{KuleshovOrlov} for all del Pezzo surfaces and by Ishii,  Okawa and Uehara \cite{IshiiOkawaUehara} for the second Hirzebruch surface $\FF_2$.  However,  in general the question is still widely open.  

\medskip

As an application of our result on the groups generated by spherical twists on K3 surfaces,  we settle the first three-dimensional case of Bondal-Polishchuk's conjecture.  Namely,  consider threefolds with a full exceptional collection of four vector bundles.  It turns out that such a threefold $X$ is automatically Fano and has one of the four types ${X = \PP^3, Q_3, V_5}$ and $V_{22}$,  where only $V_{22}$ has non-trivial moduli (see Lemma \ref{lm:4Fanos} and \S\ref{sec:examples} for more details).
Since $\rk K_0(X)=4=\dim X+1$,  it follows from \cite{MR992977}
(see Proposition \ref{pr:braidaction})
that the full $B_n \ltimes \ZZ^n$-orbit of the given full exceptional collection consists of full exceptional collections of vector bundles,  up to shifts.  A converse to this result was obtained by Polishchuk \cite{Polishchuk}: he showed that the action of $B_4 \ltimes \ZZ^4$ on full exceptional collections of shifted vector bundles is transitive.  However,  even for $\PP^3$ it was unknown if these constitute all full exceptional collections.  The following result settles this:

\begin{theorem}[Theorem \ref{thm:main}, Corollary \ref{cor:allsheaves}] \label{thm:main1} Let $X$ be a threefold with a full exceptional collection of four vector bundles.
Then:
\begin{enumerate}
\item \label{thm:main1:it1}
all full exceptional collections on $X$ consist of shifted vector bundles;
\item \label{thm:main1:it2}
the action of $B_4 \ltimes \ZZ^4$ by mutations and shifts is transitive and free. 
\end{enumerate}
\end{theorem} Note that in addition to the transitivity conjectured by Bondal and Polishchuk,  we also prove that the action is free. 
By Polishchuk's result,  \eqref{thm:main1:it1} implies the transitivity part of \eqref{thm:main1:it2}.
However,  we actually prove \eqref{thm:main1:it2} first using a new approach detailed below and then we deduce \eqref{thm:main1:it1} from it. 

\begin{remark} We remark that Theorem \ref{thm:main1}\eqref{thm:main1:it1} does not rule out the existence of exceptional objects in $\Dscr^b(X)$ which are not shifts of vector bundles, nor does it yield a classification of exceptional vector bundles on $X$, since even for $\PP^3$ it is unknown whether every exceptional vector bundle can be included in a full exceptional collection.  
\begin{remark} 
We should also stress that all exceptional collections in Theorem \ref{thm:main1} are assumed to be full.  Even for $\PP^3$ it is currently unknown whether every exceptional collection of length $4$ is full.  Recently an example of a rational surface with a non-full exceptional collection of maximal length was discovered by Krah \cite{Krah}.  \end{remark}
\end{remark}
\begin{remark} A general result by Positselski \cite{Positselski} states that if  $\rk K_0(X)=\dim X+1$ then every full strong exceptional collection 
(see Definition \ref{def:exceptional}) on $X$ consists of shifts
of vector bundles. So to prove Theorem \ref{thm:main1} it would be enough to show that every full exceptional collection is strong. We do not know how to do this directly however.
\end{remark}

Let us now elaborate on the way Theorem \ref{thm:introsphtwistfree} is applied to prove Theorem \ref{thm:main1}\eqref{thm:main1:it2}.  Our proof starts from the following elementary yet crucial observation.  Recall that if $H$ is a group then the \emph{Hurwitz action} of $B_n$ on $H^n$ is  given by (see \S\ref{subsec:hurwitz}):
\[
\sigma_i(g_1,\ldots,g_i,g_{i+1},\ldots, g_n)=(g_1,\ldots,g_i g_{i+1}g_i^{-1},g_i,\ldots, g_n)
\] 
Let $X$ be a Fano threefold of Picard rank $1$ and let $Y$ be a generic anticanonical divisor $Y\subset X$.  Then $Y$ is a K3 surface of Picard rank $1$.  The restriction of an exceptional object $E$ on~$X$ to $Y$ is a spherical object $E \vert_Y$.  Moreover,  if $X$ has a full exceptional collection,  then it turns out that the group $\Gscr \subset \Aut \Dscr^b(Y)$ generated by spherical twists and the Hurwitz action appear naturally: 
\begin{observation}[Lemma \ref{eq:twistmutation}]  \label{lem:twistmutation:intro}
Let $E$ be an exceptional object on $X$ and let $\phi(E):=T_{E\vert_Y}\in \Aut \Dscr^b(Y)$ be the spherical twist associated to the spherical object $E\vert_Y$. 
When extended to sequences of objects in $\Dscr^b(X)$, the map
$\phi$ intertwines the action of the braid group $B_n$ on the set of full exceptional collections in $\Dscr^b(X)$ and the Hurwitz action of $B_n$ on $\Gscr^n$. 
\end{observation}

\medskip

Theorem \ref{thm:introsphtwistfree} claims that the group $\Gscr$ in Observation \ref{lem:twistmutation:intro} is free.  More precisely,  it specialises to the following: 

\begin{theorem}[Theorem \ref{thm:examples}] \label{thm:examplesintro} Let $X$ be a threefold with a full exceptional collection of $4$ vector bundles,  namely,  ${X  = \PP^3, Q_3, V_5}$ or $V_{22}$. Let $Y \subset X$ be a generic anticanonical divisor which is a K3 surface of Picard rank $1$ and degree $4,6,10$ and $22$ respectively.  Then the group $\Gscr\subset \Aut \Dscr^b(Y)$
generated by all spherical twists is isomorphic to the free group $F_4$ on $4$ generators which are the spherical twists along the spherical vector bundles obtained by restricting a full exceptional collection of vector bundles on $X$ to $Y$.  
\end{theorem}

To summarise,  the action of the braid group $B_4$ on the set of full exceptional collections in $\Dscr^b(X)$ is closely related to the Hurwitz action of $B_4$ on $F_4^4$,  a power of a free group.  The latter action can be understood well abstractly \cite{NVdB}.  Note that in principle restricting to an anticanonical divisor looses some information.  However this turned out to be remedied by the fact that we are considering full exceptional collections and not individual exceptional objects.  We obtain Theorem \ref{thm:main1}\eqref{thm:main1:it2} from Theorem \ref{thm:examplesintro} and the fact that the product of the spherical twists associated to a full exceptional collection
only depends on the Serre functor on $X$ (this is a special case of \cite[Theorem A.1]{KuznetsovPerry}): 
\begin{proposition}[Proposition \ref{pr:compserre}]  Let $(E_1,\cdots, E_n)$ be a full exceptional collection in a Fano variety $X$ and let $i:Y\hookrightarrow X$ be a smooth anticanonical divisor.
Let $T_i :=T_{E_i\vert_Y}$.
One has  $ T_1 \dots T_n(-) =  - \otimes i^* \omega_X [2]$
in $\Aut \Dscr^b(Y)$ 
 \end{proposition} 
Combining this with our abstract result concerning the Hurwitz action on powers of free groups (\cite[Theorem 2.6(1)]{NVdB},  see Theorem \ref{thm:hurwitz_orbit}) we are able to deduce that any exceptional collection can be mutated
into one furnishing a chosen set of generators for $\Gscr$ (see Lemma \ref{lm:freegrp_implies_transitive}). This essentially finishes the proof of Theorem \ref{thm:main1}\eqref{thm:main1:it2}. 

\section{Acknowledgement}
We are very grateful to Sasha Kuznetsov and Tom Bridgeland for  the wealth of helpful comments and corrections on an earlier version of the text.  We also thank Tom Bridgeland for promptly answering our questions concerning \cite{BayerBridgeland} and Daniel Huybrechts for pointing out important missing references.  
The first author is also greatly indebted to Alexey Bondal for introducing her to the crucial results of \cite{BayerBridgeland}. 
\section{Generalities} 
We work over the complex numbers.  Unless otherwise specified $X$ is
a smooth projective variety.  We denote by $\coh(X)$ the category of
coherent sheaves on~$X$ and by $\Dscr^b(X) := \Dscr^b(\coh(X))$ the
bounded derived category of $\coh(X)$.  A Fano variety is a smooth
projective variety $X$ whose anticanonical divisor $-K_X$ is ample.
By a K3 surface $Y$ we mean a complex algebraic surface with trivial
canonical bundle $\omega_Y \simeq \Oscr_Y$ and $\H^1(Y, \Oscr_Y) = 0$.
We refer the reader to \cite{Huybrechts} for a comprehensive account
on K3 surfaces.

\subsection{Exceptional collections and Fano threefolds}

\begin{definition} \label{def:exceptional}
An object $E \in \Dscr^b(X)$ is \emph{exceptional} if $\Hom^i(E,E) = 0$ for any $i \neq 0$ and $\Hom^0(E,E) \cong \CC$.  An ordered collection $\Escr = (E_1, \dots,E_m)$ of objects of $\Dscr^b(X)$ is called an \emph{exceptional collection} if each $E_i$ is exceptional and $\Hom^\bullet(E_j,E_k) = 0$ for all $j>k$.  An exceptional collection $\Escr$ is \emph{full} if $\Dscr^b(X)$ is the smallest triangulated subcategory of $\Dscr^b(X)$ containing $E_1, \dots,E_m$.  We say that $\Escr$ is \emph{strong} if $\Hom^k(E_i,E_j) = 0$ for any $k \neq 0$. 
\end{definition} 

\begin{definition} Let $(E,F)$ be an exceptional pair.  We define the left (respectively right) \emph{mutation} of $(E,F)$ as $(L_E F,  E)$ (respectively $(F, R_F E)$),  where the objects $L_E F$ and $R_F E$ are given by the distinguished triangles
\begin{align*}
\Hom^\bullet(E,F) \otimes E \to F\to L_E F \\
R_F E \to E \to \Hom^\bullet(E,F)^\ast \otimes F 
\end{align*}

\end{definition} 
\begin{proposition}[{Bondal, \cite[Assertion 2.3]{MR992977}}]\label{pr:braidaction} The braid group $B_n$ acts on the set of exceptional collections in $\Dscr^b(X)$ of length $n$ by mutations.  More precisely,  the standard generator $\sigma_i \in B_n$, $1 \leq i \leq n-1$ acts by
\begin{align*} \sigma_i(E_1, \dots,E_{n}) = (E_1, \dots, E_{i-1}, L_{E_{i}}E_{i+1}, E_{i},E_{i+2}, \dots, E_{n})
\end{align*}
\end{proposition}

\begin{proposition}[{Bondal, \cite[Assertion 9.2]{MR992977}}]\label{prop:sheavesremain} Let $X$ be a smooth projective variety with $\rk K_0(X) = \dim X+ 1$.  Then mutations of full exceptional collections of shifted sheaves in $\Dscr^b(X)$ also consist of shifted sheaves.  
\end{proposition} 

\begin{lemma}\label{lm:4Fanos} Let $X$ be a smooth projective threefold.  The following conditions are equivalent.
\begin{enumerate}
\item \label{lm:4Fanos:it1} $\Dscr^b(X)$ admits a full exceptional collection of length $4$ consisting of vector bundles.  
\item \label{lm:4Fanos:it15} $X$ is Fano and $\Dscr^b(X)$ admits a full exceptional collection of length $4$.
\item \label{lm:4Fanos:it2} $X$ is Fano of Picard rank $1$ and $h^{1,2}(X) = 0$. 
\item \label{lm:4Fanos:it3} In the classification of \cite{IskovskikhProkhorov} (see also \cite{fanography}) $X$ is one of the following types of Fano threefolds: $\PP^3$,  the 3-dimensional quadric $Q_3$,  $V_5$ or $V_{22}$.
\end{enumerate}
\end{lemma}
\begin{remark} Among the four types in \ref{lm:4Fanos:it3} only $V_{22}$ has non-trivial moduli.  \end{remark}
\begin{proof} $\eqref{lm:4Fanos:it1} \implies \eqref{lm:4Fanos:it15}$: This is proved in \cite[Theorem 3.4]{BondalPolishchuk}. 
$\eqref{lm:4Fanos:it15} \implies \eqref{lm:4Fanos:it2}$: The fact that a smooth projective variety possessing a full exceptional collection has diagonal Hodge diamond ($h^{p,q}(X) = 0$ for $p\neq q$) is well known and follows from the additivity of Hochschild homology with respect to semiorthogonal decompositions (see e.g.  \cite[Corollary 7.5]{Kuznetsov09}).  For the fact that $X$ has Picard rank one, see the beginning of the proof of \cite[Theorem 3.4]{BondalPolishchuk} (as one can see,  the condition that the full exceptional collection consists of sheaves is not used in this part of the statement).  Note that we also automatically have $K_0(X) \cong \ZZ^4$.  $\eqref{lm:4Fanos:it2} \implies \eqref{lm:4Fanos:it3}$: This follows from the classification of Fano threefolds,  see \cite{IskovskikhProkhorov}.
$\eqref{lm:4Fanos:it3}\implies \eqref{lm:4Fanos:it1}$:
Full exceptional collections of vector bundles of length $4$ are known for each of the listed threefolds (constructed by Beilinson \cite{Beilinson}, Kapranov \cite{Kapranov}, Orlov \cite{OrlovV5} and Kuznetsov \cite{KuznetsovV22}, \cite{KuznetsovV22MPIM} respectively,  the constructions will be reviewed in \S \ref{sec:examples}).  \end{proof} 

\subsection{Spherical twists} 

\begin{definition} An object $S \in \Dscr^b(X)$ is \emph{spherical} if $\omega_X\otimes S\cong S$ and $\Hom^i(S,S) = \CC$ if ${i \in \{ 0, \dim X\}}$ and $\Hom^i(S,S) = 0$
  otherwise. The set of spherical objects in $\Dscr^b(X)$ up to isomorphism is denoted by $\Sph(X)$.
\end{definition} 

\begin{definition}[Seidel-Thomas, \cite{SeidelThomas}] A spherical object $S \in  \Dscr^b(X)$ gives rise to an autoequivalence $T_S \in \Aut  \Dscr^b(X)$ (referred to as \emph{spherical twist}) such that for every $A \in  \Dscr^b(X)$ there is a distinguished triangle 
\begin{align*}
\Hom^\bullet(S,A) \otimes S \xrightarrow{ev} A \to T_S(A).
\end{align*} \end{definition}  

Alternatively,  $T_S$ can be defined as the Fourier-Mukai transform with kernel $\cone(S \boxtimes S^\vee \to \Oscr_\Delta) \in \Dscr^b(X \times X)$.  

We also need a more general notion of a twist with respect to a spherical functor. 
\begin{definition} Let $F: \Bscr \to  \Dscr$ be a functor between triangulated categories with fixed dg-enhancements in the sense of \cite{BK} and assume that $F$ has a right adjoint $R$ and a left adjoint $L$.  The \emph{twist} $T_F $ and the \emph{cotwist} $C_F$ with respect to $F$ are defined via the distinguished triangles 
\begin{align*}
FR \xrightarrow{\varepsilon} \id_{\Dscr} \to T_F\\ 
C_F \to \id_{\Bscr}  \xrightarrow{\eta} RF,
\end{align*}
where $\varepsilon$ and $\eta$ are the counit and the unit of the adjunction $(F,R)$.  We say that the functor $F$ is \emph{spherical} if $T_F$ and $C_F$ are autoequivalences. \footnote{It is possible to define spherical functors without referring to dg-enhancements,  see \cite[Definition 2.8]{KuznetsovCY} (they are needed to define twists and cotwists, however).} \end{definition}

 \begin{remark} To make this definition sensible one needs to  assume in addition that $F, R$ and $L$ have lifts to quasi-functors and consider cones that descend from the dg level.  We refer the reader to \cite{AL} for a comprehensive treatment.  In the geometric context one can circumvent discussing dg-enhancements by considering Fourier-Mukai kernels.  
 \end{remark}

\begin{lemma}[{\cite[Lemma 8.21]{HuybrechtsFM}}]\label{lm:twistconj} Let $S$ be a spherical object and $\Phi \in \Aut \Dscr^b(X)$ an autoequivalence. Then 
\begin{align*} 
T_{\Phi(S)} \cong \Phi T_S \Phi^{-1}
\end{align*} 
 In particular,  for any two spherical objects $S_1, S_2$, one has $T_{T_{S_1}(S_2)} \cong T_{S_1} T_{S_2} T_{S_1}^{-1}$. \end{lemma} 

Passing from $S$ to $T_S$ loses some information,  but not much: 
\begin{lemma}\label{lem:sphericalshift} Assume $\dim X \geq 2$ and let $S,S'$ be two spherical objects in $\Dscr^b(X)$. Then $T_{S} \cong T_{S'}$ if and only if $S \cong S'[n]$ for some $n\in \ZZ$.
\end{lemma}
\begin{proof}
We explain the non-obvious direction.  So assume $T_S\cong T_{S'}$.
We then have $T_{S'}(S) = T_S(S) = S[1-\dim X]$.  On the other hand,  $T_{S'}(S)$ fits into the distinguished triangle 
\begin{align*}
\Hom^\bullet(S',S)\otimes S' \xrightarrow{} S \xrightarrow{} S[1-\dim X],
\end{align*}
Since $\dim X \geq 2$,  then the map $S \xrightarrow{0} S[1-\dim X]$ is $0$ because $\Hom(S,S[1-\dim X]) = 0$.  This implies that $S$ is a summand of $\Hom^\bullet(S',S)\otimes S'$.
Since both $S$ and $S'$ are indecomposable,  this yields that  $S = S'[n]$ for some $n \in \ZZ$.  
\end{proof}

\begin{remark} Lemma \ref{lem:sphericalshift} also holds for $\dim X = 1$, i.e.  $X$ a smooth elliptic curve.  However, we will not use it in this case, so we omit the proof. \end{remark} 
Now let $X$ be a smooth Fano variety and $i: Y \hookrightarrow X$ a smooth anticanonical divisor.  If $E \in \Dscr^b(X)$ is an exceptional object,  then the object $i^*E \in \Dscr^b(Y)$  is spherical (see \cite[Proposition 3.15]{SeidelThomas}). 

\begin{lemma}\label{eq:samehom} If $(E,F)$ is an exceptional pair in $\Dscr^b(X)$, then 
\begin{equation}
\label{eq:res}
\Hom_X^\bullet(E,F) \cong \Hom_Y^\bullet(i^* E, i^*F)
\end{equation}
\end{lemma}
\begin{proof} Observe that by adjunction 
\begin{align*} 
\Hom_Y(i^* E,i^* F[k])  \cong \Hom_X( E, F\otimes i_* \Oscr_Y[k])
\end{align*}
for any $k \in \ZZ$.  Applying $\Hom_X(E,-)$ to the distinguished triangle 
\begin{align*} 
F \otimes \omega_X\to F \to F\otimes i_* \Oscr_Y \r
\end{align*}
(obtained from identifying $\omega_X\cong \Oscr_X(-Y)$) we get the long exact sequence 
\begin{align*} 
\dots \to \Hom_X(E,F[k]) \to  \Hom_X( E, F\otimes i_* \Oscr_Y[k]) \to \Hom_X(E,F\otimes \omega_X[k+1]) \to \dots
\end{align*}
By Serre duality $\Hom_X(E,F\otimes \omega_X[k+1]) \cong  \Hom_X(F,E[\dim X-k-1])^* = 0$ for any $k \in \ZZ$, since $(E,F)$ is an exceptional pair.  Hence \eqref{eq:res} follows.
\end{proof}
 
 Now suppose $X$ has a full exceptional collection 
 \[
 {\Dscr^b(X) = \langle E_1, \dots, E_n\rangle}.
 \]
 
 We denote $T_k := T_{i^*E_k} \in \Aut \Dscr^b(Y) $.  The following fact is a special case of \cite[Theorem A.1]{KuznetsovPerry} by Kuznetsov and Perry and is well known to experts,  but we include the proof for the reader's convenience.  
 
\begin{proposition}\label{pr:compserre} There is an isomorphism of endofunctors of $\Dscr^b(Y)$ 
\begin{equation}
\label{eq:shift}
 T_1 \dots T_n(-) =  - \otimes i^* \omega_X [2]. 
\end{equation}
 \end{proposition} 

\begin{proof} Let $\Escr := \bigoplus_{k=1}^n E_k$ and $A:= \REnd_X(\Escr)$.  Note that $A$ is a smooth dg-algebra.  Then the functor 
\begin{equation}
\label{eq:equivalence}
\RHom_X(\Escr,-): \Dscr^b(X) \to \perf(A) 
\end{equation}
 is an equivalence of categories with quasi-inverse given by $-\otimes_A^\LL \Escr$ (see \cite{MR992977}).  Denote $\overline{\Escr} := \bigoplus_{k=1}^n i^* E_k$.  By \cite[Theorem 3.0.1]{Barbacovi} and Lemma \ref{eq:samehom}
the composition $T_1 \dots T_n$ is the twist along the spherical functor $$-\otimes_A^\LL \overline{\Escr}: \perf(A) \to  \Dscr^b(Y).$$ 
Taking the composition with the equivalence  \eqref{eq:equivalence} one observes that $T_1 \dots T_n$ is also the twist along the spherical functor $i^*: \Dscr^b(X) \to \Dscr^b(Y)$.  Thus,  for any $\Fscr \in \Dscr^b(Y)$ we have
\[
 T_1 \dots T_n(\Fscr) = T_{i^*}(\Fscr) = \cone(i^*i_* \Fscr \to \Fscr) =  \Fscr \otimes \Oscr_Y(-Y)[2],
\]
where the last equality is given by e.g.  \cite[Corollary 11.4]{HuybrechtsFM}. 
\end{proof}

\begin{remark}\label{rem:doublecover} We will also consider collections of spherical objects which occur in the following way.  
Let $Y$ be a K3 surface of degree $2$. Then $Y$ is realised as a double covering $\pi: Y \to \PP^2$ branched along a sextic curve.   If $E \in \Dscr^b(\PP^2)$ is exceptional, then $\RR\pi^*(E) \in \Dscr^b(Y)$ is spherical 
(\cite[Proposition 3.13,  Example 3.14b]{SeidelThomas}). One also checks that \eqref{eq:shift} holds for $n=3$ with the right-hand side replaced by $- \otimes \pi^* \omega_{\PP^2} [1]$.  This is another special case of  \cite[Theorem A.1]{KuznetsovPerry},  applied to the spherical functor $\pi^*$ instead of $i^*$. 
\end{remark} 

\subsection{The Hurwitz action}\label{subsec:hurwitz}

\begin{definition} Let $G$ be a group.  Then the $n$-strand braid group $B_n$ with the standard generators $\sigma_1, \dots, \sigma_{n-1}$ acts on $G^n$ via 
\begin{align*}
\sigma_i(g_1,\ldots,g_i,g_{i+1},\ldots, g_n)&=(g_1,\ldots,g_i g_{i+1}g_i^{-1},g_i,\ldots, g_n). 
\end{align*}
This action is referred to as \emph{the Hurwitz action}.\footnote{In \cite{NVdB} a  different convention for the Hurwitz action is used,  but this is inconsequential.}\end{definition} 

For $X$ Fano, $i: Y \hookrightarrow X$ a smooth anticanonical divisor and $E$ an exceptional object in $\Dscr^b(X)$ we denote by $\phi(E) := T_{i^*(E)} \in \Aut \Dscr^b(Y) $ the corresponding spherical twist.  Assume that $\Dscr^b(X)$ has a full exceptional collection of length $n$. We extend $\phi$ to the set $FEC(X)$ of full exceptional collections in $X$ by applying it term-wise: 
\begin{align*} 
\phi: FEC(X) &\to (\Aut \Dscr^b(Y))^n \\
(E_1,\dots,E_n) &\mapsto (\phi(E_1), \dots, \phi(E_n))
\end{align*}
Proposition \ref{pr:compserre} implies that the image of $\phi$ is contained in the subset 
\[
A_n(Y/X) := \{(g_1, \dots, g_n) \ \vert  \ g_1\cdots g_n = - \otimes i^* \omega_X[2]\} \subset (\Aut \Dscr^b(Y))^n,\]
which is stable under the Hurwitz action. 
\begin{lemma}\label{eq:twistmutation} 
The map $\phi$ intertwines the two $B_n$-actions,  namely,  the mutation action on $FEC(X)$ defined in Proposition \ref{pr:braidaction} and the Hurwitz action on $ (\Aut \Dscr^b(Y))^n$ (or $A_n(Y/X)$).
\end{lemma}
\begin{proof}
It suffices to prove that for an exceptional pair $(E,F)$ in $\Dscr^b(X)$ we have
\begin{align*}
\phi(L_E F ) = \phi(E) \phi(F) \phi(E)^{-1}.
\end{align*}
By Lemma \ref{lm:twistconj} one has \begin{align*}
\phi(E) \phi(F) \phi(E)^{-1} = T_{T_{i^*(E)}(i^*(F))}.
\end{align*}
Hence we need to show that $i^*(L_E F) = T_{i^*(E)}(i^*(F))$ or, equivalently, 
\begin{align*} i^* \cone(\Hom^\bullet_X(E,F) \otimes E \to F)  \cong \cone(\Hom^\bullet_Y(i^* E,i^* F) \otimes i^* E \to i^* F) 
\end{align*} 
This now follows from Lemma \ref{eq:samehom}. 
 \end{proof} 

Let $F_m = \langle x_1, \dots, x_m \rangle$ be the free group on $m$ generators.  We are going to use the following results concerning the Hurwitz action on powers of free groups. 

\begin{theorem}[e.g. {\cite[XI, Corollary 1.8]{MR2463428}}]\label{thm:hurwitz_stab} The stabiliser of $(x_1, \dots, x_m) \in F_m^m$ under the Hurwitz action of $B_m$ on $(F_m)^m$ is trivial.  \end{theorem}
\begin{theorem}[{\cite[Example 1.3]{NVdB}}]\label{thm:hurwitz_stab1} The stabiliser of $(x_1,x_2,x_1,x_2)$ under the Hurwitz action of $B_2$ on $(F_2)^4$ is freely generated by $\sigma_1^{-1} \sigma_2 \sigma_1$ and $\sigma_3 \sigma_2 \sigma_3^{-1}$.  \end{theorem}

\begin{theorem}[{\cite[Theorem 2.6(1)]{NVdB}}]\label{thm:hurwitz_orbit} Under the Hurwitz action of $B_n$ on $(F_m)^n$,  the orbit of $t = (t_1, \dots,t_n)$ with $t_i \in \{x_1, \dots, x_m\}$ is given by the set of elements in $$\bigcup_{\tau \in S_n} C(t_{\tau(1)})\times\cdots\times C(t_{\tau(n)})$$ whose product is $t_1\cdots t_n$, where $C(x)$ denotes the conjugacy class of $x$ and $\tau \in S_n$ runs over all permutations of $n$ elements. 
\end{theorem}

\begin{remark}\label{rem:permutfree} In fact,  one can see from the proof of Theorem \ref{thm:hurwitz_orbit} or by passing to the abelianisation that any element in 
\[
C(x_{i_1}) \times \dots \times C(x_{i_n}) \in (F_m)^n
\]
with product $t_1\cdots t_n$ and $1 \leq i_j \leq n$ is in the orbit of $t = (t_1, \dots, t_n)$.  In other words,  the condition that $x_{i_1}, \dots, x_{i_n}$ is a permutation of $t_1, \dots, t_n$ is automatic. 
 \end{remark} 

\subsection{Bridgeland's stability conditions and derived autoequivalences of K3 surfaces}\label{sec:K3prelim}
Let $Y$ be a K3 surface with Picard rank $1$.
Let $H$ be an ample generator of ${\Pic(Y) \cong \NS(Y)}$ and put $H \cdot H = d:=2\delta$.  The extended Neron-Severi group (or equivalently the numerical Grothendieck group)
\begin{equation}\label{eq:Z3}
\Nscr(Y) = \H^0(Y,\ZZ) \oplus \Pic(Y) \oplus  \H^4(Y,\ZZ) \cong \ZZ^3
\end{equation}
is equipped with the Mukai pairing given by 
\[ \langle (r_1,l_1,s_1), (r_2,l_2,s_2) \rangle  := dl_1l_2 - r_1s_2 - r_2s_1 
\]
or, equivalently,  a quadratic form $q(r,l,s) = \delta l^2 - rs$. 
In particular $\Nscr(Y)$ has signature $(2,1)$.
For an object $E \in \Dscr^b(Y)$ let
\begin{equation}
\label{eq:mukai}
v(E) := (\rk(E),c_1(E), c_1(E)^2/2 - c_2(E) + \rk(E)) \in \Nscr(Y)
\end{equation}
 be the corresponding Mukai vector.  Then for $E, F \in \Dscr^b(Y)$ we have 
\[
\langle v(E), v(F) \rangle = - \chi(E,F) = - \sum_{i \in \ZZ} (-1)^i \dim \Hom^i(E,F).
\] 
Hence if $E \in \Dscr^b(Y)$ is a  $2$-spherical object and $v(E) = (r,l,s)$, then 
\[
 \langle v(E), v(E) \rangle = dl^2 - 2rs = -2
\]
 We say that a Mukai vector $\rho \in \Nscr(Y)$ is\emph{ spherical} (or a \emph{root}) if $\langle \rho, \rho \rangle = -2$. Set 
 \begin{align*}
  \Delta(Y) := \{ \rho \in \Nscr(Y) \ | \  \langle \rho,  \rho \rangle = -2 \} \\
  \Delta^+(Y) := \{ \rho = (r,l,s) \in \Delta(Y) \ | \  r > 0 \} 
\end{align*} 
 For any $\rho \in \Delta^+(Y)$ there exists a unique spherical vector bundle $E_\rho$ with $v(E_\rho) = \rho$ and any spherical sheaf on $Y$ is a slope stable vector bundle.  This is a combination of the results of Yoshioka \cite{Yoshioka} and Mukai \cite{Mukai87},  see footnote (2) on page 10 of \cite{BayerBridgeland}.
 
If $E \in \Dscr^b(Y)$ is a spherical object,  then the spherical twist $T_E$ acts on cohomology by the reflection at the hyperplane orthogonal to $v(E) \in \Nscr(Y)$: 
\[
 s_{v(E)}(\alpha) = \alpha + \langle \alpha, v(E) \rangle v(E).
 \]

Let $\Aut^+ \H^*(Y,\ZZ) \subset \Aut \H^*(Y,\ZZ) $  be the subgroup of
all Hodge isometries which preserve the orientation of the positive
4-planes\footnote{The Mukai lattice $\H^*(Y,\ZZ)$ has signature $(4,20)$. We say that an isometry $\rho \in \Aut \H^*(Y,\ZZ)$ preserves the orientation of the positive 4-planes if under the orthogonal projection a fixed orientation of a positive 4-plane in $\H^*(Y,\ZZ)$ is the same as the induced orientation on its image under $\rho$.  This does not depend on the choice of a positive 4-plane and an orientation on it.  For more details we refer to \cite{HMS}.}.  Analogously we define $\Aut^+ \Nscr(Y)$ as the group of isometries of $\Nscr(Y)$ which preserve the orientation of the positive 2-planes. By a result due to
Huybrechts, Macr\`i and Stellari (\cite[Theorem 2]{HMS}, see also
\cite[Corollary 16.3.13]{Huybrechts}) the subgroup
$\Aut^+ \H^*(Y,\ZZ)$ is precisely the image of the natural map
${\Aut \Dscr^b(Y) \to \Aut \H^*(Y,\ZZ)}$.  We denote the kernel of
this map by $\Aut^0 \Dscr^b(Y)$. 
The space
\begin{align}\label{eq:defP}
\Pscr(Y) := \{v \in  \Nscr(Y)_{\CC} \ | \ \Re(v), \Im(v) \text{ span a positive 2-plane}\}
\end{align}
has two connected components distinguished by the orientation.  Let $ \Pscr^+(Y)$ be the component containing $\exp(iH) = (1,i,-\delta)$. Define an open subset 
\begin{align*} \Pscr_0^+(Y)  :=  \Pscr^+(Y) \setminus \bigcup_{\rho \in \Delta(Y)} \rho^\perp,
\end{align*}
 where $\rho^\perp$ denotes the orthogonal complement with respect to the extended Mukai pairing on $\Nscr(Y)_\CC$.  
 
 For the details regarding Bridgeland's stability conditions on K3 surfaces and in general we refer the reader to \cite{BridgelandStab}, \cite{BridgelandK3}.  
 
\begin{definition}\label{def:stab}
A \emph{numerical stability condition} $\sigma=(Z,\Pscr)$ on $\Dscr^b(Y)$ consists of a group homomorphism $Z: \Nscr(Y) \to \CC$ (referred to as \emph{the central charge}) and a collection of full additive subcategories $\Pscr (\varphi)\subset  \Dscr^b(Y)$ for each \emph{phase} $\varphi\in \RR$ such that:
\begin{enumerate}
  \item If $0\neq E\in \Pscr(\varphi)$, then $Z(E) = m(E)\exp(i\pi\varphi)$
    for some $m(E)\in \RR_{>0}$.
  \item $\Pscr (\varphi+1) = \Pscr (\varphi)[1]$ for all $\varphi\in \RR$.
  \item If $\varphi_1>\varphi_2$ and $A_j\in \Pscr(\varphi_j)$, then $\Hom_{ \Dscr^b(Y)}(A_1,A_2)=0$. 
  \item For every $0\neq E\in  \Dscr^b(Y)$ there is a finite sequence of real numbers $\varphi_1>\varphi_2>\cdots>\varphi_n$ and a collection of distinguished triangles

\noindent
\begin{tikzpicture}[
  >=stealth,
  baseline=(E1.base),
  x=\dimexpr\linewidth/14\relax, 
  y=1cm
]
  \node (E0)   at (0,0)    {$0 = E_0$};
  \node (E1)   at (2.8,0)  {$E_1$};
  \node (E2)   at (5.6,0)  {$E_2$};
  \node (dots) at (8.4,0)  {$\cdots$};
  \node (Enm1) at (11.2,0) {$E_{\,n-1}$};
  \node (En)   at (14,0)   {$E_n = E$};
  \draw[->] (E0)--(E1); \draw[->](E1)--(E2); \draw[->](E2)--(dots);
  \draw[->] (dots)--(Enm1); \draw[->](Enm1)--(En);
  \node (A1) at ($(E0)!0.5!(E1) + (0,-2)$) {$A_1$};
  \node (A2) at ($(E1)!0.5!(E2) + (0,-2)$) {$A_2$};
  \node (An) at ($(Enm1)!0.5!(En) + (0,-2)$) {$A_n$};
  \draw[->] (E1)--(A1); \draw[->, dashed](A1)--(E0);
  \draw[->] (E2)--(A2); \draw[->, dashed](A2)--(E1);
  \draw[->] (En)--(An); \draw[->, dashed](An)--(Enm1);
\end{tikzpicture}
with $0 \neq A_j \in \mathcal{P}(\varphi_j)$ for all $j$.
\end{enumerate}
\end{definition}
 
Lemma 5.2 in \cite{BridgelandStab} shows that the categories $\Pscr(\varphi)$ in the definition above are in fact abelian for all $\varphi \in \RR$.  For $I \subset \RR$ an interval,  we let $\Pscr(I) \subset \Dscr^b(Y)$ be a subcategory generated by all $\Pscr(\phi)$ with $\phi \in I$.  A stability condition is called \emph{locally finite} if there exists $\varepsilon >0 $ such that the category $\Pscr ((\varphi - \varepsilon, \varphi+\varepsilon))$ has finite length for all $\varphi \in \RR$.  Denote the space of numerical locally finite stability conditions on $\Dscr^b(Y)$  by $\Stab(Y)$.  As shown in \cite[Corollary 1.3]{BridgelandStab}, $\Stab(Y)$ is a finite-dimensional complex manifold.

Recall that there is a natural left action of $\Aut \Dscr^b(Y)$ on $\Stab(X)$. This action is faithful and commutes with the right action of the universal cover $\widetilde{\GL^+(2,\RR)}$ of the group of orientation-preserving automorphisms $\GL^+(2,\RR)$ (\cite[Lemma 8.2]{BridgelandStab}).  There is a subgroup $\CC \subset \widetilde{\GL^+(2,\RR)}$ acting freely by rotating the central charge and relabelling the phases. 
 
  Since the Mukai pairing is non-degenerate,  for every $Z: \Nscr(Y) \to \CC$ there exists $Z^\vee \in \Nscr(Y)_{\CC}$ such that $Z(-) = \langle Z^\vee, - \rangle$.  Let 
  \[
  \pi: \Stab(Y) \to \Nscr(Y)_{\CC}
  \]
   be the forgetful map sending a stability condition $(Z,\Pscr)$ to $Z^\vee$.  By \cite[Theorem 1.2]{BridgelandStab},  this is a local homeomorphism. 

\begin{theorem}[Bridgeland,  {\cite[Theorem 1.1]{BridgelandK3}}]\label{thm:BK3} Let $Y$ be a K3 surface.  There exists a connected component $\Stab^\dagger(Y)$ of $\Stab(Y)$ which is mapped onto $\Pscr_0^+(Y)$ by $\pi$.  Moreover,  the induced map $\pi: \Stab^\dagger(Y) \to \Pscr_0^+(Y)$ is a Galois covering and the subgroup  of $\Aut^0 \Dscr^b(Y)$ preserving $\Stab^\dagger(Y)$ can be identified with the group of deck transformations of $\pi$. 
\end{theorem} 

Theorem \ref{thm:BK3} implies that there is an injective homomorphism 
\[
 \varphi: \pi_1( \Pscr_0^+(Y)) \to \Aut^0 \Dscr^b(Y).
 \]
 It was conjectured by Bridgeland \cite[Conjecture 1.2]{BridgelandK3} that $\Stab^\dagger(Y)$ is simply-connected and is preserved by $\Aut \Dscr^b(Y)$ (hence,  in particular,  $\varphi$ is bijective).  For K3 surface of Picard rank $1$ this conjecture was proved by Bayer and Bridgeland: 

 \begin{theorem}[Bayer-Bridgeland,  {\cite[Theorem 1.3, Remark 6.10]{BayerBridgeland}}]\label{thm:mainBB} Let $Y$ be a K3 surface with Picard rank $1$.  Then $\Stab^\dagger(Y)$ is preserved by $\Aut \Dscr^b(Y)$ and is contractible.  In particular,  there is an exact sequence of groups 
 \begin{equation}
\label{eq:mainBB}
 1 \to  \Aut^0 \Dscr^b(Y)  \cong \pi_1( \Pscr_0^+(Y)) \to \Aut \Dscr^b(Y) \to \Aut^+ \H^*(Y,\ZZ)  \to 1.
\end{equation}
Moreover,  for any $\rho \in \Delta(Y)$,  the group $\Aut^0 \Dscr^b(Y) $ acts transitively on the set of all spherical objects with Mukai vector $\rho$. 
 \end{theorem}
 
 The following explicit description of the fundamental group  $\pi_1( \Pscr_0^+(Y))$ was obtained by Kawatani: 
 \begin{theorem}[Kawatani,  {\cite[Theorem 1.3]{Kawatani}}]\label{thm:kawatani} Let $Y$ be a K3 surface with Picard rank 1.  The group $ \pi_1( \Pscr_0^+(Y))  \cong \ZZ \times \langle T_{E_\rho}^2 \ | \ \rho \in \Delta^+(Y) \rangle$ is the product of $\ZZ$ acting by even shifts and the free group generated by the squares $T_{E_\rho}^2$ of spherical twists corresponding to all spherical vector bundles $E_\rho$, $\rho \in \Delta^+(Y)$.  
 \end{theorem}
 
 \begin{remark}\label{rem:kawatani} We will use these explicit generators of $\Aut^0 \Dscr^b(Y)$ later on.  Note that in general the fundamental group of a topological space does not have a canonical set of generators, and the fundamental groups associated with different base points are not canonically identified,  so let us briefly explain where the generators in Theorem \ref{thm:kawatani} come from.  One can show that $\Pscr_0^+(Y)$ is diffeomorphic to $\GL_2^+(\RR) \times \HH^0$ (see the beginning of \S\ref{sec:SO} below),  where $\HH^0$ is  the upper half plane $\HH$ minus a discrete set of points corresponding to roots in $\Delta^+(Y)$. 
  Hence $\pi_1( \Pscr_0^+(Y)) \cong \ZZ \times \pi_1(\HH^0)$.  Let $\Vscr(Y) \subset \HH^0$ be the open subset obtained by removing the vertical line segments connecting the real line and the punctures (see Figure 1 on page 11 of \cite{BayerBridgeland})\footnote{To see that these vertical line segments do not intersect,  one observes that the real parts of all punctures are distinct.  By Corollary \ref{cor:kawatani} the real part of the point corresponding to $\rho = (r,l,s) \in \Delta^+(Y)$ is $l/r$.  But $l$ and $r$ are coprime,  since $\delta l^2 - rs = -1$,  so $s$ is determined uniquely from $l/r$ and hence $\rho$ is the unique root whose corresponding puncture has real part $l/r$}.  $\Vscr(Y)$ is open and contractible and has a canonical lift to $\Stab^\dagger(Y)$,  as explained in \cite[\S 6]{BridgelandK3}.  Its points are stability conditions corresponding to tiltings of the heart $\coh(Y)$. Pick an arbitrary base point $y$ in $\Vscr(Y)$ 
(since $\Vscr(Y)$ is contractible the choice does not matter) and consider a collection of paths connecting it to punctures and not intersecting the relative interiors of the vertical segments.   The corresponding collection of loops around the punctures in the direction of the orientation (they intersect the vertical segments exactly once) gives a canonical set of generators of $\pi_1(\HH^0,y)$.\end{remark} 
\begin{corollary} \label{cor:orbitreps}
Let $Y$ be a K3 surface with Picard rank 1.
Let $\Gscr\subset \Aut \Dscr^b(Y)$ be the subgroup generated by spherical twists and let $G$ be its image in $\Aut^+ \H^\ast(Y,\ZZ)$. The  map
\begin{align*}
\Sph(Y)/( \langle \ZZ[2],\Gscr \rangle )\r \Delta(Y)/G  \\
S\mapsto v(S)
\end{align*}
is a bijection.
\end{corollary}
\begin{proof} Surjectivity follows from the fact that every $\rho\in \Delta^+(Y)$ is the Mukai vector of the spherical vector bundle $E_\rho$.
To prove injectivity assume $S,T\in \Sph(Y)$ are such that $f(v(S))=v(T)$ for $f\in G$. Lifting $f$ to $F\in \Gscr$ we obtain
$v(F(S))=v(T)$. But then by Theorem \ref{thm:kawatani} we have $T=F'(F(S))$ for ${F'\in \ZZ[2] \times \{H^2\mid H\in \Gscr\}\subset \langle \ZZ[2],\Gscr \rangle}$. This finishes the proof.
\end{proof}

\subsection{Fuchsian groups}
\label{sec:fuchsian}
We refer to \cite{Katok} for the basics on Fuchsian groups. The group $\PSL_2(\RR)$ acts on the upper half plane model $\HH$ of the hyperbolic plane by \emph{M\"obius transformations}, i.e.
\[
z\mapsto \frac{az+b}{cz+d}
\]
with
\[
\begin{pmatrix}
a&b\\
c&d
\end{pmatrix}\in \PSL_2(\RR).
\]
In fact $\PSL_2(\RR)$ is exactly the Lie group of orientation-preserving isometries of the hyperbolic plane.   
Note that $\HH\subset \CC=\AA^1(\CC)\subset \PP^1(\CC)$. We let $\overline{\HH}$ be the closure of $\HH$ in $\PP^1(\CC)$.
We have $\overline{\HH}=\HH\cup \PP^1(\RR)=\HH\cup \RR\cup \{\infty\}$. The action of the group $\PSL_2(\RR)$ extends to $\overline{\HH}$. 
\begin{definition} A \emph{Fuchsian group} is a discrete subgroup of $\PSL_2(\RR)$.
\end{definition} 

\begin{theorem}[{\cite[Theorem 2.2.6, Corollary 2.2.7]{Katok}}]\label{thm:fuchsian} Let $\Gamma \subset \PSL_2(\RR)$ be a subgroup. Then the following are equivalent: 
\begin{enumerate}
\item $\Gamma$ is Fuchsian. 
\item $\Gamma$ acts properly discontinuously on $\HH$ (that is,  for any compact set $K \subset \HH$ there are only finitely many $g \in \Gamma$ with $gK \cap K \neq \varnothing$). 
\item All orbits $\Gamma x$ for $x\in \HH$ are discrete subsets of $\HH$. 
\end{enumerate} \end{theorem}

Theorem \ref{thm:fuchsian} implies that for a Fuchsian group $\Gamma$ a $\Gamma$-stabiliser of a point $x \in \HH$ is finite (otherwise there is a compact set $K$ such that $\varnothing \neq gK \cap  K \ni x$ for infinitely many $g \in \Stab(x)$).  In fact,  every non-trivial $\Gamma$-stabiliser of a point is a maximal finite cyclic subgroup (see \cite[p. 71 and Theorem 2.3.5]{Katok}).  The maximality uses the fact that if $g \in \PSL(2,\ZZ)$ has a fixed point $x \in \HH$,  then this fixed point is unique (see \cite[{\S 2.1}]{Katok}).  Maximal finite cyclic subgroups of $\Gamma$ are referred to as \emph{elliptic subgroups}.  Conversely,  one can show that every elliptic subgroup of $\Gamma$ is the stabiliser of a point $x \in \HH$.  We say that an element $g \in \Gamma$ is \emph{elliptic} if $\langle g \rangle$ is an elliptic subgroup.  A point $x  \in \HH$ is \emph{elliptic of order $k$} if $gx = x$ for an elliptic element $g \in G$ of order $k$.  A \emph{parabolic} element $g \in \Gamma$ is a non-trivial element stabilising a point $x \in \RR \cup \{\infty \}$ for the action of $G$ extended to $\overline{\HH}$.  A point $x \in \RR \cup \{\infty \}$ is a \emph{cusp} if $gx = x$ for a parabolic element $g \in \Gamma$.

Theorem \ref{thm:fuchsian} implies in particular that the topological quotient $V := \HH/\Gamma$ is well defined and is an orbifold Riemann surface.  The orbifold points of $V$ correspond to elliptic elements of $\Gamma$ up to conjugacy and the cusps correspond to parabolic elements up to conjugacy.  The orbifold fundamental group $\pi_1(V)$ is isomorphic to $\Gamma$.
The following lemma says that the quotient is not pathological.
\begin{lemma} Let $\Gamma$ be a Fuchsian group. The orbifold points are discrete in $\HH/\Gamma$.
\end{lemma}
\begin{proof}
This follows from the fact that the set of orbifold points is $\Gamma$-invariant and discrete in $\HH$ (see \cite[Corollary 2.2.8]{Katok}).
\end{proof}
An important piece of data associated with a Fuchsian group is  a \emph{fundamental domain} (or \emph{region}).
\begin{definition}[{\cite[\S3.1]{Katok}}] Let $\Gamma$ be a Fuchsian group.
A closed region $\Fscr\subset \HH$ (i.e.\ the closure of a non-empty open set $\Fscr^\circ$, the interior of $\Fscr$) is said to be a \emph{fundamental domain (region)} for $\Gamma$
if 
\begin{enumerate}[label=(\emph{\roman*}), ref=\emph{\roman*}]
\item $\bigcup_{g\in \Gamma}\Fscr=\HH$.
\item $\Fscr^\circ \cap g(\Fscr^\circ)=\varnothing$ for all $g\in \Gamma-\{1\}$.
\end{enumerate}
\end{definition}
A hyperbolically convex fundamental domain always exists,  see \cite[\S3.2]{Katok}. 
In the sequel we will need the following result on the existence of hyperbolically convex fundamental domains satisfying an additional condition.  We define a \emph{translation} $t_\alpha$ for $\alpha\in \RR$ to be  a M\"obius transformation of the form $z\mapsto z+\alpha$. 
\begin{lemma} \label{lem:closure}
Let $\Gamma$ be a Fuchsian group which is normalised by some translation $t_\alpha$, $\alpha>0$ such that $\langle \Gamma,t_\alpha\rangle$ is still Fuchsian.
Then $\Gamma$ has a hyperbolically convex fundamental domain containing $\infty$ in its closure in $\overline{\HH}$.
\end{lemma}
\begin{proof} For $g=\left(\begin{smallmatrix} a&b\\c&d\end{smallmatrix}\right)$ with $\det g=1$ an $c \neq 0$ define
\[
E_g=\{z\in \HH\mid |cz+d|>1\}.
\]
Note that $E_g$ is a half-plane in the hyperbolic plane, and in particular it is hyperbolically convex.
Assume first that some $t_\alpha$ is actually contained in $\Gamma$ and let $\alpha$ be minimal with this property.  Let $\Gamma_0$ be the subgroup of $\Gamma$ generated by $t_\alpha$. Then according to \cite[Definition 5.6]{CarolineSeries} we can take $\Fscr$ to be the  \emph{Ford domain}
\[
\Fscr=\overline{\Sigma_{\theta,\alpha} \cap \bigcap_{g\in \Gamma-\Gamma_0} E_g}
\]
with 
\[
\Sigma_{\theta,\alpha}=\{z\in \HH\mid \theta <\Re z<\theta+\alpha\}
\]
for $\theta \in \RR$ arbitrary.

Assume now that $\Gamma$ does not contain a non-trivial translation and put $\overline{\Gamma}=\langle \Gamma,t_\alpha\rangle$. 
Then $\overline{\Gamma}=\coprod_{n\in \ZZ} t_\alpha^n \Gamma$.  In particular,  $\overline{\Gamma}-\overline{\Gamma}_0=\{t_\alpha^ng\mid n\in \ZZ, g\in \Gamma-\{1\}\}$.
Note that $E_{t_\alpha^ng}=E_g$ for any $g, \alpha$ and $n$.  It follows that $\bar{\Gamma}$ has a fundamental domain
\[
\overline{\Sigma_{\theta,\alpha} \cap \bigcap_{g\in \Gamma-\{1\}} E_g}.
\]
From the coset decomposition and Lemma \ref{lem:coset} below we then get that $\Gamma$ has a fundamental domain
\[
\bigcup_{n} t^n_\alpha(\overline{\Sigma_{\theta,\alpha} \cap \bigcap_{g\in \Gamma-\{1\}} E_g})=\overline{\bigcap_{g\in \Gamma-\{1\}} E_g}
\]
Such a domain may be called a Ford domain of infinite width.  Here we have used that the intersection is obviously locally finite and furthermore
$t_\alpha(E_g)=E_{t_\alpha g t^{-1}_\alpha}$ implying that $\bigcap_{g\in \Gamma-\{1\}} E_g$ is $t_\alpha$-invariant.
\end{proof}
We have used the following standard lemma in the proof above.
\begin{lemma}
\label{lem:coset}
Let $\Lambda\subset \Gamma$ be Fuchsian groups and let $\Gamma=\coprod_{i\in I} \Lambda \gamma_i$ be a coset decomposition. Assume that $\Fscr$ is a fundamental domain for $\Gamma$
such that the collection of closed sets  $(\gamma_i \Fscr)_i$ is locally finite. Then $\bigcup_{i\in I} \gamma_i \Fscr$ is a fundamental domain for $\Lambda$.
\end{lemma}
\begin{proof}
See e.g.\ \cite[Theorem 12]{YurikoUmemoto}. The proof in loc.\ cit.\ assumes that $I$ is finite,  but one checks that it goes through under the local finiteness hypothesis.
\end{proof}

Recall the standard definition of the Hecke congruence subgroups,  which are examples of Fuchsian groups:\footnote{We take our congruence subgroups to be subgroups of $\PSL_2(\RR)$ rather than $\SL_2(\RR)$ since we are mainly interested in their action on the upper half plane $\HH$ by M\"obius transformations.}
\begin{align*} 
\Gamma_0(n) := \bigg\{  \begin{pmatrix}  a & b \\ c & d\end{pmatrix} \in \PSL_2(\ZZ) \ | \ c \equiv 0 \pmod n \bigg\}
\end{align*} 
\begin{definition}\label{def:AL} Let $e\divides n$ satisfy $e>0$. An \emph{Atkin–Lehner element}\footnote{Traditionally an Atkin-Lehner element is normalised to have determinant $1$ which involves dividing by $\sqrt{e}$. 
We do not
do this here since we consider $\Gamma^+_0(n)$ as a subgroup of $\PGL^+_2(\QQ)$.}
 associated to $e$ is a matrix with determinant $e$ of the form 
\[
\begin{pmatrix}
ae&b\\
cn & de
\end{pmatrix}\in M_2(\ZZ).
\] 
It is a trivial exercise to show that an Atkin-Lehner element associated to $e$ exists if and only if $e$ is an \emph{exact divisor} of $n$, i.e.\ $\gcd(e,n/e)=1$.  
The set of the images of all Atkin-Lehner elements form a group \footnote{Let $e$ and $e'$ be exact divisors of $n$.  Then one checks that,  up to the scalar factor $\gcd(e,e')$,  the product of two Atkin-Lehner corresponding to $e$ and $e'$ respectively is an Atkin-Lehner corresponding to $ee'/ \gcd(e,e')^2$, which is again an exact divisor of $n$. } which is referred to as the \emph{Atkin-Lehner} group and is denoted by
$\Gamma^+_0(n)\subset \PGL_2^+(\QQ)$.
\end{definition} 
The group $\Gamma_0^+(n)$  contains $\Gamma_0(n)$, and for square free $n$,  it is its normaliser (see \cite[\S 3]{ConwayNorton}).  We have $\Gamma^+_0(n)/\Gamma_0(n)\cong (\ZZ/2\ZZ)^c$ where $c$ the number of distinct prime factors in $n$.

The special Atkin-Lehner element (associated to $e = n$)
\[ 
f_n = \begin{pmatrix}
               0 & -1 \\ n & 0
\end{pmatrix} \in \Gamma_0^+(n)
\]
is called the \emph{Fricke involution}. 

 It is clear that $f_n$ normalises $\Gamma_0(n)$.  The  \emph{Fricke group} is defined as
  \[ 
 F(n) := \langle \Gamma_0(n), f_n \rangle \subseteq \Gamma_0^+(n)
 \]
and one has $F(n)/\Gamma_0(n)=\ZZ/2\ZZ$. 
So $F(n) = \Gamma_0^+(n)$ if and only if $n$ is a prime power. Hence the smallest $n$ for which $F(n)$ and $\Gamma^+_0(n)$ differ is $n=6$. For example,  we have
\[ 
\begin{pmatrix}
2 & -1 \\ 6 & -2
\end{pmatrix}, 
\begin{pmatrix}
3 & -1 \\ 12 & -3
\end{pmatrix} \in \Gamma_0^+(6) \setminus F(6).
\]

The groups $\Gamma_0^+(n)$ and the corresponding modular curves $X_0^+(n) := \HH / \Gamma_0^+(n)$ have been studied extensively.  For instance,  Ogg \cite{Ogg} observed that the primes $p$ for which $X_0^+(p)$ has genus $0$ are precisely the prime factors of the order of the Monster group.   

The curves $X_0^+(n)$ are non-compact and the groups $\Gamma_0^+(n)$ have a finite number of elliptic points and cusps (in other words, they are \emph{orbifold Riemann surfaces of finite type}).  This is a consequence of the fact that all these groups are commensurable with $\PSL_2(\ZZ)$,  for which the fact is classical.  For $n$ square-free $\Gamma_0^+(n)$ is known to have a unique cusp (see e.g. \cite{ChuaLang}). 

\begin{definition}\label{def:pointrefl} Let $z_0 \in \HH$.  The \emph{point reflection} at $z_0$ is a M\"obius transformations which sends any $z \in \HH$ to the point $z'$ on the geodesic through $z_0$ and $z$ such that $|z_0z| = |z_0z'|$.  Explicitly, if $z_0 = a+bi$ the point reflection at $z_0$ is represented by 
\[
\begin{pmatrix}
a & - (a^2+b^2) \\ 
1 & -a 
\end{pmatrix} \in \PSL(2,\RR).
\] \end{definition}

Note that, for instance, the Fricke involution $f_n$ is a point reflection at $\frac{i}{\sqrt{n}}$. 
\section{Geometric considerations around $\Pscr^+(Y)$, $\Pscr_0^+(Y)$} \label{sec:SO}
\subsection{Introduction}
\label{sec:introgeom}
In this section we give a more or less self-contained discussion of some results of Kawatani in \cite{Kawatani}.
Note that  $\Pscr(Y) \subset \Nscr(Y) \otimes \CC$,  as introduced in \S\ref{sec:K3prelim} is equipped
with obvious commuting $O(q,\RR)$ and $\GL_2^+(\RR)$-actions (the latter is free and is obtained via identifying $\CC\cong \RR^2$).
We let $O^+(q,\RR)$ be the index two subgroup of $O(q,\RR)$ which preserves the connected component $\Pscr^+(Y)$.

Following \cite{BridgelandK3} Kawatani considers the morphism 
\begin{align*}
e:\HH &\r \Pscr^+(Y) \\
z &\mapsto \exp(zH)
\end{align*}
and shows that this defines a $\GL_2^+(\RR)$-equivariant isomorphism
\begin{align*}
\HH\times \GL^+_2(\RR) &\cong\Pscr^+(Y) \\
(z,g) &\mapsto exp(zH)\cdot g.
\end{align*}
In particular we may transfer the standard hyperbolic metric on $\HH$ to $\Pscr^+(Y)/\GL^+_2(\RR)$. In \S\ref{sec:hyperboloid},  \ref{sec:BK} we give an intrinsic definition 
for this hyperbolic metric by showing that $\Pscr^+(Y)/\GL^+_2(\RR)$ may be identified with the hyperboloid model for the hyperbolic plane. 
We also give a convenient model of $\Pscr^+(Y)/\GL^+_2(\RR)$ as the open subset $\PP(\Nscr(Y)_{\RR})_{<0}$ of $\PP(\Nscr(Y)_{\RR})$ consisting of negative vectors for the Mukai quadratic form.

This identification is  such that $\overline{\{\rho^\perp\mid \rho\in\Delta^+(Y)\}}\subset \Pscr^+(Y)/\GL^+_2(\RR)$ is identified  with 
$\{\bar{\rho}\mid \rho\in \Delta^+(Y)\}\subset \PP(\Nscr(Y)_{\RR})_{<0}$
and the latter is again identified with a discrete set of points in $\HH$ via the formula \eqref{eq:rhoformula} below. Hence if we use $(-)_0$ or $(-)^0$ to denote the
complement of root data then we have the identifications
\[
\Pscr_0^+(Y)/\GL^+_2(\RR)\cong \PP(\Nscr(Y)_{\RR})^0_{<0} \cong \HH^0.
\]

\subsection{The hyperboloid model}\label{sec:hyperboloid}
Let $W$ be a real vector space of dimension ${n+1}$ equipped with a quadratic form~$q$ of signature $(n,1) = (+\cdots +,-)$. Let $\langle-,-\rangle$ be the corresponding pairing.
In particular, $\langle x,x\rangle=2q(x)$.
Put
\[
\Gr(n,W)_{>0}:=\{V\subset W\mid \dim V=n \text{ and }  q\vert_V\text{ is positive definite}\}.
\]
By sending $V$ to $V^{\perp_q}$ and using the fact that $q$ has signature $(n,1)$ we see that\footnote{We use the convention that $\PP(W)=
(W-\{0\})/\RR^\ast$ (not the dual).}
\[
\Gr(n,W)_{>0}\cong \PP(W)_{<0}:= \{\bar{x}\in \PP(W)\mid q(x)<0\}.
\]
Consider
\[
W_{-1}:=\{x\in W\mid q(x)=-1\}.
\]
If $p_0\in W_{-1}$ then $\langle p_0, x\rangle \neq 0$ for any  $x \in W_{-1}$ (since 
$\langle p_0, x\rangle=0$ implies $x\in p_0^\perp$ and hence $q(x)>0$, because of the signature of $q$).
Thus the sign of the function $\langle p_0 ,-\rangle$ divides $W_{-1}$
into two connected components which are exchanged by $x\mapsto -x$ (it is easy to see that there are precisely two components).  Let $W_{-1}^+$ be the component which contains $p_0$ (i.e. $\langle p_0, -\rangle \vert_{W_{-1}^+}<0$).
Then
\begin{equation}
\label{eq:iso}
W_{-1}^+\r \PP(W)_{<0}:x\mapsto \bar{x}
\end{equation}
is a diffeomorphism. Moreover one can check that the Lorentzian metric 
$ds^2=q(dx)$ on $W$ restricts to a Riemannian metric 
on $W_{-1}$.
The pair $(W_{-1}^+, q(dx)\vert_{W_{-1}^+})$ is the so-called \emph{hyperboloid} model for hyperbolic space.
 The hyperboloid metric on $W_{-1}^+$ may be transferred to $\PP(W)_{<0}$ via~\eqref{eq:iso}. It is obvious from the construction that $O(q,\RR)$ acts by isometries on $W_{-1}^+$ and $\PP(W)_{<0}$.
 \subsection{Application to K3 surfaces}\label{sec:BK}
Now let $Y$ be a K3 surface with Picard rank $1$.  We are going to use the notations and definitions from \S\ref{sec:K3prelim}.  Recall from
\S\ref{sec:K3prelim} that the quadratic form induced by the Mukai pairing on the extended N{\'e}ron-Severi group $\Nscr(Y)$ is $q(r,l,s) = \delta l^2 -rs$,  where $d = 2\delta$ is the degree of $Y$.  
 For $z\in \CC$ we put $e(z): = \exp(zH)\in \Nscr(X)_\CC$.  If $z=x+iy$ then we have
\begin{equation}
\label{eq:e}
\begin{aligned}
e(x+iy)&= 1+(x+iy)H+\frac{1}{2}(x+iy)^2H^2\\
&= (1,xH,\frac{x^2-y^2}{2}H^2)+i(0,yH,xyH^2)\\
&= (1,x,\delta(x^2-y^2))+i(0,y,2\delta xy).
\end{aligned}
\end{equation}
Here we identify $\Pic(Y) \cong H\cdot \ZZ$ with $\ZZ$ as in \eqref{eq:Z3}.  Note that if $z\in \CC$ then it follows
from \eqref{eq:e} that
\begin{equation}
\label{eq:orthogonal}
\begin{gathered}
\langle \Re e(z), \Im e(z)\rangle =0,\\
q(\Re e(z))=q(\Im e(z))=\delta y^2.
\end{gathered}
\end{equation}

It follows from \eqref{eq:orthogonal} that $e(z)\in \Pscr(Y)$ for $z\in \HH$,  where $\HH$ denotes the upper half-plane and $\Pscr(Y)$ is defined in (\ref{eq:defP}). 
If we identify $\CC\cong \RR^2$ then it is clear that
\[
\Gr(2,\Nscr(Y)_\RR)_{>0}\cong\Pscr(Y)/\GL_2(\RR)\cong \Pscr^+(Y)/\GL^+_2(\RR).
\]

 Let $\Nscr(Y)_{\RR,-1}^+$ be the connected component of $\Nscr(Y)_{\RR,-1}$ 
containing the points such that $r>0$.
\begin{lemma}
\label{lem:transfer}
The composition
\begin{equation}
\label{eq:composition}
\HH\xrightarrow{e} \Pscr^+(Y)\r \Pscr^+(Y)/\GL^+_2(\RR)\cong \Gr(2,\Nscr(Y)_\RR)_{>0}\cong \Nscr(Y)_{\RR,-1}^+
\end{equation}
is an isomorphism given by
\begin{equation}
\label{eq:w}
\HH\r  \Nscr(Y)_{\RR,-1}^+:x+iy\mapsto \frac{1}{\sqrt{\delta} y}(1,x,\delta (x^2+y^2)).
\end{equation}
\end{lemma}
\begin{proof}
As explained in \S\ref{sec:hyperboloid} the composition  \eqref{eq:composition}
sends $z$ to $w$ where $w\in \RR^3$ is orthogonal to $\Re e(z), \Im e(z)$ and which satisfies $q(w)=-1$, $r>0$. Using~\eqref{eq:e} one checks
that one may take $w$ as in \eqref{eq:w}.
\end{proof}

\begin{remark} One can check that if we equip $\Nscr(Y)^+_{\RR,-1}$ with the Riemannian metric introduced in \S\ref{sec:hyperboloid} then the pullback of this metric under \eqref{eq:w}
is the standard metric on $\HH$ given by
\[
ds^2=\frac{dx^2+dy^2}{y^2}.
\]
\end{remark} 
\begin{corollary}
\label{cor:kawatani}
The composition
\[
\HH\xrightarrow{e} \Pscr^+(Y)\r \Pscr^+(Y)/\GL^+_2(\RR)\cong \Gr(2,\Nscr(Y)_\RR)_{>0}\cong \PP(\Nscr(Y)_\RR)_{<0}
\]
is given by
\begin{equation}
\label{eq:pinv}
\HH\r \PP(\Nscr(Y)_\RR)_{<0}:(x,y)\mapsto [1:x:\delta (x^2+y^2)].
\end{equation}
The inverse map is given by
\begin{equation}
\label{eq:inverse}
p: \PP(\Nscr(Y)_\RR)_{<0}\r \HH:[r:l:s]\mapsto\frac{l+i\sqrt{\delta^{-1}rs-l^2}}{r}.
\end{equation}
\end{corollary}
\begin{proof}
To verify the second formula we compute for 
\[
x=\frac{l}{r},\quad y=\frac{\sqrt{\delta^{-1}rs-l^2}}{r}
\]
\begin{align*}
\delta(x^2+y^2)&=\frac{\delta l^2}{r^2}+\frac{rs-\delta l^2}{r^2}=\frac{s}{r}
\end{align*}
so that
\[
[1:x:\delta(x^2+y^2)]=[r:l:s]\qedhere
\]
\end{proof}
\begin{corollary} \label{cor:root}
If $\rho:=[r:l:s]\in \PP(\Nscr(Y)_{\RR})_{<0}$ corresponds to a root, i.e.\ if $\delta l^2-rs=-1$,  then
\begin{equation}
\label{eq:rhoformula}
p(\rho)=\frac{l\delta+i\sqrt{\delta}}{r\delta}
\end{equation}
\end{corollary}

\subsection{Transfer of the $\PSL_2(\RR)$ action}
\label{sec:transfer}
Let $q$ be as in \S\ref{sec:BK}.
In this section we discuss how to transfer the action of $\PSL_2(\RR)$ on $\HH$ to $\PP(\Nscr(Y)_{\RR})_{<0}$.
\begin{lemma} \label{lem:transfer2}
Let $\widetilde{\phi}:\HH\r \HH$ be the M\"obius transformation corresponding to a matrix 
\[ \phi := 
\begin{pmatrix}
a&b\\
c&d
\end{pmatrix}.
\]
Define the matrix 
\begin{equation}\label{eq:phiformula} 
\Phi := 
\begin{pmatrix}
d^2 &2cd & c^2/\delta\\
bd& bc+ad& ac/\delta\\
b^2\delta &2ab\delta & a^2
\end{pmatrix}
\end{equation}
Then we have a commutative diagram
\[
\begin{tikzcd}
\PP(\Nscr(Y)_{\RR})_{<0}\ar[d,"\widetilde{\Phi}"']\ar[r,"p"] & \HH\ar[d,"\widetilde{\phi}"]\\
\PP(\Nscr(Y)_{\RR})_{<0}\ar[r,"p"'] & \HH
\end{tikzcd}
\]
where $\widetilde{\Phi}$ is given by
\[
\small{\begin{pmatrix}
r\\l\\s
\end{pmatrix}
\mapsto
\Phi 
\begin{pmatrix}
r\\l\\s
\end{pmatrix}.}
\]
In addition,  $\det \Phi = (\det \phi)^3$.
\end{lemma}
\begin{proof}
We will prove that $\Phi\circ p^{-1}=p^{-1}\circ \phi$. If $z=x+iy$ then the formula \eqref{eq:pinv} can be written as
\[
p^{-1}(z)=[1:(z+\bar{z})/2: \delta z\bar{z}]
\]
We then compute
\begin{align*}
(\Phi\circ p^{-1})(z)&=
\begin{pmatrix}
d^2 &2cd & c^2/\delta\\
bd& bc+ad& ac/\delta\\
b^2\delta &2ab\delta & a^2
\end{pmatrix}
\begin{pmatrix}
1\\(\bar{z}+z)/2\\\delta z\bar{z}
\end{pmatrix}\\
&=
\begin{pmatrix}
(cz+d)(c\bar{z}+d)\\
((az+b)(c\bar{z}+d)+(a\bar{z}+b)(cz+d))/2\\
\delta(az+b)(a\bar{z}+b)
\end{pmatrix}\\
&=(p^{-1}\circ \phi)(z)\qedhere
\end{align*}

 One checks that $\det \Phi=(\det\phi)^3$ by direct computation.
\end{proof}

Recall that the subgroup of index two in  $O(q,\RR)$ that preserves $\Nscr(Y)_{\RR,-1}^+$ is denoted by $O^+(q,\RR)$. It is the group of isometries of the hyperboloid model. The subgroup $\SO^+(q,\RR)$ is the group of orientation preserving isometries of the hyperboloid model.  Recall also that $\Aut^+ \Nscr(Y)$ is defined as the group of isometries of $\Nscr(Y)$ which preserve the orientation of the positive 2-planes (see \S \ref{sec:K3prelim} for details). 

Note that in the natural map
\begin{equation}
\label{eq:isom}
\Aut^+(\Nscr(Y)_\RR)\r O^+(q,\RR)\cong \Isom(\PP(\Nscr(Y)_\RR)_{<0})
\end{equation}
the two $+$'s have different meanings. In fact \eqref{eq:isom} kills $-1$ (the image of $[1]$) and its image is $\SO^+(q,\RR)$. More precisely \eqref{eq:isom} induces an isomorphism
\begin{equation}
\label{eq:isom2}
\Aut^+(\Nscr(Y)_\RR)/\{\pm 1\}\r\SO^+(q,\RR)
\end{equation}

The subgroup of orientation preserving isometries of $\HH$ is $\PSL_2(\RR)$.
So from Lemma \ref{lem:transfer} and Lemma \ref{lem:transfer2} we obtain 
isomorphisms of Lie groups
\begin{equation}
\label{eq:isolie}
\PSL_2(\RR)\cong \SO^+(q,\RR):\overline{\phi}\mapsto \frac{\Phi}{\det\phi}
\end{equation}
We state a few results without proof.  They can be verified 
easily using Lemma \ref{lem:transfer}.
\begin{corollary}\label{cor:translation}
When $\alpha \in \ZZ$ the isometry on $\PP(\Nscr(Y)_{\RR})_{<0}$ induced by $ - \otimes \Oscr_Y(\alpha) \in \Aut \Dscr^b(Y)$ is given by the formula.
\begin{equation}
\label{eq:talpha}
t_\alpha:\PP(\Nscr(Y)_{\RR})_{<0}\r \PP(\Nscr(Y)_{\RR})_{<0}:[r:l:s]\mapsto [r:l+\alpha r: s+2\delta \alpha l+\delta\alpha^2 r]
\end{equation}
 $t_\alpha$ corresponds to
\[
t_\alpha:\HH\r \HH:z\mapsto z+\alpha.
\]
\end{corollary}
\begin{remark} Note that the formula above makes sense for $\alpha\in \RR$. \end{remark}

Recall that the Fricke group $F(\delta) \subset \PSL_2(\ZZ)$ is obtained by adjoining the \emph{Fricke involution} $f_\delta$ to the congruence subgroup $\Gamma_0(\delta)$ (see \S \ref{sec:fuchsian}).

\begin{corollary}\label{cor:twi} Let $S \in \Dscr^b(Y)$ be a spherical object with $v(S) = (r,l,s) \in \Delta^+(Y)$ a positive root (cf. \eqref{eq:mukai}).  Recall that the action of $T_S$ on $\PP(\Nscr(Y)_{\RR})_{<0}$ is given by 
\[
s_{v(S)}: v \mapsto v + v(S) \langle v(S),  v \rangle.
\]
 Then $s_{v(S)}$ corresponds to the point reflection (Definition \ref{def:pointrefl}) at $p([r:l:s]) = \frac{l}{r} + \frac{i}{r\sqrt{\delta}} \in \HH$.  Explicitly it is given by the M{\"o}bius transformation
 \[
 \begin{pmatrix} \delta l & -s \\ \delta r & - \delta l\end{pmatrix} \in F(\delta).
\]
\end{corollary}

\section{Groups generated by spherical reflections in $\Aut^+\H^*(Y)$}\label{sec:reflections}
\subsection{Free products generated by elements of order two}
In this section we will obtain a precise description of the groups $G \subset \Aut^+\H^*(Y)$ generated by spherical reflections. (see \refeq{eq:GasKer}).  We also derive necessary and sufficient conditions for $G$ to be finitely generated,  depending on the degree of the K3 surface $Y$ in question (Corollary \ref{cor:finitegen}).  The results of this section will be used further to understand groups generated by spherical twists in $\Aut \Dscr^b(Y)$.  As a preparation,  we begin by some general facts regarding groups freely generated by elements of order $2$ and orbifold Riemann surfaces. 

\subsection{Free products generated by elements of order two}
\begin{lemma}
\label{lem:ordertwo}
Let $G$ be a group and assume that $R\subset G-\{1\}$ is a set of order two elements such that
\[
G=\ast_{r\in R}\langle r\rangle, 
\]
where $\ast$ denotes the free product. 
Then any element of order two in $G-\{1\}$ is conjugate to some element in $R$.
\end{lemma}
\begin{proof} Let $x\in G-\{1\}$ be such that $x^2=1$. Expressing $x$ as a reduced word
in $R$ we see that $x$ must be palindromic. If $x$ has odd length then it is of
the form $wrw^{-1}$. Even length is impossible since then $x=ww^{-1}=1$.
\end{proof}
\begin{lemma}
\label{lem:ordertwo2}
Let $G$ be a group and assume that $R\subset G-\{1\}$ is a set of order two elements such that
\[
G=\ast_{r\in R}\langle r\rangle
\]
Let $S\subset G-\{1\}$ be a collection of elements such that
\begin{enumerate}[{label=(\alph*), ref=\emph{\alph*}}]
\item \label{it:ordertwo2:it1} $G$ is generated by $S$;
\item \label{it:ordertwo2:it2} For each $s\in S$ we have $s^2=1$;
\item \label{it:ordertwo2:it3} $S$ is closed under conjugation.
\end{enumerate}
Then
\begin{enumerate}
\item \label{it:ordertwo2:it4} $R$ is a set of orbit representatives for the action of $G$ on $S$ by conjugation (in particular $R\subset S$).
\item  \label{it:ordertwo2:it5} $S$ is the set of order two elements in $G-\{1\}$.
\end{enumerate}
\end{lemma}
\begin{proof} Let $s\in S$. By Lemma \ref{lem:ordertwo} $s$ is conjugate to some $r\in R$
and by \eqref{it:ordertwo2:it3} $r\in S$. So every element in $S$ is conjugate to an element in $R\cap S$.
Assume first that $R\cap S\neq R$ and choose $r\in R\setminus (R\cap S)$.
Since by \eqref{it:ordertwo2:it1} $S$ generates $G$ it follows that $r$ is a product of conjugates of elements in $R\cap S$.  Let $\overline{r}$ be the image of$r$ in $G/[G,G] \cong \bigoplus_{r \in R} \ZZ/ 2 \ZZ r$. Then $\bar{r}$
is a sum of $\bar{r}'$ for $r'\in S\cap R$.  However, this is impossible since $\bar{r}'$ for $r'\in R$ is
a $\ZZ/2\ZZ$-basis for $G/[G,G]$. We conclude that $S\subset R$ and furthermore every conjugation orbit in $S$ intersects $R$. If the
intersection has cardinality $\ge 2$ then we find two elements in $R$ that are conjugate. Passing again to $G/[G,G]$ shows
that this is impossible. So $R$ is a set of orbit representatives. This finishes the proof of \eqref{it:ordertwo2:it4}.

If $g\in G-\{1\}$ is any element of order two then by Lemma \ref{lem:ordertwo} it is conjugate to an element in $R$.
Hence since $R\subset S$ and $S$ is closed under conjugation,  it is in $S$. This finishes the proof of \eqref{it:ordertwo2:it5}.
\end{proof}
\begin{lemma} 
\label{lem:concrete} Let the notations and assumptions be as in Lemma \ref{lem:ordertwo2}.
We have a presentation
\begin{equation}
\label{eq:concrete}
1\r F\r \Gscr\r G\r 0
\end{equation}
where $\Gscr$ is defined as the free group generated by all symbols $\tilde{r}$ for $r\in R$ 
and $F$ is the free group generated by all elements of the form $w\tilde{r_0}^2 w^{-1}$,  where $r_0 \in R$ and $w$ run over all square free words in the alphabet $\{ \tilde{r} \ \mid \ r \in R \}$ not divisible by $\tilde{r_0}$ on the right.
In particular,  sending $w\tilde{r}^2 w^{-1}\mapsto \bar{w}r \bar{w}^{-1}\in S$ provides a natural indexation of the given free generators
of $F$ by the elements of $S$.
\end{lemma}
\begin{proof} The fact that the given elements are free generators of $F$ follows naturally from the topological proof of the Nielsen-Schreier theorem
(see \cite[\S2.2.3]{MR1211642}).  Consider the directed coset graph $\Gamma_{\Gscr/F}$ for $\Gscr/F \cong G$, i.e.  its vertices are right cosets $Fg$ for $g \in \Gscr$ and edges are of the form $Fg \to Fg \tilde{r}$ for $r \in R$.  Then $\Gamma_{\Gscr/F}$ is a tree with doubled edges (for each edge $Fg \to Fg \tilde{r}$ there is an edge back $Fg\tilde{r}  \to Fg \tilde{r}^2 = Fg$).  There is a natural spanning tree obtained by picking all outgoing arrows starting from the vertex corresponding to the trivial coset $F$. This spanning tree provides the generators given in the statement.  We leave the details to the reader.
\end{proof}
\begin{remark} The given generators of $F$ are not canonical in any way.
For example there is an uncountable number of spanning trees one may contract, each yielding a different set of generators. 
\end{remark}
\subsection{Preliminaries on orbifold Riemann surfaces}
\label{subsec:prelimrs}
Let $V$ be a not necessarily compact orbifold Riemann surface with a discrete (possibly infinite) set of orbifold points $(x_i)_i$ with orders $(e_i)_i$. 
Let $V^\circ=V-\{(x_i)_i\}$. So $V^\circ$ is an ordinary Riemann surface with punctures given by $(x_i)_i$.
Choose a non-orbifold base point $x\in V$. For any path\footnote{Whenever we talk of a path in $V^\circ$ between two points in $(x_i)_i$ we mean a path in $V$ which does not intersect $(x_i)_i$ except for the endpoints. } $\gamma:x\r x'$ in $V^\circ$ we denote by $\tau_\gamma$ the small loop around $x'$, starting in $x$, 
in the direction of the orientation, obtained by thickening $\gamma$.
\begin{lemma} \label{lem:orbi}
Choose  non-intersecting paths $\gamma_i$ from $x$ to $x_i$ in $V^\circ$ and let $({{\tau}}_i)_i=(\tau_{\gamma_i})_i$ be the corresponding loops around the punctures in $V^\circ$.
Then we have a presentation
\begin{equation}
\label{eq:orbi}
\pi_1(V,x)=\pi_1(V^\circ, x)/(({{\tau}}_i^{e_i})_i).
\end{equation}
\end{lemma}
\begin{proof} See the discussion on the construction of $\pi_1(V,x)$ in \cite[p423ff.]{MR705527}.
\end{proof}
\begin{corollary}
\label{cor:riemann}
Let $V$, $(x_i)_i$, $(e_i)_i$, ${{\tau}}_i$, $x$ be as above. Let $V$ be an orbifold Riemann surface.
Let $W$ be the coarse moduli space of $V$ (i.e.\ $W$ is
obtained from $V$ by replacing the orbifold points by ordinary
points).  If $W$ is contractible then there is a presentation
\[
\pi_1(V,x)=\ast_i (\langle {{\tau}}_i\rangle/({{\tau}}_i^{e_i}))
\]
\end{corollary}
\begin{proof} Since $W$ is contractible, by the uniformization theorem we have $W=\HH$ or $W=\CC$. 
We have $V^\circ=W-\{(x_i)_i\}$ and it easy to see that $\pi_1(V^\circ,x)=\ast_i \langle {{\tau}}_i\rangle$. We now use  \eqref{eq:orbi}.
\end{proof}
\begin{remark} \label{rem:funddomain}
Assume $V=[\HH/G]$ is an orbifold quotient for a Fuchsian group $G$. Let $(y_i)_i\in \HH$ be orbit representatives for the fixed points of elements 
of $G$. Assume the $(y_i)_i$ are contained in a connected fundamental domain $\Fscr$ and choose $y$ in the interior of $\Fscr$ (the $(y_i)_i$ must be on the boundary). Let
$x$, $(x_i)_i$ be the images of $y$, $(y_i)_i$ in $V$.
Then a system of paths $(\gamma_i)_i$ from $x$ to $x_i$ as in Lemma \ref{lem:orbi} can be constructed as the image of a corresponding system of paths $\tilde{\gamma}_i$
from $y$ to $y_i$ in $\Fscr$. The $y_i$ are then the fixed points of the ${{\tau}}_i$ under the identification $G\cong \pi_1(V,x)$ corresponding to the choice of $y$.
\end{remark}
When $V$ is of finite type, i.e. $\{(x_i)_i\}=\{x_1,\ldots,x_m\}$ and $V=\overline{V}-\{x_{m+1},\ldots,x_{m+n}\}$ with $\overline{V}$ compact of genus $g$, then \eqref{eq:orbi} can be made
more precise in the sense that one has the well-known presentation for $\pi_1(V,x)$ given by
\[
a_1b_1a_1^{-1}b_1^{-1}\cdots a_gb_ga_g^{-1}b_g^{-1}{{\tau}}_1\cdots {{\tau}}_m {{\tau}}_{m+1}\cdots {{\tau}}_{m+n}=1,\quad {{\tau}}_i^{e_i}=1, i=1,\ldots,m
\]
where  we have augmented $\{{{\tau}}_1,\ldots{{\tau}}_n\}$ with additional loops $\{\tau_{n+1},\ldots,\tau_{m+n}\}$, going around the punctures $x_{m+1},\cdots x_{m+n}$ (see \cite[p423ff.]{MR705527}). 
The loops ${{\tau}}_1,\ldots,{{\tau}}_{m+n}$ can be constructed as in Remark \ref{rem:funddomain} by including points in the closure of the fundamental domain $\Fscr$ in $\overline{\HH} = \HH \cup \RR \cup \infty$

It follows that if $n\ge 1$ then $\pi_1(V,x)$ is a free product of cyclic groups 
\begin{multline}
\label{eq:freeproduct2}
\langle a_1 \rangle\ast \langle b_1\rangle \ast \cdots\ast \langle a_g\rangle \ast \langle b_g\rangle\ast \langle{{\tau}}_{m+1}\rangle \ast \cdots\ast \langle {{\tau}}_{m+n-1}\rangle 
\ast \langle{{\tau}}_1\rangle/({{\tau}}_1^{e_1}) \ast \cdots \ast \langle {{\tau}}_m\rangle/({{\tau}}_m^{e_m}) \\
\cong
F_{2g+n-1}\ast \ZZ/e_1\ZZ\ast\cdots \ast \ZZ/e_m\ZZ.
\end{multline}
\subsection{Groups generated by point reflections in the hyperbolic plane}\label{sec:pointrefl}
Recall from Definition \ref{def:AL} that $\Gamma_0^+(\delta) \subset \PSL_2(\RR)$, $\delta \in \NN$ is the group consisting of all Atkin-Lehner elements. 
\begin{lemma}
\label{lem:fuchs} Let $(s_i)_i$ be a set of point reflections in  $\Gamma^+_0(\delta)$.  Let $G$ be a subgroup of
 $\Gamma^+_0(\delta)$ 
which is generated by 
$(s_i)_i$. 
Let $S$ be the union of the conjugacy
  classes of the $(s_i)_i$ in G. 
Then 
\begin{enumerate}
\item \label{it:fuchs:it1} The set of fixed points of elements of  $G$ is the union of the orbits of the fixed points of the reflections $(s_i)_i$. Moreover the stabilisers of 
fixed points are isomorphic to $\ZZ/2\ZZ$.
\item \label{it:fuchs:it2} The Riemann surface $\HH/G$ is contractible.
\item \label{it:fuchs:it3} Every element of order two in $G$ is conjugate to some $s_i$.
\item \label{it:fuchs:it4} 
There exists a set of orbit representatives
  $R\subset S$ for the action of $G$ on $S$ by conjugation such that $G$ is the free product 
  \[
  G = \ast_{r \in R} \langle r \rangle \cong  \ast_{r \in R} \ZZ/2\ZZ. 
  \]
\end{enumerate}
\end{lemma}
\begin{proof}
Let $V=[\HH/G]$ be the orbifold quotient so that the ordinary quotient $W=\HH/G$ is the topological quotient (the coarse moduli space of $V$).  Choose a base point $x$ and loops $({{\tau}}_i)_i$ as in Lemma \ref{lem:orbi}.
Since $\HH$ is simply connected, after choosing a lift $y\in \HH$ of $x$, we have a canonical identification $G\cong \pi_1(V,x)$. Since every path in $W$ can be lifted to a path in $V$, after a small deformation, we have a surjection:
\[
G\cong \pi_1(V,x) \r \pi_1(W,x).
\]
Since $W$ is an ordinary Riemann surface,  its fundamental group has no torsion (\cite[Proposition IV.6.5]{FarkasKra}).
Since $G$ is generated by $2$-torsion elements
this implies that $\pi_1(W,x)$ is trivial, or in other words that $W$ is simply connected. Since $W$ is a (ramified) covering of the non-compact (see \S \ref{sec:fuchsian}) Riemann
surface $\HH/\Gamma^+_0(\delta)$, $W$ cannot be
compact. Hence $W$ is contractible by the uniformization theorem. This finishes the proof \eqref{it:fuchs:it2}.

By Lemma \ref{cor:riemann} we now have
\[
G = \pi_1(V,x)=\ast_i (\langle {{\tau}}_i\rangle/({{\tau}}_i^{e_i}))
\]
Consequently for each $i$ we have a surjection
\[
G\r \langle {{\tau}}_i\rangle/({{\tau}}^{e_i}_i)
\]
whose image lies in the $2$-torsion part of $\langle{{\tau}}_i\rangle/({{\tau}}_i^{e_i})$. It follows that $e_i=2$ for all $i$. 
We now apply Lemma \ref{lem:ordertwo2} 
with $R=\{({{{\tau}}}_i)_i\}\subset G$ and $S$ as in the statement of
Lemma \ref{lem:fuchs}. From Lemma \ref{lem:ordertwo2}\eqref{it:ordertwo2:it4}
we obtain \eqref{it:fuchs:it4}. From Lemma \ref{lem:ordertwo2}\eqref{it:ordertwo2:it5} we obtain that $S$ is precisely the set of order two elements in $G$.
We then get \eqref{it:fuchs:it3} from the definition of $S$. 

To prove \eqref{it:fuchs:it1} we observe that by definition the fixed points of the elements of $G$ are orbits of the fixed points of the $({{\tau}}_i)_i$.
By Lemma \ref{lem:ordertwo2}\eqref{it:ordertwo2:it4} we have $\{({{\tau}}_i)_i\}\subset S$. Hence the fixed points of $G$ are the fixed points of the elements of $S$ (since
$S$ is closed under conjugation) which in turn are the orbits of the fixed points of the $(s_i)_i$. This proves \eqref{it:fuchs:it1}.
\end{proof}
\begin{remark}\label{rem:topreal}
One may now give a topological realisation of the exact sequence  \eqref{eq:concrete}.
Let the notations be as in the statement and the proof of Lemma
\ref{lem:fuchs}. Let $\Sscr\subset \HH$ be the set of fixed points of
elements of $G$. By Lemma \ref{lem:fuchs}\eqref{it:fuchs:it1} we see that sending $s\in S$ to its unique fixed point provides a bijection between $S$ and $\Sscr$.
By Lemma \ref{lem:fuchs}\eqref{it:fuchs:it1} we have
a ramified $G$-covering $\HH\r \HH/G$ which restricts to an unramified
$G$-covering $\HH-\Sscr\r V^\circ$ where as before $V^\circ$ is obtained from
$V=[\HH/G]$ by removing the orbifold points. We get a short exact
sequence
\begin{equation}\label{eq:orbifoldses}
1\r \pi_1(\HH-\Sscr,y)\xrightarrow{\varphi} \pi_1(V^\circ,x)\r G\r 1
\end{equation}
If $\gamma$ is a path from $y$ to a point $y'$ in $\Sscr$ and $\bar{\gamma}$ is the image of $\gamma$ in $V^\circ$ then since $y'$ is a ramification point of order two we have
\begin{equation}
\label{eq:ram}
\varphi(\tau_\gamma)=\tau^2_{\overline{\gamma}}.
\end{equation}

By Lemma \ref{lem:fuchs}\eqref{it:fuchs:it2} we have that $\HH/G$ is contractible. Hence $\pi_1(V^\circ,x)$ is the free group generated by $(\tau_i)_i$, so its free basis is indexed by $R$ as in Lemma \ref{lem:concrete}.
On the other hand $\pi_1(\HH-\Sscr,y)$ is freely generated by a system of loops in $\HH-\Sscr$ associated to a system of non-intersecting paths from $y$ to the points in $\Sscr$, so its free basis is indexed by $S$ as in Lemma \ref{lem:concrete}. 
\end{remark}
\subsection{Application}
Assume now we are in the standard setting of \S \ref{sec:SO} where we have identified $\PP(\Nscr(Y)_\RR)_{<0}\cong \HH$.  Let $G$ be the group generated by reflections at roots in $\PP(\Nscr(Y)_\RR)_{<0}$.  As mentioned before,  it is known that
$\SO^+(q, \ZZ)\cong\Gamma^+_0(\delta)$ (see Proposition \ref{prop:SOisGamma+} for an independent proof).

Recall also that the modular curve $X_0^+(\delta) := \HH / \Gamma_0^+(\delta)$ is an orbifold Riemann surface of finite type (see \S \ref{sec:fuchsian}),  hence we have a presentation of $\Gamma^+_0(\delta)$
as in \eqref{eq:freeproduct2}.  Let $(x_i)_{i\in T}\subset (x_i)_i$ be those fixed points (when lifted to $\HH$) which correspond to roots.  In the notations of the presentation \eqref{eq:freeproduct2} for $\Gamma^+_0(\delta)$,  the cyclic subgroup $\langle \tau_i \rangle / (\tau_i^{e_i})$ is stabiliser of $x_i$, $i\in T$.  This stabiliser contains the point reflection with respect to $x_i$ which has order two,  hence $e_i$ must be even.  Moreover,  $\tau^{e_i/2}$ is the unique element of order $2$ in $\langle \tau_i \rangle / (\tau_i^{e_i})$, hence it is the point reflection at $x_i$.  To summarise,  $G$ is generated by the conjugates of $(\tau^{e_i/2}_i)_{i\in T}$.

\begin{lemma}\label{lem:freeproduct}
Let $(G_i)_i$ be a collection of groups with normal subgroups $H_i\subset G_i$.  
The kernel of
$
\ast_i G_i\r \ast_i (G_i/H_i)$
is generated by the conjugates of the $H_i$.
\end{lemma}
\begin{proof} Let $K$ be the subgroup of $\ast_i G_i$ generated by the conjugates of $H_i$. Then $K$ is a normal subgroup and one checks that $(\ast_i G_i)/K$ is
the categorical coproduct of the $(G_i/H_i)_i$, by verifying the universal property.
\end{proof}
By Lemma \ref{lem:freeproduct} and the discussion preceding it, together with \eqref{eq:freeproduct2},
we find that $G$ can be described as the kernel 
\begin{equation}\label{eq:GasKer}
0 \to G \hookrightarrow \Gamma^+_0(\delta)\twoheadrightarrow H \to 0
\end{equation} where $H$ is the free product of the following factors
\begin{enumerate}[label=(\emph{\roman*}), ref=\emph{\roman*}]
\item \label{it:f1} $\langle a_i\rangle$, $\langle b_i\rangle$ for $i=1,\ldots g$,  $\langle
\tau_i\rangle$ for $i=m+1,\ldots, m+n-1$.
\item \label{it:f2} $\langle \tau_i\rangle /(\tau_i^{e_i})$ for $i=1,\ldots,m$, $i\not\in T$.
\item \label{it:f3} $\langle \tau_i\rangle /(\tau_i^{e_i/2})$ for $i=1,\ldots,m$, $i\in T$.
\end{enumerate}
\begin{corollary}\label{cor:finitegen} The group $G$ is finitely generated if and only if all of the following
conditions hold:
\begin{enumerate}
\item the genus of $X^+_0(\delta):=\HH/\Gamma^+_0(\delta)$ is zero;
\item $X^+_0(\delta)$ has one cusp;
\item\label{it:cor:fingen} there is at most one point $x\in X^+_0(\delta)$ with one of the following properties:
\begin{enumerate}
\item $x$ is a fixed point which is not a root point;
\item $x$ is a root point of order $>2$.
\end{enumerate}
\end{enumerate}
\end{corollary}
\begin{proof}
It is a standard fact,  which is a consequence of Schreier's lemma, that a finite-index subgroup of a finitely generated group is finitely generated\footnote{Alternatively,  this can be proved by identifying a finitely generated group with the fundamental group of a CW complex with a finite 1-skeleton.  Then its covering space corresponding to a subgroup of finite index is also a CW complex with a finite 1-skeleton. },  so $G$ is finitely generated provided that $H$ is finite.  The converse is given by \cite[Theorem 1]{KarrassSolitar}, i.e.  if $G$ is finitely generated, then $H$ is finite.  On the other hand, $H$ is finite if and only if it consists of a single group of type
\eqref{it:f2} or \eqref{it:f3}. This is equivalent to the description in the statement. 
\end{proof}

\begin{lemma}\label{lm:fingendeg} The group $G$ is finitely generated if and only if 
\[
\delta \in \{1,2,3,5,7,11,17,19,23,29,31,41,47,59,71\}.
\]
 \end{lemma}
\begin{proof} We determine all $\delta$ for which all of the conditions in Corollary \ref{cor:finitegen} are satisfied.  \cite[Table 4]{ChuaLang} contains the signatures of all $\Gamma_0^+(\delta)$ with genus zero.  Picking those with one cusp,  we get a list of $44$ possible value of $\delta$.  It remains to determine those for which the condition (\ref{it:cor:fingen}) in Corollary \ref{cor:finitegen} is satisfied.  According to \cite[Table 4]{ChuaLang} the possible orders of elliptic elements of $\Gamma_0^+(\delta)$ are $2,3,4$ and $6$.  A fixed point of order $3$ cannot be a root point.  A fixed point of order $4$ or $6$ is a root point precisely when the unique involution in its cyclic stabiliser is a root reflection.  Using the elliptic element representatives listed in \cite[Remark 5.4]{ChuaLang} one explicitly checks that the unique fixed point of order $4$ is not a root point in all cases,  while the unique fixed point of order $6$ is a root point only for $\delta = 3$.  Finally,  using \cite[Table 6]{ChuaLang} one determines the number of non-root fixed points of order $2$ for each of the $44$ possible values of $\delta$.  All fixed points of order $2$ are root points for $\delta \in \{1,2,3,7,11,19,23,31,47,59,71\}$.  In addition,  for $\delta \in \{5,17,29,41\}$ there is a unique non-root fixed point and no fixed points of other orders.  For other values of $\delta$ there are at least $2$ non-root fixed points of order $2$.  This gives the list in the statement. 
\end{proof} 

The description of $G$ as a kernel of a map between free products allows us to give explicit generators for $G$ via Lemma \ref{lem:freeproduct2} below which is a more precise version of Lemma \ref{lem:freeproduct}.
\begin{lemma} \label{lem:freeproduct2}
Let $A_i\subset G_i$ be right coset representatives for $H_i$.
Let $G_i/H_i\r G_i:aH_i\mapsto a$ for $a\in A_i$ be the corresponding section and
extend this to a section $\phi: \ast_i (G_i/H_i)\r \ast_i G_i$.  
The kernel of $\ast_i G_i\r \ast_i (G_i/H_i)$
is generated by $(\phi(u) H_i \phi(u)^{-1})_i$ for $u\in \ast_i (G_i/H_i)$.
\end{lemma}
\begin{proof} One checks that the subgroup of $\ast_i G_i$ generated
by  $(\phi(u) H_i \phi(u)^{-1})_i$ is closed under by conjugation by $G_i$.
\end{proof}

\section{Groups generated by spherical twists in $\Aut \Dscr^b(Y)$} 
\subsection{Introduction}
In this section we obtain our first main result (Theorem \ref{thm:sphtwistfree}),  which gives a precise description of the groups generated by spherical twists in $\Aut \Dscr^b(Y)$ for K3 surfaces $Y$ of Picard rank $1$. 

\subsection{Introduction}
We start with some introductory comments.  Throughout $Y$ is a K3 surface of Picard rank $1$ and degree $d = 2\delta$. First observe the following.
\begin{lemma} The group $\Aut^+ \H^\ast(Y,\ZZ)$ acts properly discontinuously (cfr. \S\ref{sec:fuchsian})
on $\PP(\Nscr(Y)_\RR)_{<0}\cong \Pscr^+(Y)/\GL^+_2(\RR)$ (cfr. \S\ref{sec:SO}) and hence
also on $\Pscr^+(Y)$. In particular the quotients 
\[
[\Pscr^+(Y)/\Aut^+ \H^\ast(Y,\ZZ)]\text{ and } [\PP(\Nscr(Y)_\RR)_{<0}/\Aut^+ \H^\ast(Y,\ZZ)]
\]
are smooth orbifolds.
\end{lemma}
\begin{proof} The action of $\Aut^+ \H^\ast(Y,\ZZ)$ on $\Pscr^+(Y)/\GL^+_2(\RR)$ factors through the action of $\Aut^+ \Nscr(Y)$, so by Lemma \ref{lem:HvsN} it suffices to prove this result for $\Aut^+ \Nscr(Y) $ instead of $\Aut^+ \H^\ast(Y,\ZZ) $ and by 
Corollary \ref{cor:kawatani} combined with Remark \ref{rem:pm1} we may further replace $(\PP(\Nscr(Y)_\RR)_{<0},\Aut^+\Nscr(Y))$ by $(\HH,\Gamma^+_0(\delta))$.
We conclude by the fact that $\Gamma^+_0(\delta)$ is a Fuchsian group (see \S\ref{sec:fuchsian}).
\end{proof}
\begin{lemma}[see also \cite{FanLai}]
\label{lem:tautological}
Assume $\Gscr$ is a subgroup of $\Aut \Dscr^b(Y)$ containing $\Aut^0\Dscr^b(Y)$ and let $G$ be its image in $\Aut^+ H^\ast(Y,\ZZ)$. We have an isomorphism
\[
\Gscr=\pi_1([\Pscr_0^+(Y)/G])
\]
such that the long exact sequence of homotopy groups for the map ${\Pscr_0^+(Y)\r [\Pscr_0^+(Y)/G]}$ yields
\begin{equation}
\label{eq:tautological}
 1\r \Aut^0\Dscr^b(Y)\r  \Gscr\r G\r 1
\end{equation}
\end{lemma}
\begin{proof} 
Recall that the space $\Stab^\dagger(Y)$ is contractible and $\pi: \Stab^\dagger(Y) \to \Pscr_0^+(Y)$ is a Galois covering whose group of deck transformations can be identified with $\Aut^0\Dscr^b(Y) $ (see Theorems \ref{thm:BK3}, \ref{thm:mainBB}).  Hence we obtain a tautological isomorphism
\[
\Gscr=\pi_1([\Stab^\dagger(Y)/\Gscr])=\pi_1([\Pscr_0^+(Y)/G])
\]
and it follows from the general theory that this leads to the exact sequence \eqref{eq:tautological}.
\end{proof}

In what follows,  we will have to give some special treatment to the case of K3 surfaces of degree $2$. 
\begin{proposition}{\cite[Corollary 15.2.12]{Huybrechts}}\label{prop:d2}
 A K3 surface $Y$ of Picard rank $1$ has a non-trivial automorphism group if and only if it has degree $2$. In the latter case there is a (unique) nontrivial automorphism $\iota \in \Aut(Y)$,  which is the covering involution.  
The involution $\iota$ acts as the identity on $\Nscr(Y)$ and as $-1$ on the transcendental lattice $T(Y)=\Nscr(Y)^\perp\subset H^2(Y,\ZZ)$.  
\end{proposition} 

Now let $N$ be the subgroup of $\Aut \Dscr^b(Y)$ generated by
\[
\begin{cases}
[1] &\text{if $d\neq 2$}\\
[1],\iota^\ast &\text{if $d\neq 2$}
\end{cases}
\]

and let $\bar{N}$ be its image in $\Aut^+ \H^\ast(Y,\ZZ)$.  In Theorem \ref{thm:iotaseq} we will show that there is a short exact sequence
\begin{equation}
\label{eq:togamma1}
1\r \Aut^0\Dscr^b(Y)/[2] \r \Aut\Dscr^b(Y)/N\r F(\delta)\r 1
\end{equation}

The following was also observed in \cite{FanLai} (see (1.2) on p.4):
\begin{lemma}
\label{lem:tautological2}
Let $\Gscr$ be a subgroup of $\Aut\Dscr^b(Y)/N$ that contains $\Aut^0\Dscr^b(Y)/[2]$ and let $G$ be the image of $\Gscr$ in $F(\delta)$
under the rightmost map in \eqref{eq:togamma1}. Then
\[
\Gscr=\pi_1([\HH^0/G])
\]
in such a way that the long exact sequence for homotopy groups for the map ${\HH^0\r [\HH^0/G]}$ is precisely
\begin{equation}
\label{eq:tautological2}
 1\r \Aut^0\Dscr^b(Y)/[2]\r  \Gscr\r G\r 1
\end{equation}
\end{lemma}
\begin{proof}
 Put
\[
\Stab^n(Y):=\Stab^\dagger(Y)/\widetilde{\GL}_2^+(\RR)
\]
where $\widetilde{\GL}_2^+(\RR)$ is the universal cover of $\GL_2^+(\RR)$. Since $\widetilde{\GL}^+_2(\RR)$ and $\Stab^\dagger(Y)$ are contractible, the same is true for $\Stab^n(Y)$.

The element $[1]\in \Aut\Dscr^b(Y)$ acts on $\Stab^\dagger(Y)$solely by shifting the phases and multiplying the central charge by $-1$.  It is contained in $\widetilde{\GL}^+_2(\RR)$ so it acts trivially on $\Stab^n(Y)$. Similarly if $d=2$ then will show in Lemma \ref{lem:iotaprops} that $\iota$ 
as introduced in  Proposition \ref{prop:d2} acts trivially on $\Stab^\dagger(Y)$. 

Recall that $\Pscr^+_0(Y)/\GL_2^+(\RR) =\PP(\Nscr(Y)_\RR)_{<0}-\overline{\Delta^+(Y)}\cong \HH^0$.
We have a commutative diagram where the horizontal maps are coverings with the indicated Galois group (see \cite[Corollary 2.6]{Kawatani} for the fact the the bottom arrow is a Galois covering): 
\[
\begin{tikzcd}
\Stab^\dagger(Y)\ar[d,"{\widetilde{\GL}^+_2(\RR)}"']\arrow[rrr,"\Aut^0\Dscr^b(Y)"] &&& \Pscr^+_0(Y) \ar[d,"\GL_2^+(\RR)"]\\
\Stab^n(Y)\arrow[rrr,"{\Aut^0\Dscr^b(Y)/[2]}"'] &&& \PP(\Nscr(Y)_\RR)_{<0}-\overline{\Delta^+(Y)}
\end{tikzcd}
\]
Then we find in 
a similar way as in Lemma \ref{lem:tautological}
\[
\Gscr=\pi_1([\HH^0/G])
\]
and this isomorphism induces \eqref{eq:tautological2}.
\end{proof}
We will use the following result.
\begin{lemma} \label{lem:generation}
 Let $Y$ be a K3 surface with Picard rank 1.
Let 
\[
\Gscr=\langle T_S\mid S\in \Sph(Y)\rangle \subset \Aut \Dscr^b(Y)\]
be the subgroup generated by spherical twists.  Let $G$ be its image in $\Aut^+ \H^\ast(Y,\ZZ)$.
Let $(S_i)_{i\in I}\in \Sph(Y)$ be such
that $\{v(S_i)_i\}$ intersects all $G$-orbits in $\Delta(Y)$. Then we have:
\begin{enumerate}
\item \label{lem:generation:it1} if   $(T_{S_i})_i$ generates the image $\bar{\Gscr}$ of $\Gscr$ in $\Aut \Dscr^b(Y)/N$ then
 $(T_{S_i})_i$ generates~$\Gscr$;
\item \label{lem:generation:it2} if in addition the $T_{S_i}$ yield free generators for $\bar{\Gscr}$ then they are free generators for $\Gscr$.  In particular, $\Gscr \cong \bar{\Gscr}$.
\end{enumerate}
\end{lemma}
\begin{proof}
The claim \eqref{lem:generation:it2} is an obvious consequence of \eqref{lem:generation:it1}, so we concentrate on the latter.
The hypotheses imply that $\Gscr$ is generated by $(T_{S_i})_i$ and $N$. In other words if $S$ is a spherical object in $\Dscr^b(Y)$ then by Corollary \ref{cor:orbitreps}
and Lemma \ref{lem:iotaprops}\eqref{lem:iotaprops:it3}
$S$ can be written as
\[
S\cong T_{S_{i_1}}^{\pm 1}\cdots T_{S_{i_n}}^{\pm 1}  S_j[2n]
\]
for a suitable $S_j$ and $n\in \ZZ$. But then by Lemma \ref{lm:twistconj} we have
\[
T_S=T_{S[-2n]}= T_{S_{i_1}}^{\pm 1}\cdots T_{S_{i_n}}^{\pm 1} T_{S_j} T_{S_{i_n}}^{\mp 1}\cdots T_{S_{i_1}}^{\mp 1}
\]
In other words $T_S$ is contained in the subgroup of $\Gscr$ generated by $(T_{S_i})_i$. Since this is true for any $T_S$, this subgroup must in fact be equal to $\Gscr$.
\end{proof}

\subsection{Main result}
Now we are ready to state and prove our first main result.  We recall and fix some notations.  Let $Y$ be a K3 surface of Picard rank $1$ and degree $d = 2\delta$.  Let ${\Gscr = \langle T_S \ | S\in \Sph(Y) \rangle \subset \Aut \Dscr^b(Y)}$ be the group generated by all spherical twists and let $G\subset \Aut^+ \Nscr(Y) /\{\pm 1\}$ be the corresponding group generated by all spherical reflections.  In Remark \ref{rem:pm1} it will be shown that $\Aut^+ \Nscr(Y) /\{\pm 1\} \cong \Gamma^+_0(n)$.  The group $G$ acts on the roots $\Delta(Y)$.  Denote the set of orbits $\Delta(Y)/G$ by $I$.  Recall also from Corollary \ref{cor:kawatani} that we have an isomorphism $p: \PP( \Nscr(Y)_\RR)_{<0} \to \HH$ (see \eqref{eq:inverse} for the definition).  Finally,  note that since $G \subset \Aut^+ \Nscr(Y) /\{\pm 1\} \cong  \Gamma^+_0(n)$ is a normal subgroup,  one can use Lemma \ref{lem:closure} to construct a hyperbolically convex fundamental domain $\Fscr$ for $G$ containing $\infty$ in its closure in $\overline{\HH}$. 

\begin{theorem}\label{thm:sphtwistfree} Let $Y$ be a K3 surface of
  Picard rank $1$ and degree $d = 2\delta$.  Let
  $\Gscr$ be the group generated by all spherical
  twists and $G \subset \Aut^+ \Nscr(Y) /\{\pm 1\}$ the
  corresponding group of spherical reflections.  Let $\{\rho_i\}_{i \in I}$ with $\rho_i\in\Delta^+(Y)$ be a set of orbit representatives,  up to sign,  for $G$ acting on the roots.  Assume that the set $\{ x_i = p(\bar{\rho_i}) \}_{i \in I} \in \HH$ is contained in a convex fundamental domain $\Fscr$ for $G$, constructed in Lemma \ref{lem:closure}.   Then $\Gscr$ is freely generated by
\[
\langle T_{E_i} \ | \ v(E_i) = \rho_i\in\Delta^+(Y); i\in I \rangle, 
\] where $E_i$ is the vector bundle with $v(E_i)= \rho_i$. \end{theorem} 
\begin{proof} 
To simplify the discussion we will identify $\PP(\Nscr(Y)_\RR)_{<0}$ and $\HH$. E.g.\ when we speak about ``roots'' in $\HH$ we mean points of the form 
$p(\bar{\rho})$
where $\rho\in \Delta^+(Y)$. As usual $\HH^0$ is the complement of $\Sscr:=\{p(\bar{\rho})\mid \rho\in \Delta^+(Y)\}$.  
Recall that by Lemma \ref{lem:fuchs}\eqref{it:fuchs:it1}, the points of $\Sscr$ are order two ramification points for $\HH\r\HH/G$ and furthermore 
$G$ acts freely on $\HH^0$ so $[\HH^0/G]=\HH^0/G$.
By Lemma \ref{lem:generation} it is sufficient to prove the theorem with $\Gscr$ replaced by $\bar{\Gscr}$, the image of $\Gscr$ in $\Aut \Dscr^b(Y)/N$. 
By Theorem \ref{thm:kawatani} $\bar{\Gscr}$ contains $\Aut^0 \Dscr^b(Y)/[2]$.
By Lemma \ref{lem:tautological2} the exact sequence
\begin{equation}
\label{eq:tautological3}
 1\r \Aut^0\Dscr^b(Y)/[2]\r  \bar{\Gscr}\r G\r 1
\end{equation}
is obtained from the long exact sequence of homotopy groups associated with  ${\HH^0\r \HH^0/G}$. I.e.  after choosing a base point $y\in \HH^0$ with image $x\in \HH^0/G$,
\eqref{eq:tautological3} may be identified with
\begin{equation}
\label{eq:tautological4}
 1\r \pi_1(\HH^0,y)\xrightarrow{\varphi}  \pi_1(\HH^0/G,x)\r G\r 1
\end{equation}
Let $\{\rho_i\}_{i \in I} \in \HH$ be orbit representatives for the roots contained in $\Fscr$. Note that all $\rho_i$ are automatically on the boundary of $\Fscr$ and from the construction of $\Fscr$ it is clear that the interior of $\Fscr$ does not intersect the vertical line segments discussed in Remark \ref{rem:kawatani}.  Choose the base point $y$ in the interior of $\Fscr$.  Then as explained in Remark \ref{rem:topreal} $\pi_1(\HH^0/G,x)$ is the free group freely generated by loops $(\tau_{\overline{\gamma}_i})_i$ around the punctures corresponding to the images 
$(\overline{\gamma}_i)_i$ of a system of non-intersecting paths $(\gamma_i)_i$ in 
$\Fscr$ connecting $y$ to $\rho_i$. 
Recall from \eqref{eq:ram} 
that we have
$
\varphi(\tau_{\gamma_i})=\tau^2_{\overline{\gamma}_i}
$.
We will now assume that the paths $\gamma_i$ have been chosen as in Remark \ref{rem:kawatani}. Then under the identification of \eqref{eq:tautological3} and \eqref{eq:tautological4}
$\tau_{\gamma_i}$ corresponds to $T_{E_{i}}^2$ considered as an element of $\Aut^0 \Dscr^b(Y)/[2]$ (see Theorem \ref{thm:kawatani}). 
It then follows  that $\tau^2_{\overline{\gamma}_i}$ corresponds to $T_{E_i}^2$ considered as an element of $\Aut \Dscr^b(Y)/N$. Since square roots are unique in free groups 
we obtain that  $\tau_{\overline{\gamma}_i}$ corresponds to $T_{E_i}$.
\end{proof} 

By Theorem \ref{thm:sphtwistfree}  the group $\Gscr$ is finitely generated if and only if the corresponding group $G$ of spherical reflections is finitely generated.  Combining this with Lemma \ref{lm:fingendeg}, we obtain: 
\begin{theorem}\label{thm:fingendeg} Let $Y$ be a K3 surface of
  Picard rank $1$ and degree $d = 2\delta$.  Let
  $\Gscr$ be the group generated by all spherical
  twists. Then $\Gscr$ is finitely generated if and only if $\delta \in \{1,2,3,5,7,11,17,19,23,29,31,41,47,59,71\}$.  \end{theorem}

\section{Examples}\label{sec:examples}

In this section we explicitly describe the subgroup $\Gscr \subset 
\Aut \Dscr^b(Y)$  generated by spherical twists for some specific examples of K3 surfaces with Picard rank $1$.
The first four examples we consider are very general anticanonical divisors $Y$ in Fano threefolds $X$ of Picard rank $1$ possessing full exceptional collections of vector bundles, namely,  $X$ is one of the four types $\PP^3, Q_3, V_5$ and $V_{22}$ (see Lemma \ref{lm:4Fanos}).  By a theorem due to Moishezon \cite[Theorem 7.5]{Moishezon} (see also \cite[Th{\'e}or{\`e}me 15.33]{Voisin})
such a K3 surface also has Picard rank $1$.  
In each of these cases we show that $\Gscr$  is the free group $F_4$ generated by $T_{E_1|_Y},  \dots, T_{E_4|_Y}$ for a fixed exceptional collection $\Fscr = (E_1, \dots, E_4)$ on the ambient Fano threefold $X$.  A similar result is established for the case of a K3 surface of degree $2$ arising as a double cover of $\PP^2$.  Finally,  we give an example of a K3 surface (of degree 8) for which the group $\Gscr$ is not finitely generated but we still list explicit generators.  

\medskip

As before,  let $G \subset \Aut^+ \H^*(Y,  \ZZ)$ is the group generated by spherical reflections, i.e.  point reflections with respect to the roots.  By Theorem \ref{thm:sphtwistfree},  to find a set of free generators of $\Gscr$ in each of the cases it is sufficient to find a set of orbit representatives of the fixed points in $\HH$ under the action of $G$ such that these representatives are contained in an appropriate fundamental domain (the ``Ford domain'' which we have constructed in Lemma \ref{lem:closure}).  

We obtain such a fundamental domain explicitly from a fundamental domain for the corresponding group $\Gamma_ 0^+(\delta)$ using Lemma \ref{lem:coset}.  In turn,  since in all of our cases $\Gamma_0^+(\delta) = F(\delta)$ is the Fricke group (because $\delta$ is prime),  constructing a fundamental domain for $\Gamma_ 0^+(\delta)$ amounts to finding a fundamental domain for $\Gamma_0(\delta)$ invariant under the Fricke involution and cutting it along the half-circle with radius $1/\sqrt{\delta}$ centred at $0$ (the Fricke involution interchanges the two half-planes bounded by this half-circle).  These fundamental domains for $\Gamma_ 0^+(\delta)$ can be also found in e.g. \cite{BayerTravesa},  \cite{Shigezumi}. 

 Below we give a brief summary of the results and then discuss each example in detail.  

\begin{theorem}\label{thm:examples} \begin{enumerate}
\item  \label{thm:examples:it1} Let $X$ be a Fano threefold with Picard rank $1$ and $h^{2,1} = 0$,  that is, ${X = \PP^3, Q_3, V_5}$ or $V_{22}$ (the constructions are reviewed below).  Let $Y \subset X$ be a very general anticanonical divisor which is a K3 surface of Picard rank $1$ and degree $4,6,10$ and $22$ respectively.  Then the group 
\[
\Gscr = \langle T_{S} \ | S \in \Sph(Y) \rangle \subset \Aut \Dscr^b(Y)
\]
generated by all spherical twists is isomorphic to the free group $F_4$ on $4$ generators.  The spherical twists along the spherical vector bundles obtained by restricting the following full exceptional collection $\Escr$ of vector bundles on $X$ are free generators of $\Gscr$: 
 \begin{enumerate}
 \item For $X = \PP^2$ 
 \[
 \Escr = (\Oscr, \Oscr(1),\Oscr(2), \Oscr(3)).
 \]
 \item For $X = Q_3$
 \[
 \Escr = (\Oscr, \Sscr^*,\Oscr(1), \Oscr(2)),
 \]
where $\Sscr$ is the spinor bundle.
 \item For $X = V_5$
 \[
 \Escr = (\Uscr_X, \Oscr,\Uscr^*_X,\Oscr(1)), 
 \]
 where $\Uscr_X$ is the restriction of the universal bundle on $\Gr(2,5)$. 
 \item For $X = V_{22}$ 
 \[
 \Escr = (\Oscr, \Uscr^*_X,E^*, \Lambda^2 \Uscr^*_X ),
\]
where $\Uscr_X$ is the restriction of the universal bundle on $\Gr(2,5)$ and $E$ is a Mukai bundle of rank $2$ (constructed in \cite{KuznetsovV22},  \cite{BKM}). 
 \end{enumerate}
\item Let $Y$ be a K3 surface of degree $2$ which is a double cover of $X=\PP^2$ branched along a sextic curve.  Then the group $\Gscr \cong F_3$ is freely generated by the spherical twists along the pullbacks of the $3$ lines bundles $\Oscr_X, \Oscr_X(1), \Oscr_X(2)$ forming a full exceptional collection on $X$.  
\item Let $Y$ be a K3 surface of Picard rank $1$ and degree $8$ (e.g. $Y \subset X$ a very general anticanonical section for $X$ be a complete intersection of $3$ quadrics in $\PP^6$). Then $\Gscr$ is freely generated by $T_{\Oscr_Y(k)}$ for all $k \in \ZZ$ (in particular, it is infinitely generated). 
\end{enumerate}
\end{theorem}
\subsection{$Y \subset \PP^3$ (degree 4)}
Let $X = \PP^3$.  Then a very general anticanonical divisor $Y \subset X$ is a smooth quartic K3 surface.   A standard full exceptional collection on $X$ consisting of line bundles is due to Beilinson \cite{Beilinson}: 
\[
\Escr = (\Oscr, \Oscr(1),\Oscr(2), \Oscr(3))
\]
The quadratic form corresponding to the Mukai pairing on $\Nscr(Y)$ is $q(r,l,s) = 2l^2 - rs$. The spherical vector bundles $\Oscr(i)|_Y$ on $Y$ have the following classes in $\Nscr(Y)$: 
\[
 v(\Oscr|_Y) = (1,0,1),  \ v(\Oscr(1)|_Y) = (1,1,3),  \ v(\Oscr(2)|_Y) = (1,2,9), \ v(\Oscr(3)|_Y) = (1,3,19)
\]
Using the formula in Corollary \ref{cor:root},  one finds that the corresponding $4$ points in $\HH$ are 
\[
\frac{i}{\sqrt{2}} , 1+\frac{i}{\sqrt{2}},2+\frac{i}{\sqrt{2}}, 3+\frac{i}{\sqrt{2}}
\]
The group $\Gamma_0^+(2) \cong \SO^+(q, \ZZ)$ has one elliptic point of order $2$,  one elliptic point of order $4$ and one cusp (\cite[Table 4]{ChuaLang}).  The curve $X_0^+(2)$ has genus 0.

The elliptic elements of $\Gamma_0^+(2)$ can be understood in the following way.   Let $s_0 \in \Gamma_0^+(2)$ be the transformation corresponding to $T_{\Oscr_Y} \in  \Aut \Dscr^b(Y)$, namely 
\[
s_0 =\begin{pmatrix} 0 & -1 \\ 2 & 0\end{pmatrix}
\]
\[
s_0: z \mapsto \frac{-1}{2z}
\] The element $s_0 = f_2$ is the Fricke involution generating an elliptic subgroup of order $2$ in $\Gamma_0^+(2)$.  Its fixed point in $\HH$ is $\frac{i}{\sqrt{2}}$.

Now let $t \in \Gamma_0^+(2)$ be the translation 
\[
t = \begin{pmatrix} 1 & 1 \\ 0 & 1\end{pmatrix}.
\]
\[
t: z \mapsto z+1
\] Observe that by Corollary \ref{cor:translation} $t$ corresponds to the autoequivalence 
\[
 - \otimes \Oscr_Y(1) \in \Aut \Dscr^b(Y). 
\]
We claim that $s_0t$ generates an elliptic subgroup of order $4$.  Indeed,  note that $s_k := t^ks_0t^{-k}$ corresponds to the spherical twist $T_{\Oscr_Y(k)}$.  Hence $(s_0t)^4 = s_0s_1s_2s_3t^{4} = 1$ by Proposition \ref{pr:compserre},  since $\omega_X \cong \Oscr_X(-4)$.  Note that up to even shifts this relation holds not only numerically,  but on the level of derived categories as well.  The subgroup generated by $s_0t$ is maximal since otherwise there would exist an elliptic point of higher order. 
By the discussion after Lemma \ref{lem:freeproduct} the subgroup $G \subseteq \Gamma_0^+(2)$ generated by reflections is the kernel 
\[
 1 \to G \to \Gamma_0^+(2) \cong \ZZ/2 \ZZ \ast  \ZZ/4 \ZZ \xrightarrow{s_0 \mapsto 0, (s_0t) \mapsto 1} \ZZ/4 \ZZ \to 1
\]
 In particular,  $[ \Gamma_0^+(2): G] = 4$ and it is easy to see that $G$ is the free product of $4$ copies of $\ZZ/2 \ZZ$,  generated by $s_0, ts_0t^{-1}, t^2s_0t^{-2}$ and $t^3s_0t^{-3}$.  These elements correspond to the spherical twists with respect to exactly the restrictions of the vector bundles of the exceptional collection $\Escr$.

 Figure \ref{fig:P3} depicts a fundamental domain for $G$ which contains the corresponding $4$ fixed points of these generators.  
\begin{figure}[h!]\label{fig:P3}
\caption{A fundamental domain for $G \subset \Gamma_0^+(2)$ containing ${\frac{i}{\sqrt{2}} , 1+\frac{i}{\sqrt{2}},2+\frac{i}{\sqrt{2}}}$ and $3+\frac{i}{\sqrt{2}}$. }
\centering
\includegraphics[width=0.7\textwidth]{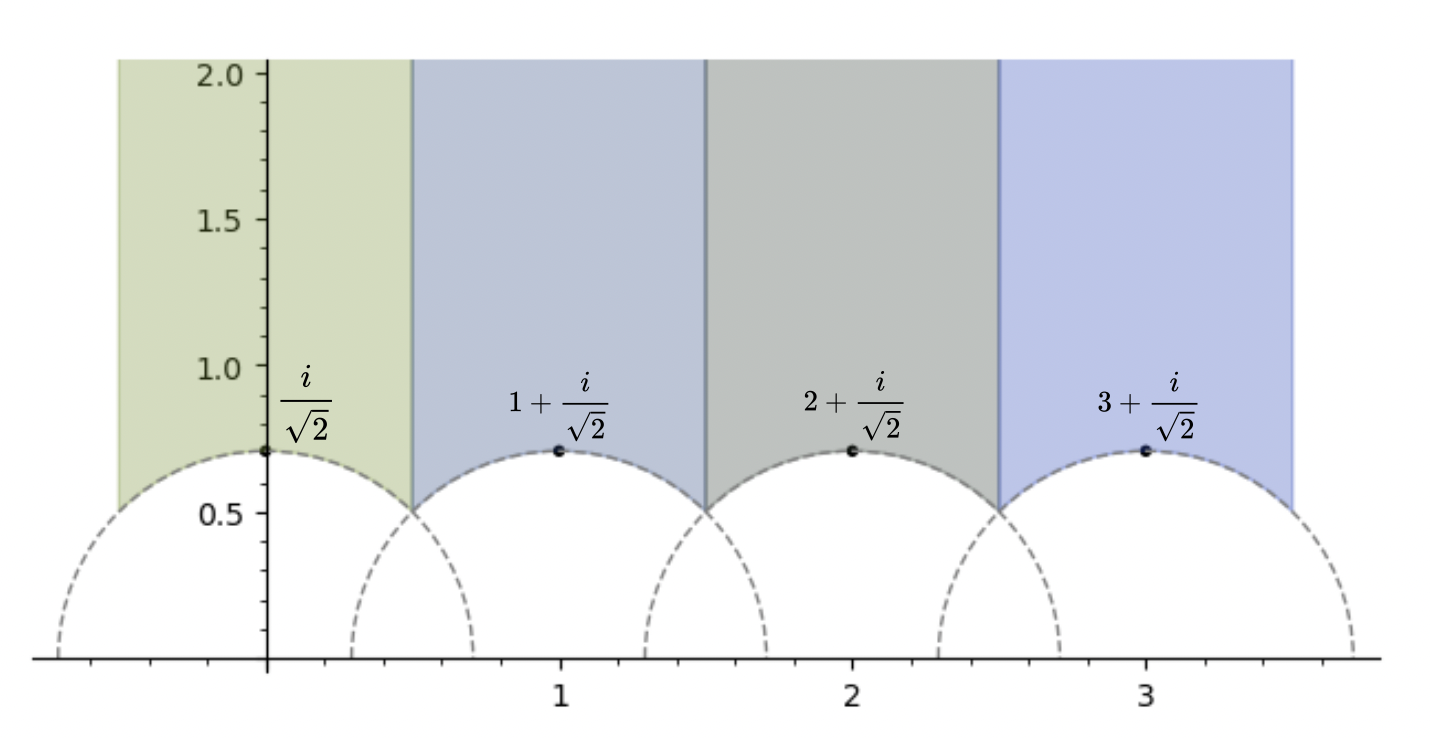}
\end{figure}
\subsection{ $Y \subset Q_3$ (degree 6)}
Let $X$ be a smooth quadric threefold.   Then a very general anticanonical divisor $Y \subset X$ is a K3 surface of degree $6$.   Kapranov \cite{Kapranov} constructed the following full exceptional collection of vector bundles on $Q_3$:
\[
\Escr = (\Oscr, \Sscr^*,\Oscr(1), \Oscr(2)),
\]
where $\Sscr$ is the spinor bundle.
The quadratic form corresponding to the Mukai pairing on $\Nscr(Y)$ is $q(r,l,s) = 3l^2 - rs$. The spherical vector bundles on $Y$ corresponding to the objects of $\Escr$ have the following classes in $\Nscr(Y)$: 
\[ v(\Oscr|_Y) = (1,0,1),  \ v(\Sscr^*|_Y) = (2,1,2),  \ v(\Oscr(1)|_Y) = (1,1,4), \ v(\Oscr(2)|_Y) = (1,2,13)
\]
Using the formula in Corollary \ref{cor:root},  one finds that the corresponding $4$ points in $\HH$ are 
\begin{align}\label{eq:Q3points}
\frac{i}{\sqrt{3}} , \frac{1}{2}+\frac{i}{2\sqrt{3}}, 1+\frac{i}{\sqrt{3}},  2+\frac{i}{\sqrt{3}}
\end{align}
The group $\Gamma_0^+(3) \cong \SO^+(q, \ZZ)$ has one elliptic point of order $2$,  one elliptic point of order $6$ and one cusp (\cite{ChuaLang},  Table 4).  The curve $X_0^+(3)$ has genus $0$.  As in the previous case,  let $s_0 \in \Gamma_0^+(3)$ be the transformation corresponding to $T_{\Oscr_Y} \in  \Aut \Dscr^b(Y)$ and $t \in  \Gamma_0^+(3)$ corresponding to $- \otimes \Oscr_Y(1) \in \Aut \Dscr^b(Y)$.  Then,  similarly to the previous example,  by Proposition \ref{pr:compserre} we have $(ts_0)^3 = s$, where $s$ denotes the transformation corresponding to $T_{\Sscr^*|_Y}$.  Hence $ts_0$ generates an elliptic subgroup of order $6$.  Same as in the previous example,  up to an even shift,  the relation $(\Oscr_Y(1) \otimes T_{\Oscr_Y}(-))^3 \cong T_{\Sscr^*|_Y}(-) \in \Aut \Dscr^b(Y)$ holds on the level of the derived category as well.  The subgroup $G$ generated by reflections has index $3$ in $\Gamma_0^+(3)$: 
\[
 1 \to G \to \Gamma_0^+(3) \cong \ZZ/2 \ZZ \ast  \ZZ/6 \ZZ \xrightarrow{s_0 \mapsto 0, (ts_0) \mapsto 1} \ZZ/3 \ZZ \to 1
\]
Hence $G \cong (\ZZ/2 \ZZ)^{\ast 4}$ is generated by $s_0, ts_0t^{-1}, t^2s_0t^{-2}, (ts_0)^3$ (we are listing the generators corresponding to the chosen roots).

Figure \ref{fig:Q3} depicts a fundamental domain of $G$ which contains the points (\ref{eq:Q3points}).  It is constructed by the same procedure as in the case of $Y \subset \PP^3$. 
\begin{figure}[h!]\label{fig:Q3}
\caption{A fundamental domain for $G \subset \Gamma_0^+(3)$ containing $\frac{i}{\sqrt{3}} , \frac{1}{2}+\frac{i}{2\sqrt{3}}, 1+\frac{i}{\sqrt{3}}$ and $2+\frac{i}{\sqrt{3}}$. }
\centering
\includegraphics[width=0.7\textwidth]{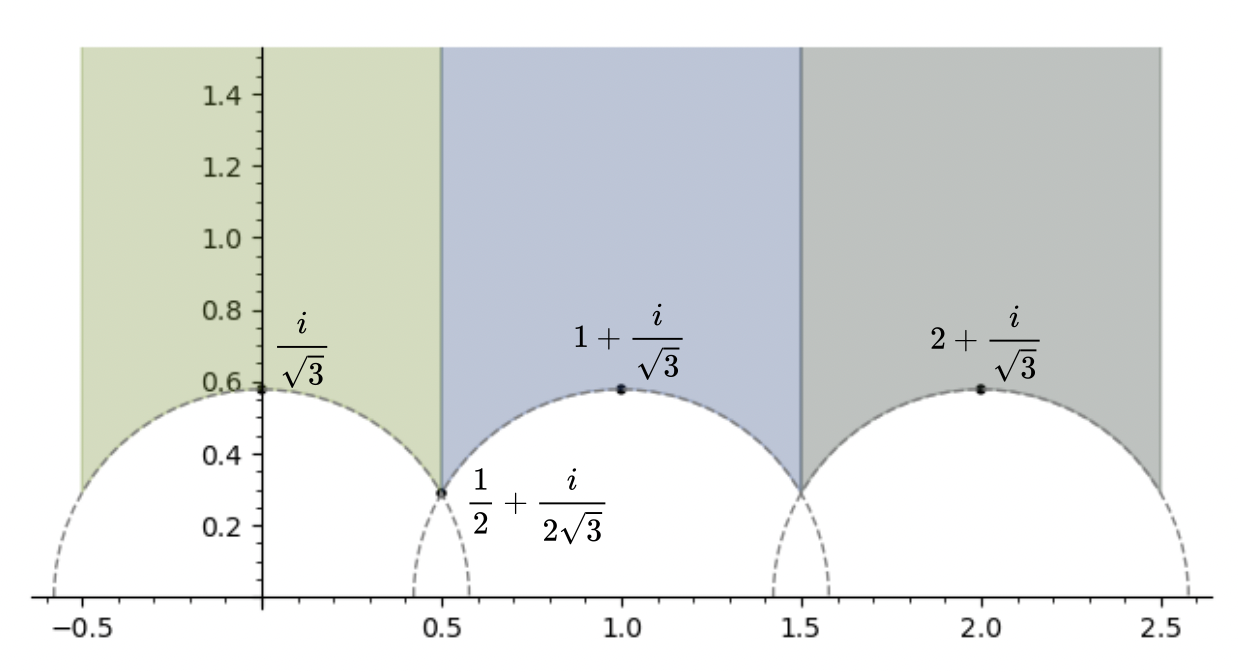}
\end{figure}
\subsection{ $Y \subset V_5$ (degree 10)}
Let $X$ be a section of the Pl{\"u}cker embedding of $\Gr(2,5)$ by a general subspace of codimension $3$.  Then $X$ is a Fano variety of type $V_5$ and a very general anticanonical divisor $Y \subset X$ is a K3 surface of degree $10$.  We consider a full exceptional collection on $X$ constructed by Orlov in \cite{OrlovV5}, namely, 
\[
\Escr = (\Oscr, \Qscr_X, \Uscr^*_X,\Oscr(1)),
\]
where $\Uscr_X$ and $\Qscr_X$ are the restrictions of the universal and the quotient bundles on $\Gr(2,5)$ respectively.  Applying the mutation $\sigma_1$ to $\Escr$ one obtains an exceptional collection of shifted vector bundles whose restrictions to $Y$ have the following classes in $\Nscr(Y)$: 
\begin{align}\label{eq:V5roots} (-2,1,-3), (1,0,1),(2,1,3),(1,1,6).
\end{align}
Since shifting a spherical object does not affect the corresponding spherical twist,  we can replace the first object $E_0'$ by $E_0'[1]$. This amounts to changing the sign of the first vector, making all ranks positive. 
And the $4$ points in $\HH$ are 
\begin{align}\label{eq:V5points}
\frac{-1}{2}+\frac{i}{2\sqrt{5}},  \frac{i}{\sqrt{5}} , \frac{1}{2}+\frac{i}{2\sqrt{5}}, 1+\frac{i}{\sqrt{5}}. 
\end{align}
The group $\Gamma_0^+(5) \cong \SO^+(q, \ZZ)$ has three elliptic points of order $2$ and one cusp (\cite[Table 4]{ChuaLang}).  The curve $X_0^+(5)$ has genus 0.  The elements $s_0, s_1 \in \Gamma_0^+(5)$ corresponding to $T_{\Oscr_Y}$ and $T_{\Qscr_Y}$ respectively generate two non-congruent elliptic subgroups of order $2$.  The remaining elliptic subgroup of order $2$ is generated by $s_0s_1t \in  \Gamma_0^+(5)$, where as usual $t$ corresponds to $\Oscr_Y(1) \otimes (-) \in \Aut \Dscr^b(Y)$.  This element is not a point reflection with respect to a root. \footnote{In fact, $s_0s_1t$ corresponds to the twist along the spherical functor $\langle \Oscr_X, \Qscr_X \rangle = \langle \Uscr_X,   \Oscr_X \rangle \hookrightarrow \Dscr^b(X) \to \Dscr^b(Y)$ (a \emph{rotation functor}, see \cite{KuznetsovPerry}).  We thank Alexander Kuznetsov for pointing this out to us. } One can verify that $(s_0s_1t )^2 = 1$ holds on the level of derived categories as well,  namely,  up to an even shift the autoequivalence $ \psi(-) := T_{\Oscr_Y} T_{\Qscr_Y} ( - \otimes \Oscr_Y(1) )$ has order two.  Indeed,  it is easy to check that up to a shift $\Uscr^*|_Y = \Oscr_Y(1) \otimes T_{\Oscr_Y}(\Qscr_Y)$.  Using this expression to write $T_{\Uscr^*}$ in terms of $T_{\Oscr_Y}$, $T_{\Qscr_Y}$ and $\Oscr_Y(1) \otimes (-) $ and substituting to the formula in Proposition \ref{pr:compserre} with $\omega_X = \Oscr_X(-2)$,  we get that $\psi$ has order $2$ up to an even shift.  By the discussion after Lemma \ref{lem:freeproduct} the subgroup $G$ generated by reflections is the kernel
\[
 1 \to G \to \Gamma_0^+(5) \cong (\ZZ/2 \ZZ)^{\ast 3} \xrightarrow{s_0 \mapsto 0, s_1 \mapsto 0, (s_0s_1t) \mapsto 1} \ZZ/2 \ZZ \to 1. 
\]
One can see that $G \cong (\ZZ/2 \ZZ)^{\ast 4}$ is generated by $s_0, s_1, ts_0t^{-1}$ and $ts_1t^{-1}$ (the generators corresponding to the roots (\ref{eq:V5roots}) are $s_0, s_0s_1s_0, t s_0 t^{-1}$ and $t s_0 s_1 s_0 t^{-1}$).

Figure \ref{fig:V5} shows a fundamental domain for $G$ containing the four points (\ref{eq:V5points}).  
\begin{figure}[h!]\label{fig:V5}
\caption{A fundamental domain for $G \subseteq \Gamma_0^+(5)$ containing $\frac{-1}{2}+\frac{i}{2\sqrt{5}},  \frac{i}{\sqrt{5}} , \frac{1}{2}+\frac{i}{2\sqrt{5}}$ and $1+\frac{i}{\sqrt{5}}$. }
\centering
\includegraphics[width=0.7\textwidth]{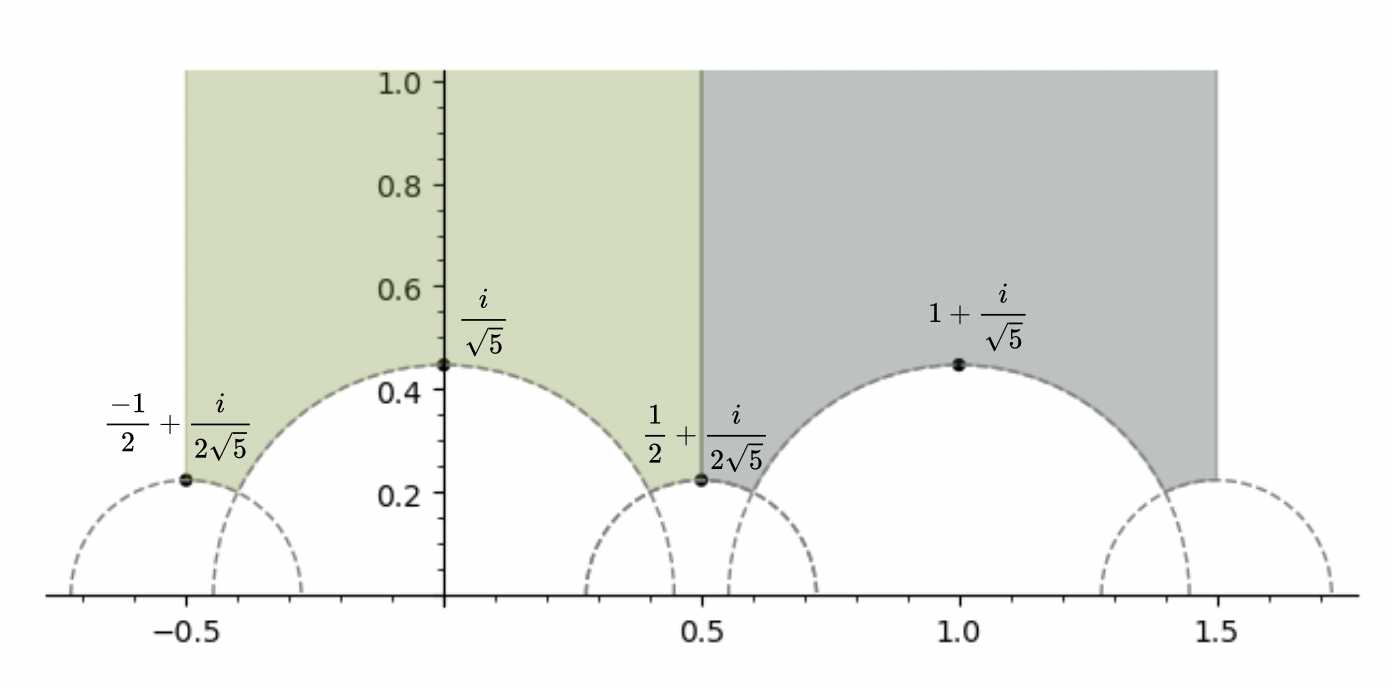}
\end{figure}
\subsection{ $Y \subset V_{22}$ (degree 22)} 
Let $X$ be the zero locus of a general section of $(\Lambda^2 \Uscr^*)^{\oplus 3}$ in $\Gr(3,7)$, where $\Uscr$ denotes the universal bundle.  $X$ a Fano variety of type $V_{22}$ and a very general anticanonical divisor $Y \subset X$ is a K3 surface of degree $22$.  A full exceptional collection of vector bundles on $X$ was constructed by Kuznetsov in \cite{KuznetsovV22} (see \cite[Theorem 2]{KuznetsovV22MPIM} for the proof of the fullness,  also see \cite[Theorem 4.1]{KuznetsovFano}):
\[
\Fscr = (\Oscr, \Uscr^*_X,E^*, \Lambda^2 \Uscr^*_X ). 
\]
Here $E$ is a vector bundle of rank $2$,  for the construction see e.g. \cite{KuznetsovV22}, \cite{BKM}.  Applying a sequence of $3$ mutations,  we obtain another exceptional collection, which automatically consists of shifted vector bundles by Proposition \ref{prop:sheavesremain}: 
\[
\Fscr' = \sigma_1 \sigma_2 \sigma_3 \Fscr = (L_{\Oscr} L_{\Uscr^*_X}L_{E^*} \Lambda^2 \Uscr^*_X,\Oscr, \Uscr^*_X,E^*),
\]
where $L_{\Oscr} L_{\Uscr^*_X}L_{E^*} \Lambda^2 \Uscr^*_X = \Uscr_X$ up to a shift.  One checks that the corresponding spherical objects on $Y$ have the following classes in $\Nscr(Y)$: 
\[
 (-3,1,-4), \, (1,0,1), (3,1,4), (2,1,6)
\]
As in the previous example,  we are free to change the sign of the first vector,  making all ranks positive. 
The corresponding $4$ points in $\HH$ are 
\[
\frac{-1}{3}+\frac{i}{3\sqrt{11}}, \frac{i}{\sqrt{11}},  \frac{1}{3}+\frac{i}{3\sqrt{11}}, \frac{1}{2}+\frac{i}{2\sqrt{11}}.
\]
The group $\Gamma_0^+(11) \cong \SO^+(q, \ZZ)$ has $4$ elliptic points of order $2$ and one cusp (\cite{ChuaLang},  Table 4).  The curve $X_0^+(11)$ has genus $0$.  Hence $G \cong \Gamma_0^+(11)$. 
\begin{figure}[h!]\label{fig:V22}
\caption{A fundamental domain for $G = \Gamma_0^+(11)$ containing $\frac{-1}{3}+\frac{i}{3\sqrt{11}}, \frac{i}{\sqrt{11}},  \frac{1}{3}+\frac{i}{3\sqrt{11}},$ and $\frac{1}{2}+\frac{i}{2\sqrt{11}}$. \\ }
\centering
\includegraphics[width=0.7\textwidth]{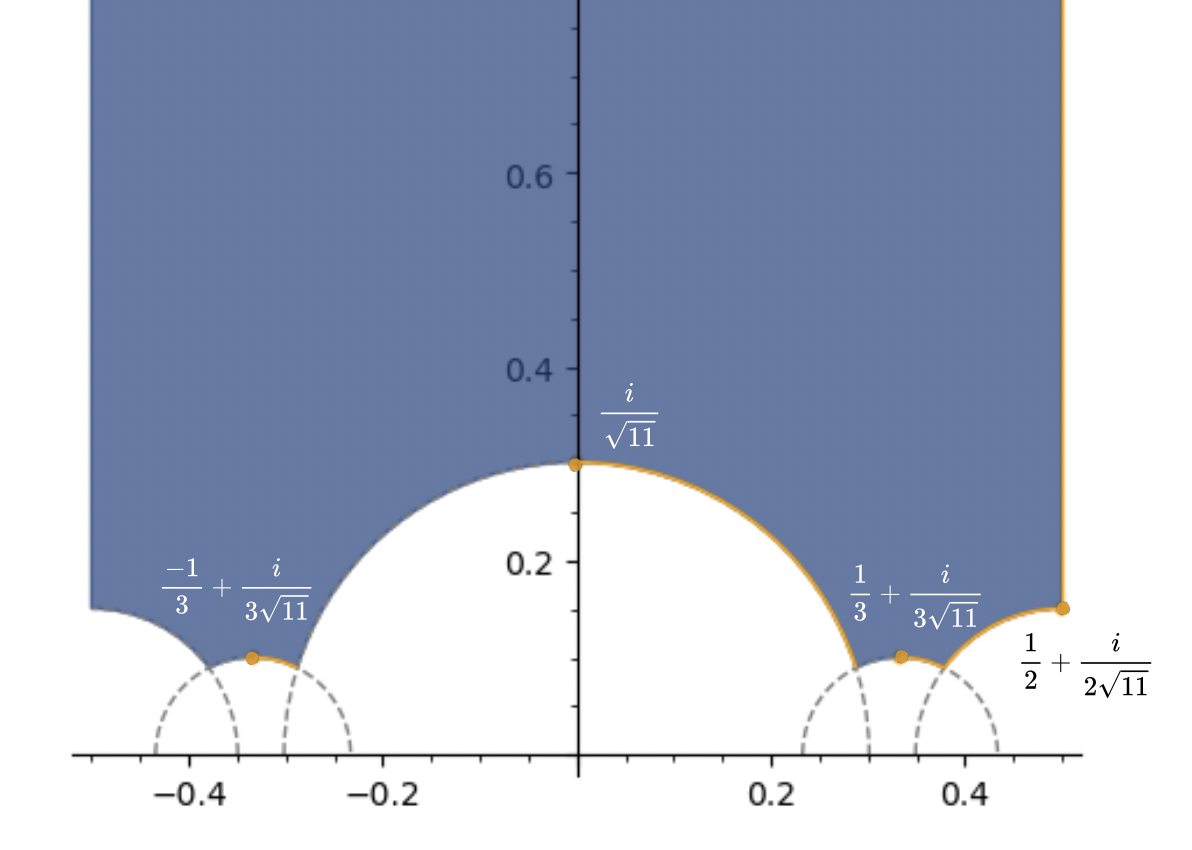}
\end{figure}

\subsection{Double cover of $\PP^2$ (degree 2)}
Let $\pi: Y \to \PP^2$ be a very general double covering branched along a sextic curve.  Then $Y$ is a K3 surface of degree $2$.  Pulling back the standard exceptional collection $(\Oscr,\Oscr(1), \Oscr(2))$ from $\PP^2$ to $Y$,  we obtain a collection of $3$ spherical vector bundles with the following classes in $\Nscr(Y)$: 
\[
(1,0,1), (1,1, 2), (1,2,5)
\]
The corresponding points in $\HH$ are $i,1+i$ and $2+i$.  The group $\Gamma_0^+(1) \cong \PSL_2(\ZZ)$ has one elliptic point of order $2$,  one elliptic point of order $3$ and one cusp.  The curve $X_0^+(1)$ has genus $0$. As before,  the element $s_0$ corresponding to $T_{\Oscr_Y}$ generates an elliptic subgroup of order $2$.  By Remark \ref{rem:doublecover} the element $s_0t$,  where $t \in  \Gamma_0^+(1)$ corresponds to $ - \otimes \Oscr_Y(1) \in \Aut \Dscr^b(Y)$,  generates an elliptic subgroup of order $3$. 

By the discussion after Lemma \ref{lem:freeproduct} the group $G$ generated by reflections is the kernel 
\[
 1 \to G \to \Gamma_0^+(1) \cong \ZZ/2 \ZZ \ast  \ZZ/3 \ZZ \xrightarrow{s_0 \mapsto 0, t^3 \mapsto 0} \ZZ/3 \ZZ \to 1
\]
One observes that $[\Gamma_0^+(1) : G] = 3$ and $G \cong \ZZ/2 \ZZ \ast  \ZZ/2 \ZZ \ast \ZZ/2 \ZZ$, generated by $s_0$, $ts_0t^{-1}$ and $t^2s_0t^{-2}$.  These elements correspond to spherical twists with respect to the pullbacks of the bundles forming the standard exceptional collection $(\Oscr_X, \Oscr_X(1), \Oscr_X(2))$ on $X = \PP^2$.  Figure \ref{fig:doublecover} shows a fundamental domain for $G$ containing the corresponding $3$ fixed points.

\begin{figure}[h!]\label{fig:doublecover} 
\caption{A fundamental domain for $G \subset \Gamma_0^+(1) \cong \PSL_2(\ZZ)$ containing $i, 1+i$ and $2+i$.}
\centering
\includegraphics[width=0.5\textwidth]{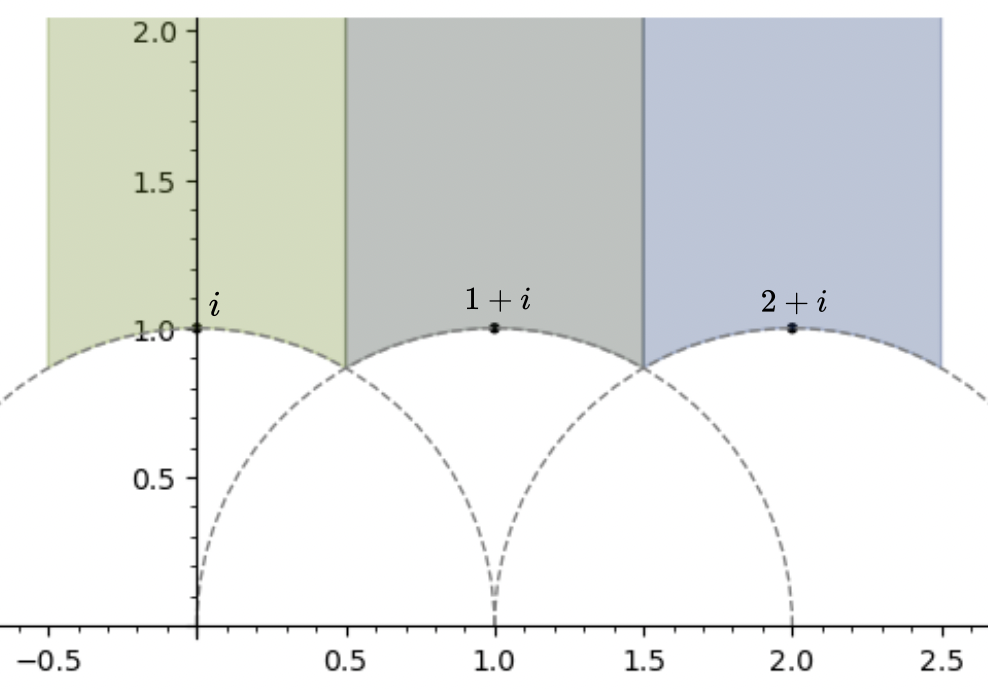}
\end{figure}

\subsection{An example of an infinitely generated group of spherical twists (degree 8)} 
Let $X$ be a general complete intersection of $3$ quadrics in $\PP^6$.  Then $X$ is a Fano variety of Picard rank $1$,  index $1$ and degree $8$.  A very general anticanonical section $Y \subset X$ is a K3 of degree $8$.  The quadratic form corresponding the the Mukai pairing on $\Nscr(Y)$ is $q(r,l,s) = 4l^2-rs$.  By Theorem \ref{thm:fingendeg} the group $\Gscr$ generated by all spherical twists is infinitely generated.  Indeed,  the group $\Gamma_0^+(4)$ has $1$ elliptic point of order $2$ and,  as in the previous examples,  the curve $X_0^+(4)$ has genus $0$.  However,  $\Gamma_0^+(4)$ has \emph{two} cusps (see \cite[Table 4]{ChuaLang}).  Thus, as explained in Corollary \ref{cor:finitegen}, the subgroup $G \subset \Gamma_0^+(4)$ generated by reflections has infinite index and is therefore infinitely generated.  As before,  the element $s_0 \in \Gamma_0^+(4)$ corresponding to $T_{\Oscr_Y}$ generates an elliptic subgroup of order $2$.  Its fixed point is $\frac{i}{2}$ and by the discussion after Lemma \ref{lem:freeproduct} the group $G = \ast_{n \in \ZZ} \ZZ/ 2\ZZ $ is the kernel 
\[
1 \to G \to \ZZ/ 2\ZZ \ast \ZZ \xrightarrow{s_0 \mapsto 1} \ZZ  \to 1
\]
$G$ is generated by the conjugates $t^k s_0 t^{-k}$,  $k \in \ZZ$ corresponding to the spherical twists $T_{\Oscr(k)_Y} \in \Aut \Dscr^b(Y)$. 

 Figure \ref{fig:deg8} shows a fundamental domain of $G$ containing the corresponding fixed points $\{k + \frac{i}{2} \}_{k \in \ZZ}$. 
 We conclude that $\Gscr \cong \langle T_{\Oscr_Y(k)} \ | \ k \in \ZZ \rangle $ is freely generated by the spherical twists with respect to all line bundles.  

\begin{figure}[h!]\label{fig:deg8}
\caption{A fundamental domain for $G \subset \Gamma_0^+(4)$ containing the fixed points $\{k + \frac{i}{2} \}_{k \in \ZZ}$ \\ }
\centering
\includegraphics[width=0.7\textwidth]{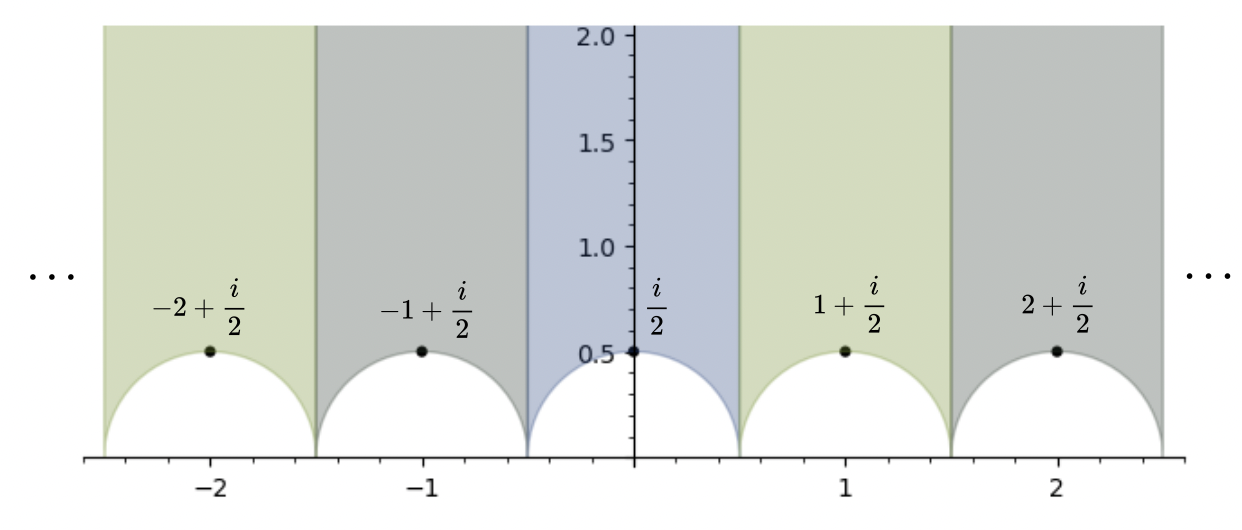}
\end{figure}

\begin{remark} In other examples considered in this section the associated Fano variety had a full exceptional collection.  This is not the case for a complete intersection $X$ of $3$ quadrics in $ \PP^6$,  since $h^{2,1}(X) = 14 \neq 0$ (see \cite[\S 12.2]{IskovskikhProkhorov}),  so the Hodge diamond of $X$ is not diagonal.  However,  the derived category of $X$ is known to have a non-trivial semiorthogonal decomposition $\Dscr^b(X) = \langle \Oscr_X,  \Dscr^b(\coh \Bscr ) \rangle$,  where $\Bscr$ is the even part of the sheaf of Clifford algebras (\cite[Theorem 6.1]{BondalOrlov}, see \cite{Kuznetsov08} for the proof).  \end{remark}
\begin{remark} The group $\Gscr$ can be finitely generated even if the K3 surface in question is an anticanonical section of a Fano threefold without a full exceptional collection.  For instance,  let $X$ be a general section of the Pl{\"u}cker embedding of $\Gr(2,6)$ by a subspace of codimension $5$.  Then $X$ is a Fano threefold of Picard rank $1$,  index $1$ and degree $14$ and a very general anticanonical section $Y \subset X$ is a K3 surface of degree $14$.  By Theorem \ref{thm:fingendeg} for such $Y$ the group $\Gscr$ is finitely
generated.  On the other hand,  $X$ has no full exceptional collection since its Hodge diamond is not diagonal ($h^{2,1} = 5$,  see \cite[\S 12.2]{IskovskikhProkhorov}). \end{remark}

\section{Exceptional collections on Fano threefolds}

\begin{theorem}[Polishchuk, {\cite[Theorem 1.2]{Polishchuk}}]\label{thm:AP} Let $X$ be a Fano threefold with Picard rank $1$ and very ample anticanonical class.  Let $E, E'$ be exceptional vector bundles on $X$ with the same restriction to a smooth anticanonical divisor $Y$.  Assume that $\Ext^1(E,E\otimes \omega_X^{-1}) = 0$. Then $E \cong E'$. 
\end{theorem}
\begin{remark} Theorem 1.2 in the actual text \cite{Polishchuk} assumes that $E$ and $E'$ have same class in $K_0(X)$ rather than the same restriction.  However, examining Polishchuk's proof one can see that the first step is to derive $E_Y = E'_Y$ from this condition and afterwards the fact that $E$ and $E'$ have the same class in $K_0(X)$ is not used.  \end{remark} 
\begin{remark}\label{rem:ext1zero} As observed in \cite{Polishchuk},  the condition $\Ext^1(E,E\otimes \omega_X^{-1}) = 0$ is satisfied in particular when $E$ is included in a full exceptional collection of vector bundles on a Fano threefold of Picard rank $1$,  
because on a variety $X$ with $\rk K_0(X) = \dim X + 1$ every full exceptional collection extends to a geometric helix (\cite[Proposition 3.3]{BondalPolishchuk}). \end{remark} 

\begin{lemma}\label{lm:restr} Let $X$ be a Fano threefold of Picard rank $1$ and $i: Y \hookrightarrow X$ a very general anticanonical divisor.  Let $E$ be an exceptional object and $F$ an exceptional vector bundle in $\Dscr^b(X)$.  Assume $T_{\LL i^* E} \cong T_{i^* F}$.  Then $E$ is a shifted vector bundle. 
\end{lemma} 

\begin{proof}
By Lemma \ref{lem:sphericalshift} we have $\LL i^* E \cong \LL i^* F[n]$.  By shifting $E$ we may assume without loss of generality assume that $n=0$.  Having done this it remains to show that $E$ has one nonzero cohomology sheaf in degree $0$, because exceptional sheaves on $X$ are vector bundles (\cite[Theorem 2.2]{Polishchuk}).  Consider the distinguished triangle 
\begin{equation}\label{eq:triangle_restiction}
E(-Y) \to E \to  i_*\LL i^* E \to 
\end{equation}

Suppose $\Hscr^l(E) \neq 0$ for $l \neq 0$.  We have
$\Hscr^l(i_*\LL i^* E) = 0$, hence the long exact sequence associated to
(\ref{eq:triangle_restiction}) implies that the cokernel of
${\Hscr^l(E(-Y)) = \Hscr^l(E)(-Y) \to \Hscr^l(E)}$ vanishes.  On the
other hand, this cokernel is isomorphic to $\Hscr^l(E)|_Y$.  Since $Y$ is an ample divisor this implies that $\Hscr^l(E)$ must have finite support.
Hence, it suffices to show that a non-zero cohomology sheaf of an exceptional object on $X$ cannot have zero-dimensional support.  Suppose
to the contrary that it does.  For every point of $p \in X$ there
exists a very general anticanonical divisor passing through $p$, so we can
pick another anticanonical divisor $i': Y' \hookrightarrow X$ such that $\Hscr^l(E)|_{Y'} \neq 0$. By Nakayama's lemma $\Hscr^l(E)(-Y')\r \Hscr^l(E)$ cannot be surjective.  Hence,  if $\Hscr^l(E)$ has zero-dimensional torsion,  then so does $\Hscr^l(i'_*\LL i'^* E)$.
But the later is impossible,  because a spherical object on any smooth projective surface has rigid cohomology sheaves (this is
well known, see e.g.  \cite[Proposition 3.5]{IshiiUehara}) and a rigid sheaf cannot have zero-dimensional torsion.\footnote{It is easy to see that a zero-dimensional sheaf $\Tscr$ on a surface cannot be rigid since $\chi(\Tscr,\Tscr) = 0$.  It then follows from Mukai's lemma (see \cite[Corollary 2.8]{Mukai87} or \cite[Lemma 2.2]{KuleshovOrlov} for the version which holds for any smooth projective surface) that a rigid sheaf cannot have zero-dimensional torsion.}
\end{proof}

\begin{lemma}\label{lm:freegrp_implies_transitive} Let $Y$ be a smooth projective variety and let $\Gscr=\langle T_S\mid S\in \Sph(Y)\rangle \subset \Aut \Dscr^b(Y)$ be the subgroup generated by all spherical twists.  Assume there are $(S_i)_{i\in I}\in \Sph(Y)$ is a (possibly infinite) set of spherical objects such that: 
\begin{enumerate}
\item \label{lm:freegrp_implies_transitive:it1} $\Gscr$ is freely generated by $(T_{S_i})_i$;
\item \label{lm:freegrp_implies_transitive:it2} for every $S\in \Sph(Y)$ there is some $H\in \Gscr$, $i\in I$ and $k\in \ZZ$ such that $S\cong H(S_i)[k]$.
\end{enumerate}
 Let
\[ 
\phi:= (T_{S^{(1)}}, \dots, T_{S^{(n)}}),\qquad \phi' := (T_{S^{(1')}}, \dots, T_{S^{(n')}})
\] be $n$-tuples of spherical twists with the same product 
\begin{equation} 
\label{eq:prodcondition}
T_{S^{(1)}} \cdots T_{S^{(n)}} = T_{S^{(1')}}\cdots T_{S^{(n')}}
\end{equation}
in $\Aut \Dscr^b(Y)$.
Then $\phi$ and $\phi'$ are in the same $B_n$-orbit under the Hurwitz action  (see \S\ref{subsec:hurwitz}) on $\Gscr^n$.
\end{lemma}

\begin{proof} Let $\Gscr_{ab} = \Gscr/[\Gscr,\Gscr]$ be the abelianisation of $\Gscr$. For $S\in \Sph(Y)$ we will denote by $\bar{S}$ the class of $T_S$ in $\Gscr_{ab}$.  
 Then by (\ref{lm:freegrp_implies_transitive:it1}) $\Gscr_{ab}$ is a free abelian group with a distinguished basis 
given by $(\bar{S}_i)_{i\in I}$.
By \eqref{lm:freegrp_implies_transitive:it2} and Lemma \ref{lm:twistconj} for every $S\in \Sph(Y)$ there exists $H \in \Gscr$ such that $T_S=HT_{S_i} H^{-1}$ and hence $\bar{S}=\bar{S}_i$ is an element of the distinguished basis.
We also deduce from \eqref{lm:freegrp_implies_transitive:it2} that if $\bar{S}=\bar{S}'$ then $T_{S}=JT_{S'}J^{-1}$ for some $J\in \Gscr$.

By \eqref{eq:prodcondition} one has 
\begin{align*}
\bar{S}^{(1)} + \cdots + \bar{S}^{(n)} = \bar{S}^{(1')} + \cdots + \bar{S}^{(n')} 
\end{align*} 

Since all $\bar{S}^{(i)}$ and $\bar{S}^{(i')}$ are in the
distinguished basis, we have $\bar{S}^{(i)} = \bar{S}^{(\tau(i)')}$
for some permutation $\tau$ of $\{1,\ldots,n\}$.  As explained above $\bar{S}^{(i)} = \bar{S}^{(\tau(i)')}$ 
implies 
$T_{S^{(i)}} = J_i T_{S^{(\tau(i)')}} J_i^{-1}$ for some $J_i\in \Gscr$ for every
$i \in \{1,\dots, n\}$.  It now follows from \eqref{lm:freegrp_implies_transitive:it1} and Theorem
\ref{thm:hurwitz_orbit} that $\phi$ and $\phi'$ are in the same Hurwitz orbit.\footnote{Note that the indexing set $I$ can be infinite,  but the whole Hurwitz orbit of $\phi$ lies in a power of a finitely generated free group,  so we are able to apply  Theorem \ref{thm:hurwitz_orbit} anyway.}
\end{proof}

Now we are ready to prove our second main result. 

\begin{theorem}\label{thm:main} Let $X$ be a threefold with a full exceptional collection of length four consisting of vector bundles.  Then the action of the semidirect product $B_4 \ltimes \ZZ^4$ by mutations and shifts on the set of full exceptional collections in $\Dscr^b(X)$ is transitive and free.
\end{theorem}

\begin{proof} 
Let $\Fscr = (F_1,F_2,F_3,F_4)$ be a
full exceptional collection of vector bundles on $X$
and let $\Escr = (E_1, E_2,E_3, E_4)$ be an arbitrary full exceptional collection in $\Dscr^b(X)$. Denote by $S_i := {F_i}|_Y,  S_i' := {E_i}|_Y$ the corresponding spherical objects in $\Dscr^b(Y)$, where $Y$ is a very general anticanonical divisor in $X$.  Proposition \ref{pr:compserre} implies that
 \begin{align*} T_{S_1}T_{S_2}T_{S_3}T_{S_4} = T_{S_1'}T_{S_2'}T_{S_3'}T_{S_4'}
 \end{align*}
 
 By Theorem \ref{thm:sphtwistfree} we are in the context of Lemma \ref{lm:freegrp_implies_transitive}.  Hence we conclude that there exists $\sigma \in B_4$ such that 
\begin{align*} 
\sigma \cdot \phi(\Escr) = \phi(\Fscr) = (T_{S_1}, T_{S_2}, T_{S_3}, T_{S_4})
\end{align*}
Applying the corresponding sequence of mutations $\sigma \in B_4$ to $\Escr$, we obtain an exceptional collection $\sigma \cdot \Escr = \Escr' = (E_1', \dots, E_4')$ such that $\phi(\Escr') = (T_{S_1}, T_{S_2}, T_{S_3}, T_{S_4})$.  From Lemma \ref{lm:restr} we deduce that each $E_i'$, $i=0,\dots, 3$ is a shift of a vector bundle.  Moreover,  $\Ext^1(F_i,F_i\otimes \omega_X^{-1}) = 0$ (and also $\Ext^1(E_i',E_i'\otimes \omega_X^{-1}) = 0$) by Remark \ref{rem:ext1zero},  hence applying Lemma \ref{lem:sphericalshift} and Theorem \ref{thm:AP} we obtain the desired transitivity on the set of full exceptional collections.

To prove the freeness it is sufficient to check that the stabiliser of $\Fscr$ is trivial up to shift under the $B_4$-action.  
Observe that the stabiliser up to shift of an exceptional collection $\Fscr$ under the mutation action is contained in the stabiliser of the corresponding tuple $\phi(\Fscr)$ under the Hurwitz action (note that $S\mapsto T_S$ kills the shift).  On the other hand,  by the transitivity proved above,  the full $B_4$-orbit of $\phi(\Fscr)$ under the Hurwitz action is contained in 
$\langle T_{S_1}, T_{S_2}, T_{S_3}, T_{S_4} \rangle^4 \subset \Gscr^4$.  Hence by Theorem \ref{thm:examples}\eqref{thm:examples:it1}
this action is in fact the Hurwitz action of $B_4$ on the power $F_4^4$ of the free group on $4$ generators.  It follows from Theorem \ref{thm:hurwitz_stab} that the stabiliser of $\phi(\Fscr)$ is indeed trivial.
  \end{proof} 

\begin{corollary}\label{cor:allsheaves} Let $X$ be a threefold with a full exceptional collection of length four consisting of vector bundles.  Then up to shifts every full exceptional collection in $\Dscr^b(X)$ is strong and consists of vector bundles. \end{corollary} 

\begin{proof} Since the property of being strong is preserved under mutations for any smooth projective variety $X$ with $\rk K_0(X) = \dim X+ 1$ (\cite{BondalPolishchuk},  Theorem 2.3),  every full exceptional collection is strong by Theorem \ref{thm:main}. By Proposition \ref{prop:sheavesremain} mutations of exceptional collections of shifted sheaves also consist of shifted sheaves.  On the other hand,  exceptional sheaves on $X$ are vector bundles by \cite[Theorem 2.2]{Polishchuk}.The conclusion again follows from the transitivity established in Theorem \ref{thm:main}.  \end{proof} 

\section{Appendix A: Description of the groups $\SO^+(q,\QQ)$ and $\SO^+(q,\ZZ)$}\label{sec:autN}

In the Appendix we give an independent proofs for the results of Dolgachev \cite{Dolgachev} and Kawatani \cite{Kawatani14} on precise descriptions of the groups $\Aut^+ \Nscr(Y)$ and $\Aut^+ \H^*(Y,\ZZ)$ for a K3 surface $Y$ of Picard rank $1$. 

Below we put
\[
\Lambda =\begin{pmatrix} \ZZ& \ZZ\\ \delta \ZZ & \ZZ\end{pmatrix} \subset M_2(\ZZ).
\]
The group $\Gamma_0(\delta)\subset \PSL_2^+(\QQ)$ introduced in \S\ref{sec:fuchsian}
is the image of $\Lambda^\ast \cap \{\det >0\}$ where for a ring
$R$ we let $R^\ast$ be its group of units.

\medskip

We now give a local characterization of the Atkin-Lehner elements also introduced in \S\ref{sec:fuchsian}.
For $p$ a prime we put
\[
\Lambda_p =\begin{pmatrix} \ZZ_p& \ZZ_p\\ \delta \ZZ_p & \ZZ_p\end{pmatrix} \subset M_2(\ZZ_p) \subset M_2(\QQ).
\]
where $\ZZ_p$ is the localization of $\ZZ$ at the prime ideal $(p)$. 

\medskip

\begin{lemma}
\label{lem:localal}
Let $x\in \GL_2^+(\QQ)$ and let $e$ be an exact divisor of $\delta$ (i.e.  $\gcd(e,\delta/e)=1$).  Then $x$ is an Atkin-Lehner element associated to
$e$ if and only if
\begin{enumerate}
\item For every prime divisor $p$ of $e$ we have
\[
x\in \begin{pmatrix} 0&-1\\ \delta &0\end{pmatrix}\Lambda^\ast_p.
\]
\item For other primes we have
\[
x\in \Lambda^\ast_p.
\]
\end{enumerate}
\end{lemma}
\begin{proof} Let us first prove the $\Leftarrow$ direction. We have
\[
\begin{pmatrix} 0&-1\\ \delta &0\end{pmatrix}\Lambda_p
=\begin{pmatrix}
\delta\ZZ_p & \ZZ_p\\
\delta\ZZ_p & \delta\ZZ_p
\end{pmatrix}
\]
and
\[
\left(\bigcap_{p\mid e} \delta \ZZ_p\right)\cap\left(\bigcap_{p\nmid e} \ZZ_p\right)
=e\ZZ
\]
so that the hypotheses imply
\[
x\in \begin{pmatrix}
e\ZZ & \ZZ\\
\delta \ZZ & e\ZZ
\end{pmatrix}.
\]
It remains to check that $\det x=e$, or equivalently 
\[
v_p(\det x)=
\begin{cases}
v_p(\delta )&\text{if $p\divides e$}\\
0&\text{otherwise}
\end{cases}
\]
This follows from the hypotheses.
The $\Rightarrow$-direction is similar.
\end{proof}
The following is the main result in this section.
\begin{proposition}[{\cite[Remark 7.2.2]{Dolgachev}}] \label{prop:SOisGamma+}
\label{prop:iso}
The isomorphism $\PSL_2(\RR)\cong \SO^+(q,\RR)$ given by $\overline{\phi}\mapsto \frac{\Phi}{\det\phi}$ (see \eqref{eq:isolie}) restricts to isomorphisms of groups
\begin{equation}
\label{eq:iso1}
\PGL_2^+(\QQ)\cong \SO^+(q,\QQ)
\end{equation}
\begin{equation}
\label{eq:iso2}
\Gamma^+_0(\delta)\cong \SO^+(q,\ZZ)
\end{equation}
\end{proposition}

We will use the symbol $\phi$ for
a $2\times 2$ matrix representing an element of $\PSL_2(\RR)$ and the symbol $\Phi$ for the corresponding $3\times 3$-matrix occurring in the formula \eqref{eq:phiformula}.
The $\subset$-inclusion in \eqref{eq:iso1} is obvious from \refeq{eq:phiformula} and the $\subset$-inclusion in \eqref{eq:iso2} is an immediate verification using the definition of $\Gamma^+_0(\delta)$ and \refeq{eq:phiformula}.

Hence our task is to prove the $\supset$-inclusions in \eqref{eq:iso1} and \eqref{eq:iso2}.

\medskip

Assume first $\Phi/\det\phi\in M_3(\QQ)$. In this case one quickly checks that 
all defined ratios among the entries of $\phi$ are in 
$\QQ$. Hence by replacing $\phi$ with a scalar multiple we may and we will assume
that $\phi\in M_2(\QQ)$.  From this we obtain \eqref{eq:iso1}. 

\medskip

It now remains to prove the $\supset$-inclusion in \eqref{eq:iso2}.
We do this locally for
each prime.
\begin{lemma}
\label{lem:localcase}
Assume $\phi \in M_2(\QQ)$ and $\Phi/\det\phi\in M_3(\ZZ_p)$ for some prime number~$p$. If $p\nmid \delta$ then, up to a power of $p$, one has $\phi\in \Lambda_p^\ast$.
If $p\divides \delta$ 
then, up to a power of $p$, one of the following alternatives holds
\begin{equation}
\label{eq:localcase}
\phi\in \Lambda_p^\ast \text{ or } \phi \in \begin{pmatrix} 0&-1\\ \delta & 0\end{pmatrix} \Lambda_p^\ast
\end{equation}
\end{lemma}
\begin{proof}
By multiplying $\phi$ with a general element of $\Lambda^\ast_p$ we may without loss of generality assume that all entries in $\phi$ are non-zero.
By further multiplying $\phi$ with a power of $p$ we may assume that $v_p(\det\phi)\in \{0,1\}$.
We consider the two cases separately.
\begin{enumerate}
\item $v_p(\det\phi)=0$.  In other words $\det\phi$ is a unit.
So we need to check when 
\[
\begin{pmatrix}
d^2 &2cd & c^2/\delta\\
bd& bc+ad& ac/\delta\\
b^2\delta &2ab\delta & a^2
\end{pmatrix}\in M_3(\ZZ_p)
\]
Looking at $\Phi_{11}$ and $\Phi_{33}$ this immediately gives us $a,d\in \ZZ_p$. 
Put  $k=v_p(\delta)$. If $k=0$ 
then looking at $\Phi_{13}$, $\Phi_{31}$ yields $\phi\in M_2(\ZZ_p)=\Lambda_p$. Since $\det\phi$ is a unit we obtain $\phi\in \Lambda_p^\ast$. Hence assume $k>0$.
By looking at $\Phi_{13}$ one obtains the condition $2v_p(c)\ge k$.

Since $ad-bc$ is a unit and $a,d\in \ZZ_p$ we
must have $v_p(b)+v_p(c)\ge 0$. Assume first $v_p(b)+v_p(c)>0$. 

Then $v_p(a)=v_p(d)=0$. Looking at $\Phi_{21}$ and $\Phi_{23}$ yields
$b\in \ZZ_p$, $c\in p^k\ZZ_p$. Hence in this case
\[
\phi\in \Lambda_p^\ast
\]
Now assume $v_p(b)+v_p(c)=0$. From $\Phi_{31}$ we get
$2v_p(b)+k\ge 0$. Then we obtain that $k$ is even and $v_p(c)=-v_p(b)=k/2$. We obtain
\[
\phi=\begin{pmatrix}
yp^{k/2}&v p^{-k/2}\\
w p^{k/2} & xp^{k/2}
\end{pmatrix}
\]
for $v,w\in \ZZ_p^\ast$ and $x,y\in \ZZ_p$ (the constraints on $\phi_{11}$  and $\phi_{22}$ come from $\Phi_{23}$ and $\Phi_{31}$). Hence we have
shown
\[
\phi\in p^{-k/2} \begin{pmatrix} 0&-1\\ \delta&0\end{pmatrix} 
\Lambda^\ast_p.
\]
\item $v_p(\det\phi)=1$. 
So we need to check when
\[
\begin{pmatrix}
d^2/p &2cd/p & c^2/(p\delta)\\
bd/p& (bc+ad)/p& ac/(p\delta)\\
b^2\delta/p &2ab\delta/p & a^2/p
\end{pmatrix}
\in M_3(\ZZ_p).
\]
Looking at $\Phi_{11}$ and $\Phi_{33}$ this immediately gives us $a,d\in p\ZZ_p$. Since ${v_p(ad-bc)=1}$, we get $v_p(b)+v_p(c)=1$. 

Looking at $\Phi_{13}$ and $\Phi_{31}$ we get 
\begin{align*}
2v_p(c)\ge k+1\\
2v_p(b)+k-1\ge 0
\end{align*}
We now get that $k=2m-1$ (in particular, $k\neq 0$ and hence $p\divides \delta$) and $v_p(c)=m$, $v_p(b)=1-m$.  So
\[
\phi=
\begin{pmatrix}
yp^m&vp^{-m+1}\\
wp^m&xp^m
\end{pmatrix}
\]
with $v,w$ units and $x,y\in \ZZ_p$ (as before,  the constraints on $\phi_{11}$ and
$\phi_{22}$ come from $\Phi_{23}$ and $\Phi_{21}$).
Hence
\[
\phi\in p^{-m+1} \begin{pmatrix}0&-1\\\delta &0 \end{pmatrix} \Lambda^\ast_p\qedhere
\]
\end{enumerate}
\end{proof}
The following corollary finishes the proof of the $\supset$-inclusion in \eqref{eq:iso2}. 
\begin{corollary}
If $\phi  \in \GL_2^+(\QQ)$ 
and $\Phi/\det \phi\in M_3(\ZZ)$ then $\phi$ represents an element of $\Gamma^+_0(\delta)$. 
\end{corollary}
\begin{proof} It follows from Lemma
\ref{lem:localcase} that after multiplying
$\phi$ by a scalar we may assume that $\phi\in \Lambda_p^\ast$ for $p\nmid \delta$ and that
\eqref{eq:localcase} holds for all primes satisfying $p\divides \delta$. Let $e$ be the 
exact divisor of $\delta$ whose prime divisors are those for which the second alternative
in \eqref{eq:localcase} holds. Then it follows from Lemma \ref{lem:localal} that
$\phi$ represents an Atkin-Lehner element associated to $e$.
\end{proof}
\begin{remark} In \cite[Lemma 1]{MR611307} the author states a similar result as Proposition \ref{prop:iso} for the form
$q'(r,l,s)=l^2-\delta rs$ if \emph{$\delta$ is square free}. 
Remarkably, for this modified form the result
does not extend to non-square free $\delta$. 
To see this, observe that in this case we are interested in those
$\phi$ for which
\[
\frac{1}{\det \phi}
\begin{pmatrix}
d^2 &2cd & c^2\delta\\
bd& bc+ad& ac\delta\\
b^2/\delta &2ab/\delta & a^2
\end{pmatrix}\in M_3(\ZZ)
\]
For $\delta=4$ one checks that one may take
\[
\phi=
\begin{pmatrix}
a&2b_1\\
c_1/2&d
\end{pmatrix}
\]
for $a,b_1,c_1,d\in\ZZ$, $ad-b_1c_1=1$. Hence $\SO^+(q',\ZZ)$ is conjugate to $\PSL_2(\ZZ)$
rather than $\Gamma^+_0(4)$.
\end{remark}
\begin{remark} \label{rem:pm1}
One checks that \eqref{eq:isom2} restricts to an isomorphism
\[
\Aut^+(\Nscr(Y))/\{\pm 1\} \xrightarrow{\cong} \SO^+(q,\ZZ)
\]
so that combining this with \eqref{eq:iso2} we obtain an isomorphism
\[
\Aut^+(\Nscr(Y))/\{\pm 1\}  \xrightarrow{\cong} \Gamma^+_0(\delta)
\]
\end{remark}
\section{Appendix B: Understanding $\Aut^+ \H^\ast(Y)$}\label{sec:autH}

To complete our understanding of $\Aut\Dscr^b(Y)$ via Theorem \ref{thm:mainBB} 
we need to understand $\Aut^+ \H^*(Y,\ZZ)$. We first discuss some facts which are specific to the case $d=2$.  Recall that in the latter case there is a (unique) nontrivial automorphism $\iota \in \Aut(Y)$,  which is the covering involution (Proposition \ref{prop:d2}).
The following result shows that the automorphism $\iota$ is quite remarkable. 
\begin{lemma}[{\cite[Theorem 8.1]{BarbacoviKikuta}}] \label{lem:iotaprops} 
Let $Y$ be a K3 surface with Picard rank one and degree two.
One has:
\begin{enumerate} 
\item $\iota$ acts trivially on $\Stab^\dagger(Y)$;
\item \label{lem:iotaprops:it2}  $\iota^\ast$ is central in $\Aut \Dscr^b(Y)$.
\item  \label{lem:iotaprops:it3} If $S\in \Dscr^b(Y)$ is a spherical object then $\iota^\ast(S)\cong S$.
\end{enumerate}
\end{lemma}
For convenience, we give a short self-contained proof. 
\begin{proof}
\begin{enumerate}
\item
Since $\iota$ acts trivially on $\Nscr(Y)$, it acts by a deck transformation
for $\Stab^\dagger(Y)/\Pscr^+_0(Y)$.
Hence it must be an element of $\Aut^0\Dscr^b(Y)$. This is impossible since $\Aut^0\Dscr^b(Y)$ is torsion free by Theorem \ref{thm:kawatani}.
\item Let $F\in\Aut \Dscr^b(Y)$. Then $F\iota^\ast F^{-1}$ acts trivially on $\Stab^\dagger(Y)$ and so $c:=F\iota^\ast F^{-1}\iota^{\ast,-1}$
also acts trivially on $\Stab^\dagger(Y)$. On the other hand $c$ acts trivially on $\H^\ast(Y,\ZZ)$
so  $c\in \Aut^0\Dscr^b(Y)$. By Theorem \ref{thm:BK3} this implies $c=\Id$.
\item By Lemma \ref{lm:twistconj} and (2) we have $T_S=\iota^\ast T_S \iota^{\ast, -1} =T_{\iota^\ast(S)}$ and hence by Lemma \ref{lem:sphericalshift} $\iota^\ast S\cong S[l]$
for some $l$. Since $\iota$ is an automorphism, the only possibility is $l=0$. \qedhere
\end{enumerate}
\end{proof}

Proposition \ref{prop:d2} should be contrasted with the following fact.
\begin{lemma} 
\label{lem:d2} Let $Y$ be a K3 surface with Picard rank 1 and $d\neq 2$. Then there is no element $\iota$ in $\Aut^+ H^\ast(Y,\ZZ)$ which acts as the identity on $\Nscr(Y)$ and as $-1$
on $T(Y)$.
\end{lemma}
\begin{proof} Assume $\iota$ exists. It follows from the definition that $\iota$ preserves the grading on $H^\ast(Y,\ZZ)$. In particular it acts  on $H^2(Y,\ZZ)$ such
that $\iota(H)=H$ and $\iota \vert_{T(Y)} =-1$. It follows that $-\iota$ is a reflection on $H^2(Y,\ZZ)$ corresponding to $H$ and so it is given by the
standard formula
\[
(-\iota)(y)=y-\frac{2\langle H,y\rangle}{\langle H,H\rangle}H
\]
In particular we have $d \divides 2\langle H,y\rangle$ for all $y\in H^2(Y,\ZZ)$. There can be no $n\neq 1\in \NN$ which divides $\langle H,y\rangle $ for all $y$ since by unimodularity this would imply $n \divides H$,  which is impossible because $H$ is primitive.
Hence $d \divides 2$ (and therefore $d=2$).
\end{proof}
Recall the following
\begin{proposition}[{\cite[Corollary 3.3.5]{Huybrechts}}] \label{prop:TY}
Let $Y$ be  a K3 surface of odd Picard rank.
The only Hodge isometries of the transcendental lattice $T(Y)$ are $\pm \id$. 
\end{proposition}
We can now make the following observation.
\begin{lemma}
\label{lem:HvsN}
Let $Y$ be a K3 surface with Picard rank $1$.
  Let $\Aut^+ \Nscr(Y)$ be the group of isometries of $\Nscr(Y)$
  preserving the orientation of the positive $2$-planes in $\Nscr(Y)_\RR$. Then the restriction map
\begin{equation}
\label{eq:restriction}
\Aut^+H^*(Y,\ZZ)\r \Aut^+\Nscr (Y)
\end{equation}
is well-defined. Moreover it is injective for $d>2$ and for $d=2$ its kernel is given by the action of the involution $\iota$ (see Lemma \ref{prop:d2})
\end{lemma}
\begin{proof}
An element of $\Aut^+H^*(Y,\ZZ)$ acts on the transcendental lattice $T(Y)$ by $\pm \Id$, according to Proposition \ref{prop:TY}.
Since $\pm \id$ does not change the orientation of the positive 2-planes in $T(Y)_\RR$, the map \eqref{eq:restriction} is indeed well-defined.

If $\phi$ is a non-trivial element in the kernel of \eqref{eq:restriction} then it is the identity on $\Nscr(Y)$, and it acts as $-1$ on $T(Y)$. Invoking 
Proposition \ref{prop:d2} and Lemma \ref{lem:d2} finishes the proof.
\end{proof}
Our next aim is to characterise the image of \eqref{eq:restriction}. This is tackled by the following result.
 \begin{proposition}[{\cite[Theorem 7.1]{Dolgachev}}] \label{prop:imHN}
 The image of the map 
\[
\Aut^+ H^*(Y, \ZZ) \to \Aut^+ \Nscr(Y) \to \SO^+(q,\ZZ) \cong \Gamma_0^+(\delta)
\]
is equal to $F(\delta)$. 
\end{proposition} 
\begin{proof}  
Recall that  $\Nscr(Y)\cong \H^0(Y,\ZZ)\oplus  \Pic(Y)\oplus \H^4(Y,\ZZ)$.
After identifying $\H^0(Y,\ZZ)\cong \ZZ\cong \H^4(Y,\ZZ)$, an element $\Phi\in \Aut^+(\Nscr(Y))$
can be identified with an element of 
\begin{equation}
\label{eq:smallmatrix}
\End(\ZZ\oplus  \Pic(Y)\oplus \ZZ)
=
\begin{pmatrix}
\ZZ & \Pic(Y)^\vee & \ZZ\\
\Pic(Y) & \End(\Pic(Y), \Pic(Y)) & \Pic(Y)\\
\ZZ & \Pic(Y)^\vee & \ZZ
\end{pmatrix}
\end{equation}
 Our task is to understand when such an element can be lifted to an element of
\begin{equation}
\label{eq:bigmatrix}
\End(\H^0(Y,\ZZ) \oplus \H^2(Y,\ZZ)\oplus \H^4(Y,\ZZ))=
\begin{pmatrix}
\ZZ & \H^2(Y,\ZZ)^\vee & \ZZ\\
\H^2(Y,\ZZ) & \End( \H^2(Y,\ZZ), \H^2(Y,\ZZ)) & \H^2(Y,\ZZ)\\
\ZZ & \H^2(Y,\ZZ)^\vee & \ZZ
\end{pmatrix}
\end{equation}
which acts as $\pm 1$ on $T(Y)$. Note that by Lemma \ref{lem:HvsN} such an extension is unique for $d>2$ and for $d=2$ there is a $\ZZ/2\ZZ$ ambiguity which is inconsequential for what follows.

The only troublesome part of \eqref{eq:bigmatrix} is the middle column, since filling its entries involves extending a map with source $\Pic(Y)$ 
to a map with source $\H^2(Y,\ZZ)$ and $\H^2(Y,\ZZ)$ is not the direct sum of $\Pic(Y)$ and $T(Y)$  (it cannot be since $\H^2(Y,\ZZ)$ is unimodular).
Now recall
\begin{enumerate}
\item The kernel of $\Aut^+\Nscr(Y)\r \SO^+(q,\ZZ)$ is $\{\pm 1\}$ by Remark \ref{rem:pm1}.
\item The identification $\SO^+(q,\ZZ)\cong \Gamma^+_0(\delta)$ is given by an explicit formula \eqref{eq:phiformula}.
\item By definition $\Gamma^+_0(\delta)$ is the collection of (images) of Atkin-Lehner elements in $\PSL_2(\RR)$ (see Definition \ref{def:AL}). 
\end{enumerate}
Hence our mission is to determine when an Atkin-Lehner element can be lifted to $\Aut^+ \H^\ast(Y,\ZZ)$.
Let 
\[
\begin{pmatrix}
ea & b \\ \delta c & ed
\end{pmatrix} \in \Gamma_0^+(\delta)
\]
be an arbitrary Atkin-Lehner element.  Then up to sign, and after identifying $\Pic(Y)=\ZZ H$,
the corresponding middle column of \eqref{eq:smallmatrix} is
\[
\begin{pmatrix}
(2 \delta cd) H^\vee \\ \left({\displaystyle\frac{\delta}{e}} bc + ead)(H^\vee \otimes H\right) \\ (2 \delta ab )H^\vee
\end{pmatrix}
\] 
where $H^\vee$ denotes the dual basis element to $H$.
Since $\langle H,  H \rangle = 2 \delta$,  
the map $(2\delta cd)H^\vee : \ZZ H\r \ZZ$
extends to $\H^2(Y,\ZZ)$ to a map which is zero on $T(Y)$ via $cd \langle H,  - \rangle$.  
Similarly,   $(2 \delta ab)H^\vee$ extends to $\H^2(Y,\ZZ)$ via $ab \langle H,  - \rangle$.  

It remains to consider the middle entries of \eqref{eq:smallmatrix} and \eqref{eq:bigmatrix}. We analyze when an element of $\beta (H^\vee\otimes H)\in \ZZ$
extends to map $f:\H^2(Y,\ZZ)\r \H^2(Y,\ZZ)$ which is $\pm \Id$ on $T(Y)$. We assume first that it extends as $+\Id$.
 Then $f$ has the form
$x \mapsto x + \alpha \langle H, x \rangle H$ where $\alpha\in \ZZ$ because of the unimodularity of $\H^2(Y,\ZZ)$.

Expressing $f(H)=\beta H$ yields
$1 + 2\delta \alpha = \beta$, so $(2\delta) | (\beta -
1)$. Applying this with $\beta = \frac{\delta}{e} bc + ead$ and recalling
that $ead - \frac{\delta}{e} bc = 1$ by the definition of an
Atkin-Lehner element, we see that this implies $e | bc$, but this
contradicts $ead - \frac{\delta}{e} bc = 1$ unless $e=1$.  Similarly,
$f$ extends to $\H^2(Y,\ZZ)$ as $-1$ on $T(Y)$ if and only if
$e = \delta$.
\end{proof} 
\begin{corollary} $\Aut^+(\H^\ast(Y,\ZZ))\r \Aut^+(\Nscr(Y))$ is surjective if and only if $\delta$ is a prime power.
\end{corollary}
\begin{proof} We have surjectivity if and only if $\Aut^+(\H^\ast(Y,\ZZ))/\{\pm 1\}\r \Aut^+(\Nscr(Y))/\{\pm 1\}$ is surjective.
I.e. if an only if $F(\delta)=\Gamma^+_0(\delta)$ which is equivalent to $\delta$ being a prime power.
\end{proof}
We now obtain a version of \eqref{eq:mainBB} in which everything is explicit
\begin{theorem}\label{thm:iotaseq} Let $N$ be the subgroup of $\Aut \Dscr^b(Y)$ generated by
\[
\begin{cases}
[1] &\text{if $d\neq 2$}\\
[1],\iota^\ast &\text{if $d\neq 2$}
\end{cases}
\]
(recall that by Lemma \ref{lem:iotaprops}\eqref{lem:iotaprops:it2} $\iota^\ast$ is central)
and let $\bar{N}$ be its image in $\Aut^+ \H^\ast(Y,\ZZ)$. I.e.
\[
\bar{N}=\begin{cases}
\{\pm 1\} &\text{if $d\neq 2$}\\
\{\pm 1, \pm \iota\} &\text{if $d\neq 2$}
\end{cases}
\]

There is a short exact sequence
\begin{equation}
\label{eq:togamma}
1\r \Aut^0\Dscr^b(Y)/[2] \r \Aut\Dscr^b(Y)/N\r F(\delta)\r 1
\end{equation}
\end{theorem}
\begin{proof}
From Theorem \ref{thm:mainBB} we obtain in fact an exact sequence
\[
1\r \Aut^0\Dscr^b(Y)/[2] \r \Aut\Dscr^b(Y)/N\r \Aut^+\H^\ast(Y,\ZZ)/\bar{N}\r 1
\]
Now we observe that there is an injection $\Aut^+ \H^\ast(Y,\ZZ) /\bar{N}\hookrightarrow \Aut^+ \Nscr(Y) /\{\pm1\}$ (Lemma \ref{lem:HvsN}) and an isomorphism
$\Aut^+ \Nscr(Y) /\{\pm 1\}\cong \Gamma^+_0(\delta)$ (Remark \ref{rem:pm1}). The image under the composition of $\Aut^+ H^\ast(Y,\ZZ)$ was identified
with $F(\delta)$ in Proposition \ref{prop:imHN}.
\end{proof}

\bibliography{bibs}
\bibliographystyle{amsplain}
\end{document}